 \documentclass[a4paper]{amsart}
\usepackage[latin1]{inputenc}
\usepackage{amsmath}
\usepackage{amsfonts}
\usepackage{amssymb}
\usepackage{amsthm}
\usepackage[usenames, dvipsnames]{color}
\usepackage{tikz}
\usepackage[all]{xy}
\usepackage{enumerate}

\definecolor{liens}{rgb}{1,0,0}
\usepackage[colorlinks=true, linkcolor=blue, 
hyperfootnotes=true,citecolor=blue,urlcolor=black]{hyperref}
\newcommand{\mfs}{}

\newtheorem{thm}{Theorem}[section]
\newtheorem{theo}{Theorem} 
\newtheorem{cor}[thm]{Corollary}
\newtheorem{lemma}[thm]{Lemma}
\newtheorem{lem}[thm]{Lemma}
\newtheorem{prop}[thm]{Proposition}

\newtheorem{defn}[thm]{Definition}
\newtheorem{defi}[thm]{Definition} 
\newtheorem{define}[thm]{Definition} 

\newenvironment{prf}[1]{\trivlist
\item[\hskip \labelsep{\bf #1.\hspace*{.3em}}]}{~\hspace{\fill}~$\square$\endtrivlist}
\newenvironment{prop2proof}{
\begin{prf}{Proof that (1)~implies (2)~in  Proposition~\ref{Prop1}}}{
\end{prf}}
\newenvironment{prop2aproof}{
\begin{prf}{Proof that (2)~implies (1)~in  Proposition~\ref{Prop1}}}{
\end{prf}}
\newenvironment{prop3proof}{
\begin{prf}{Proof that (1)~implies (2)~in  Proposition~\ref{Prop2}}}{
\end{prf}}
\newenvironment{prop3aproof}{
\begin{prf}{Proof that (2)~implies (1)~in  Proposition~\ref{Prop2}}}{
\end{prf}}
\newenvironment{thm2proof}{
\begin{prf}{Proof of Theorem~\ref{thm2}}}{
\end{prf}}

\newcommand{\ju}[1]{#1}
\newcommand{\toto}[1]{#1}

\newtheorem{rmk}[thm]{Remark}
\newtheorem{rem}[thm]{Remark} 

\newtheorem{assumption}[thm]{Assumption}
\newtheorem{condition}[thm]{Condition}

\theoremstyle{remark}
\newtheorem{exa}[thm]{Example}
\newtheorem{ex}[thm]{Example} 

\numberwithin{equation}{section}
\def\Res{\operatorname{Res}}

\def\Z{\mathbb{Z}}
\def\C{\mathbb{C}}
\def\R{\mathbb{R}}
\def\Q{\mathbb{Q}}
\def\L{\mathbb{L}}

\def\Qbar{\overlie{\Q}}
\def\P1{\mathbb{P}^{1}}
\def\Gal{\mathrm{Gal}}
\def\Aut{\mathrm{Aut}}
\def\beq{\begin{equation}}
\def\eeq{\end{equation}}
\def\Et{E_t}
\def\Etproj{\overline{E_t}}

\def\Qbar{\overline{\mathbb Q}}
\def\NX{{\mathbb N}}
\def\ZX{{\mathbb Z}}

\def\GL{{\rm GL}}

\def\pdisp{{\rm pdisp}}

\def\ores{{\rm ores}}

\def\frakO{{\mathfrak O}}

\def\calL{{\mathcal{L}}}

\def\calM{{\mathcal{M}}}

\def\K{{\mathbf K}}

\def\Gal{{\rm Gal}}

\def\walks{\mathcal{W}}
\def\walksg{\mathcal{W}_{typ}}
\def\walkse{\mathcal{W}_{ex}}
\newcommand{\walk}[1]{w_{\textrm{#1}}}

\begin{document}

\title{On the nature of the generating series of walks in the quarter plane}
\author{Thomas Dreyfus}
\address{\toto{Institut de Recherche Math\'ematique Avanc\'ee, U.M.R. 7501 Universit\'e de Strasbourg et C.N.R.S. 7, rue Ren\'e Descartes 67084 Strasbourg, FRANCE}}
\email{\toto{dreyfus@math.unistra.fr}}
\author{Charlotte Hardouin}
\address{Universit\'e Paul Sabatier - Institut de Math\'ematiques de Toulouse, 118 route de Narbonne, 31062 Toulouse.}
\email{hardouin@math.univ-toulouse.fr}
\author{Julien Roques}
\address{Universit\'e Grenoble Alpes, Institut Fourier,  CNRS UMR 5582, 100 rue des Maths, BP 74, 38402 St Martin d'H\`eres}
\email{Julien.Roques@ujf-grenoble.fr}
\author{Michael F. Singer }
\address{Department of Mathematics, North Carolina State University,
Box 8205, Raleigh, NC 27695-8205, USA}
\email{singer@math.ncsu.edu}


\thanks{This project has received funding from the European Research Council (ERC) under the European Union's Horizon 2020 research and innovation programme under the Grant Agreement No 648132. The second author would like to thank the ANR-11-LABX-0040-CIMI within
the program ANR-11-IDEX-0002-0 for its partial support. The second author's work is also supported by ANR Iso-Galois. The work of the third author has been partially supported by the LabEx PERSYVAL-Lab (ANR-11-LABX-0025-01) funded by the French program Investissement d'avenir.
The work of the fourth author was partially supported by a grant from the Simons Foundation (\#349357, Michael Singer). 
All authors received funding from NSF grant DMS-1606334 to attend the DART VII conference where significant progress concerning these results was made.  We thank Alexey Ovchinnikov and Alice Medvedev for making this possible.
We would like to thank Mireille Bousquet-M\'elou and Kilian Raschel for discussions and comments of this work. \ju{The second author would like to thank Marcello Bernardara, Thomas Dedieu and  Stephane Lamy  for many discussions and references on elliptic surfaces.}
 {\mfs In addition, we would like to thank the anonymous referees for many useful comments and suggestions concerning this article.}}

\date{\today}

\bibliographystyle{amsalpha}

\begin{abstract}
In the present paper, we introduce a new  approach, relying on {\mfs the Galois theory of difference equations}, to study the nature of the generating series of walks in the quarter plane. Using this approach, we are not only able to recover many of the recent results about these series, but also to go beyond them.  For instance, we give for the first time hypertranscendency results, {\it i.e.}, we prove that certain of these generating series do not satisfy any nontrivial nonlinear algebraic differential equation with rational coefficients.  

\end{abstract}
\maketitle
\setcounter{tocdepth}{1}
\tableofcontents
\pagestyle{myheadings}
\markboth{T.~DREYFUS, C.~HARDOUIN, J.~ROQUES, M.F.~SINGER}{ON THE NATURE OF GENERATING SERIES}

\pagebreak
\section{Introduction}

In the recent years, the nature of the generating series of  walks in the quarter plane $\Z_{\geq 0}^{2}$ has attracted the attention of many authors, see \cite{BMM,BostanKauersTheCompleteGenerating,FIM,KurkRasch,MelcMish,RaschelJEMS} and the references therein. 

To be concrete, let us consider a walk with small steps in the quarter plane $\Z_{\geq 0}^{2}$. \ju{By ``small steps'' we mean that the} set of authorized steps, denoted by $\mathcal{D}$, is a subset of $\{0,\pm 1\}^{2}\backslash\{(0,0)\}$. 
For $i,j,k\in \Z_{\geq 0}$, we let $q_{\mathcal{D},i,j,k}$ be the number of walks in $\Z_{\geq 0}^{2}$ with {steps in $\mathcal{D}$}  starting at $(0,0)$ and ending at $(i,j)$ in $k$ steps and we consider the corresponding trivariate generating series 
$$
Q_{\mathcal{D}}(x,y,t):=\displaystyle \sum_{i,j,k\geq 0}q_{\mathcal{D},i,j,k}x^{i}y^{j}t^{k}.
$$

The main questions considered in the literature are: 
\begin{itemize}
\item is $Q_{\mathcal{D}}(x,y,t)$ algebraic over $\Q(x,y,t)$ ?  
\item is $Q_{\mathcal{D}}(x,y,t)$ holonomic, {\it i.e.}, {\mfs is $Q_{\mathcal{D}}(x,y,t)$ holonomic in each of the variables $x,y,t$.  Here, holonomic in the variable $x$ means that the $\Q(x,y,t)$-vector space spanned by the partial derivatives $\frac{\partial^i}{\partial x^i}Q_{\mathcal{D}}(x,y,t)$\toto{, $i\in \mathbb{Z}_{\geq 0}$,} of $Q_{\mathcal{D}}(x,y,t)$ is finite dimensional.}
\item is $Q_{\mathcal{D}}(x,y,t)$ $x$-hyperalgebraic (resp. $y$-hyperalgebraic), {\it i.e.},  is $Q_{\mathcal{D}}(x,y,t)$, seen as a function of $x$, a solution of some nonzero  (possibly nonlinear) polynomial differential equations with coefficients in  $\Q(x,y,t)$?  In case of a negative answer, we say that $Q_{\mathcal{D}}(x,y,t)$ is $x$-hypertranscendental (resp. $y$-hypertranscendental). 
\end{itemize}

We shall now make a brief overview of some recent works on these questions. Random walks in the quarter plane were extensively considered in \cite{FIM}.  These authors attached a group to any such walk and introduced powerful analytic tools to study the generating series of such a walk. In \cite{BMM},  Bousquet-M\'elou and Mishna give a detailed study of the various walks \toto{with small steps in the quarter plane} and  make the conjecture that \ju{such} a walk has a holonomic generating series if and only if the associated group is finite. They prove that, if the group of the walk is finite, then the generating series is holonomic, except, maybe, in one case, which was solved positively by Bostan, van Hoeij and Kauers in \cite{BostanKauersTheCompleteGenerating} (see also \cite{FR10}). In \cite{MR09} Mishna and Rechnitzer showed that two of the walks with infinite groups have nonholonomic generating series. Kurkova and Raschel proved in \cite{KurkRasch} (see also \cite{BRS,RaschelJEMS}) that for all of the 51 {\it nonsingular} walks with infinite group \ju{(see Section~\ref{sec:listofwalkssmallsteps})} the corresponding generating series is not holonomic. This work is very delicate and technical, and relies on the explicit uniformization of  certain elliptic curves.  Recently, Bernardi, Bousquet-M\'elou and Raschel \cite{BBMR16} have shown that the generating series for  9 of the nonsingular walks satisfy nonlinear differential equations  despite the fact that they are not holonomic.

In the present paper, we introduce a new, more algebraic  approach, \ju{relying on Galois theory of difference equations,}  to study the nature of the generating series of walks.   Using this approach, we are not only able to recover the above mentioned remarkable results, but also to go beyond them.  For instance, the following theorem, proved in Section~\ref{sec:hypertrF1F2}, is one of the main results of this paper. 

\begin{theo}\label{theointro:hypertr}
Except for the 9 walks considered in  \cite{BBMR16}, the generating series of  all nonsingular walks with infinite group are $x$- and $y$-hypertranscendental. 
In particular, they are nonholonomic. 
\end{theo}

\ju{We shall now briefly explain our proof of this result and the reason why the Galois theory of difference equations comes into play. 
Our starting point is the well-known fact that the generating series
$Q_{\mathcal{D}}(x,y,t)$ satisfies a
functional equation of the form 
\begin{equation*} 
K_{\mathcal{D}}(x,y,t)Q_{\mathcal{D}}(x,y,t)=xy-F_{\mathcal{D}}^{1}(x,t) -F_{\mathcal{D}}^{2}(y,t)+t\epsilon Q_{\mathcal{D}}(0,0,t)
\end{equation*}
where 
$$
F_{\mathcal{D}}^{1}(x,t):= K_{\mathcal{D}}(x,0,t)Q_{\mathcal{D}}(x,0,t), \ \ F_{\mathcal{D}}^{2}(y,t):= K_{\mathcal{D}}(0,y,t)Q_{\mathcal{D}}(0,y,t),$$
for some 
$
K_{\mathcal{D}}(x,y,t) \in \C[x,y,t] 
$ 
and some $\epsilon \in \{0,1\}$. }

\ju{Fix $t \in \C$. The equation $K_{\mathcal{D}}(x,y,t)=0$ defines a curve $\Et \subset \P1(\C) \times \P1(\C)$ whose Zariski-closure $\Etproj$ happens to be an elliptic curve in all the situations considered in this paper. }

\ju{Using the fact that the series
$Q_{\mathcal{D}}(x,y,t)$ converges for $|x| < 1$, $|y| < 1$ and ${|t| < 1/|\mathcal{D}|}$, one can see $ F_{\mathcal{D}}^{1}(x,t)$ and 
$F_{\mathcal{D}}^{2}(y,t)$ as analytic functions on some small pieces of $\Etproj$. Using the above functional equation, one can prove that these functions can be extended to multivalued meromorphic functions on the whole of $\Etproj$. Uniformizing $\Etproj$ with Weierstrass functions, we can lift these extended functions to meromorphic functions $r_{x}, r_{y}$ on the universal covering $\C$ of $\Etproj$. These are the functions that will be found to satisfy difference equations. Indeed, intersecting $\Etproj$ with horizontal, resp. vertical lines, we define two involutions $\iota_{1}, \iota_{2}$ of $\Etproj$. Their compositum $\tau := \iota_{2} \circ \iota_{1}$ is a translation on $\Etproj$. We lift these three mappings to $\C$, keeping the same notations $\iota_{1}$, $\iota_{2}$, $\tau$ for the lifted mappings. Then, it can be proved that $r_{x}$ and $r_{y}$ satisfy difference equations of the form  
 \begin{eqnarray*}
  {\tau}(r_x) -r_x & = & b_1, \\
 {\tau}(r_y) -r_y & = & b_2 
 \end{eqnarray*}
for some explicit elliptic functions $b_{1}$ and $b_{2}$. 
Now, the ``nature'' of $Q_{\mathcal{D}}(x,y,t)$ is tackled as followed: its hypertranscendency with respect to the derivation $d/dx$, resp. $d/dy$, is found to be equivalent to the hypertranscendency of $F_{\mathcal{D}}^{1}(x,t)$, resp. $F_{\mathcal{D}}^{2}(y,t)$, and then in turn to the hypertranscendency of $r_{x}$, resp. $r_{y}$. 
Results from the Galois theory of difference equations 
allow one to reduce the question of showing that $r_{x}$ and $r_{y}$ are  hypertranscendental to showing that a certain linear differential equation defined on the elliptic curve $\Etproj$ has no solutions in the function field of that curve \cite{Hard08,HS}. It turns out that the last question can be answered by some elementary considerations about the polar divisor of some elliptic functions.  
Note that the fact that difference equations come into play in the present context is classical 
(see \cite{FIM,KurkRasch,RaschelJEMS}). The novelty of our approach consists in the {\it algebraic} way we exploit these functional equations, in the light of the \ju{Galois theory of difference equations}.}

Our techniques also allow us to study the 9 exceptional cases and to recover some of the results of \cite{BBMR16}, namely : 

\begin{theo}\label{thm2}
In the 9 exceptional cases treated in \cite{BBMR16}, the generating series $Q_{\mathcal{D}}(x,y,t)$ is $x$- and $y$-hyperalgebraic. 
\end{theo}

It is very likely that our method can be used to study the generating series of {\it weighted} walks in the quarter-plane as well as {\it singular} walks. This is explained in more details in Section \ref{sec:weighted}. We hope to come back on this in future publications.

The  paper is organized as follows. In Section~\ref{sec:notation} we review several useful facts and ideas  that form the basis of this paper as well as previous investigations concerning the generating series of walks in the quarter plane: the functional equation, the elliptic curve associated to this equation together with certain involutions and automorphisms and the method by which one reduces the question of hypertranscendence to a similar question for a multivalued meromorphic function on the associated curve.  In 
Section~\ref{sec:hyper}, we present the criteria based on the Galois theory of difference equations which we will use to determine if a function is hypertranscendental. \ju{A brief introduction to the Galois theory of difference equations is given in Appendix~\ref{diffgalois}, as well as a proof of the above mentioned criteria.} In Sections~\ref{sec:prelim} and~\ref{sec:hypertrF1F2} we present the calculations that show that  for all but the nine exceptional cases, the generating series of nonsingular walks are hypertranscendental. In Section~\ref{sec:hyperalg} we show that for the nine exceptional cases, the generating series have specializations that are hyperalgebraic and deduce Theorem~\ref{thm2} from this. In Section~\ref{sec:nonholonomic} we show these series are not holonomic.  \ju{Appendix~\ref{introsec}} contains useful necessary and sufficient conditions for certain linear differential equations  on elliptic curves (similar to the telescopers appearing in the verification of combinatorial identities)  to have solutions in the function field of the curve. \ju{These criteria involve some ``orbit residues'' and Appendix~\ref{appendixB} provides useful results about them.}  These \ju{results}, together with the results of Section~\ref{sec:hyper}, are \ju{the mains tools} used in Sections~\ref{sec:prelim} and~\ref{sec:hypertrF1F2}  to determine hypertranscendency.

%
 
\section{Fundamental properties of the walks with small steps} \label{sec:notation}

We start by recalling some basic facts about random walks in the quarter-plane, see \cite{BMM,FIM,KauersYatchak,MelcMish} for more details. 

\subsection{The generating series}\label{sec: the walks and gener series} 

We consider a walk with small steps in the quarter plane $\Z_{\geq 0}^{2}$. The set of authorized steps $\mathcal{D}$ is a subset of $\{0,\pm 1\}^{2}\backslash\{(0,0)\}$. 
For $i,j,k\in \Z_{\geq 0}$, we let $q_{\mathcal{D},i,j,k}$ be the number of walks in $\Z_{\geq 0}^{2}$ with {steps in $\mathcal{D}$}  starting at $(0,0)$ and ending at $(i,j)$ in $k$ steps and we consider the corresponding trivariate generating series 
$$
Q_{\mathcal{D}}(x,y,t):=\displaystyle \sum_{i,j,k\geq 0}q_{\mathcal{D},i,j,k}x^{i}y^{j}t^{k}.
$$
\ju{The obvious fact that $\vert q_{\mathcal{D},i,j,k} \vert \leq |\mathcal{D}|^{k}$ ensures that} $
Q_{\mathcal{D}}(x,y,t)$ converges for all $(x,y,t)\in \C^{2}\times \R$ such that $|x|\leq 1$, $|y|\leq 1$ and $0<t< 1/|\mathcal{D}|$. 

\subsection{Kernel and functional equation} \label{sec:kernel and gen series}

The generating series $Q_{\mathcal{D}}(x,y,t)$ satisfies a functional equation that we shall now recall. The {\it Kernel} of the walk is defined by 
$$
K_{\mathcal{D}}(x,y,t):=xy (1-t S_{\mathcal{D}}(x,y))
$$ where 
$$
S_{\mathcal{D}}(x,y) 
:=\sum_{(i,j)\in \{0,\pm 1\}^{2}} d_{i,j}x^i y^j
$$ 
with $d_{i,j}$ is equal to $1$ if $(i,j)\in \mathcal{D}$, and to $0$ otherwise. \ju{(Note that we slightly diverge from the notation of \cite[Lemma~4]{BMM} where the Kernel of the walk is by definition equal to $\frac{K_{\mathcal{D}}(x,y,t)}{xy}$.)}

One  can consider  $S_{\mathcal{D}}(x,y)$ as a Laurent polynomial in $x$ with coefficients that are Laurent {\mfs polynomials} in $y$ and {\it vice versa}.  Using the notations of \cite{BMM,KauersYatchak}, we write 
$$
S_{\mathcal{D}}(x,y)= A_{\mathcal{D},-1}(x) \frac{1}{y} +A_{\mathcal{D},0}(x)+ A_{\mathcal{D},1}(x) y=  B_{\mathcal{D},-1}(y) \frac{1}{x} +B_{\mathcal{D},0}(y)+ B_{\mathcal{D},1}(y) x
$$ 
where $A_{\mathcal{D},i}(x) \in x^{-1}\Q[x]$ and $B_{\mathcal{D},i}(y) \in y^{-1}\Q[y]$. 
The generating series $Q_{\mathcal{D}}(x,y,t)$ satisfies the following functional equation \ju{(see \cite[Lemma~4]{BMM})}:
\begin{equation} \label{eq:funcequ}
K_{\mathcal{D}}(x,y,t)Q_{\mathcal{D}}(x,y,t)=xy-F_{\mathcal{D}}^{1}(x,t) -F_{\mathcal{D}}^{2}(y,t)+td_{-1,-1} Q_{\mathcal{D}}(0,0,t)
\end{equation}
where 
$$
F_{\mathcal{D}}^{1}(x,t):= txA_{\mathcal{D},-1}(x)Q_{\mathcal{D}}(x,0,t), \ \ F_{\mathcal{D}}^{2}(y,t):= tyB_{\mathcal{D},-1}(y)Q_{\mathcal{D}}(0,y,t).$$

\subsection{Classification of the walks with small steps}\label{sec:listofwalkssmallsteps}

There are {\it a priori} $2^8=256$ possible walks with small steps in the quarter plane $\Z_{\geq 0}^{2}$, but, as explained in \cite[$\S 2$]{BMM}, only $138$ of them are truly worthy of consideration. Moreover, taking into account natural symmetries, we are finally left with $79$ inherently different walks to study; see \cite[Figures~1 to 4]{BMM}. 

Following \cite[Section 3]{BMM} or \cite[Section 3]{KauersYatchak}, we attach to any walk in the quarter plane its group, which is by definition the group $\langle i_{1},i_{2} \rangle$ generated by the involutive birational transformations of $\C^{2}$ given by 
$$
i_1(x,y) =\left(x, \frac{A_{\mathcal{D},-1}(x) }{A_{\mathcal{D},1}(x)y}\right) \text{ and }  i_2(x,y)=\left(\frac{B_{\mathcal{D},-1}(y)}{B_{\mathcal{D},1}(y)x},y\right).
$$ 
\ju{These transformations leave $S_{\mathcal{D}}(x,y,t)$ and, hence, $\frac{K_{\mathcal{D}}(x,y,t)}{xy}$ invariant}. 
Amongst the 79 walks mentioned above, $23$ have a finite group and $56$ have an infinite group; see \cite[Theorem 3]{BMM}. 

In the finite group case, the generating series $Q_{\mathcal{D}}(x,y,t)$ is holonomic. This has been proved in \cite{BMM} for $22$ walks, and in \cite{BostanKauersTheCompleteGenerating} for the remaining walk, the so-called Gessel walk (its generating series is actually algebraic; see also \cite{FR10}). 

Amongst the walks having an infinite group, we distinguish the {\it singular} and {\it nonsingular} walks,  that is, those walks whose associated curves $\Etproj$, \ju{defined below in Section~\ref{sec:etproj}}, are singular 
or nonsingular. \ju{In the latter case, $\Etproj$ is an elliptic curve; see Proposition~\ref{prop:genuscurvewalk} below.} Amongst the $56$ walks under consideration, there are $5$ singular walks and $51$ nonsingular walks. \ju{The set of these nonsingular walks, on which this paper focuses, is denoted by $\walks$ and is described in} Figure~\ref{figure:nonsing}. \ju{This Figure reproduces  the table \cite[Figure 17]{KurkRasch} and \ju{uses notations compatible with {\it loc. cit.}}} for the convenience of the reader and for the ease of reference.  \ju{In \cite{KurkRasch}, the authors show that the  generating series for all walks in $\walks$ are nonholonomic}.  In \cite{BBMR16}, the authors show that 9 of these are hyperalgebraic. \ju{The subset of $\walks$ formed by these 9 walks, denoted by $\walkse$, is described in Figure \ref{figure:ex} with references to both  \cite[Tab. 2]{BBMR16} and Figure~\ref{figure:nonsing}}. \ju{We will refer to the elements of $\walkse$ as the ``exceptional walks''.} \ju{As noted in \cite[\S~2.2]{BMM}}, interchanging $x$ and $y$ in the steps leads to equivalent counting problems. The notation \ju{``$\walk{IIB.2}$ (after $x\leftrightarrow y$)''  refers to the walk $\walk{IIB.2}$ with the $x$ and $y$ axes interchanged}. \ju{The complement of $\walkse$ in $\walks$ is denoted by $\walksg$.} \ju{We will refer to the elements of $\walksg$ as the ``typical walks''.}

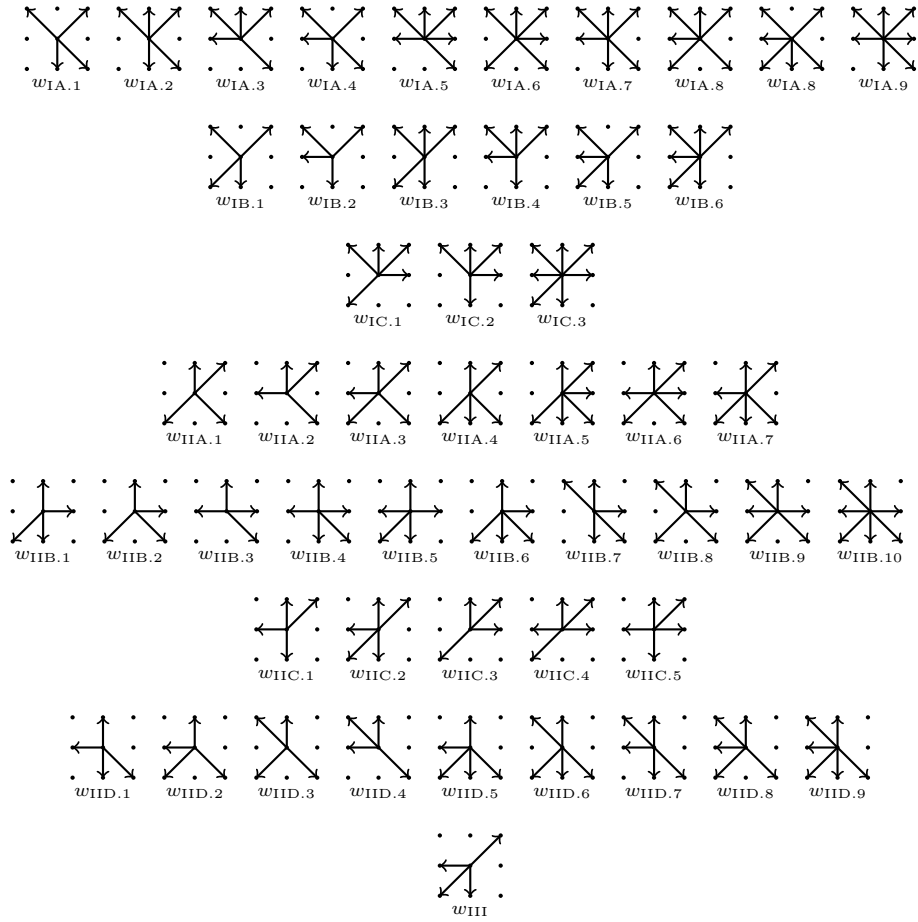
\begin{figure}
\begin{center}
\end{center}
\vskip 3 pt
$$
\underset{\walk{IA.1}}{\begin{tikzpicture}[scale=.4, baseline=(current bounding box.center)]
\foreach \x in {-1,0,1} \foreach \y in {-1,0,1} \fill(\x,\y) circle[radius=2pt];
\draw[thick,->](0,0)--(-1,1);
\draw[thick,->](0,0)--(1,1);
\draw[thick,->](0,0)--(0,-1);
\draw[thick,->](0,0)--(1,-1);
\end{tikzpicture}}
\quad 
\underset{\walk{IA.2}}{\begin{tikzpicture}[scale=.4, baseline=(current bounding box.center)]
\foreach \x in {-1,0,1} \foreach \y in {-1,0,1} \fill(\x,\y) circle[radius=2pt];
\draw[thick,->](0,0)--(-1,1);
\draw[thick,->](0,0)--(0,1);
\draw[thick,->](0,0)--(1,1);
\draw[thick,->](0,0)--(0,-1);
\draw[thick,->](0,0)--(1,-1);
\end{tikzpicture}}
\quad
\underset{\walk{IA.3}}{\begin{tikzpicture}[scale=.4, baseline=(current bounding box.center)]
\foreach \x in {-1,0,1} \foreach \y in {-1,0,1} \fill(\x,\y) circle[radius=2pt];
\draw[thick,->](0,0)--(-1,1);
\draw[thick,->](0,0)--(0,1);
\draw[thick,->](0,0)--(1,1);
\draw[thick,->](0,0)--(-1,0);
\draw[thick,->](0,0)--(1,-1);
\end{tikzpicture}}
\quad
\underset{\walk{IA.4}}{\begin{tikzpicture}[scale=.4, baseline=(current bounding box.center)]
\foreach \x in {-1,0,1} \foreach \y in {-1,0,1} \fill(\x,\y) circle[radius=2pt];
\draw[thick,->](0,0)--(-1,1);
\draw[thick,->](0,0)--(1,1);
\draw[thick,->](0,0)--(-1,0);
\draw[thick,->](0,0)--(0,-1);
\draw[thick,->](0,0)--(1,-1);
\end{tikzpicture}}
\quad
\underset{\walk{IA.5}}{\begin{tikzpicture}[scale=.4, baseline=(current bounding box.center)]
\foreach \x in {-1,0,1} \foreach \y in {-1,0,1} \fill(\x,\y) circle[radius=2pt];
\draw[thick,->](0,0)--(-1,1);
\draw[thick,->](0,0)--(0,1);
\draw[thick,->](0,0)--(1,1);
\draw[thick,->](0,0)--(-1,0);
\draw[thick,->](0,0)--(1,0);
\draw[thick,->](0,0)--(1,-1);
\end{tikzpicture}}
\quad
\underset{\walk{IA.6}}{\begin{tikzpicture}[scale=.4, baseline=(current bounding box.center)]
\foreach \x in {-1,0,1} \foreach \y in {-1,0,1} \fill(\x,\y) circle[radius=2pt];
\draw[thick,->](0,0)--(-1,1);
\draw[thick,->](0,0)--(0,1);
\draw[thick,->](0,0)--(1,1);
\draw[thick,->](0,0)--(1,0);
\draw[thick,->](0,0)--(-1,-1);
\draw[thick,->](0,0)--(1,-1);
\end{tikzpicture}}
\quad
\underset{\walk{IA.7}}{\begin{tikzpicture}[scale=.4, baseline=(current bounding box.center)]
\foreach \x in {-1,0,1} \foreach \y in {-1,0,1} \fill(\x,\y) circle[radius=2pt];
\draw[thick,->](0,0)--(-1,1);
\draw[thick,->](0,0)--(0,1);
\draw[thick,->](0,0)--(1,1);
\draw[thick,->](0,0)--(-1,0);
\draw[thick,->](0,0)--(0,-1);
\draw[thick,->](0,0)--(1,-1);
\end{tikzpicture}}
\quad
\underset{\walk{IA.8}}{\begin{tikzpicture}[scale=.4, baseline=(current bounding box.center)]
\foreach \x in {-1,0,1} \foreach \y in {-1,0,1} \fill(\x,\y) circle[radius=2pt];
\draw[thick,->](0,0)--(-1,1);
\draw[thick,->](0,0)--(0,1);
\draw[thick,->](0,0)--(1,1);
\draw[thick,->](0,0)--(-1,0);
\draw[thick,->](0,0)--(-1,-1);
\draw[thick,->](0,0)--(1,-1);
\end{tikzpicture}}
\quad
\underset{\walk{IA.8}}{\begin{tikzpicture}[scale=.4, baseline=(current bounding box.center)]
\foreach \x in {-1,0,1} \foreach \y in {-1,0,1} \fill(\x,\y) circle[radius=2pt];
\draw[thick,->](0,0)--(-1,1);
\draw[thick,->](0,0)--(1,1);
\draw[thick,->](0,0)--(-1,0);
\draw[thick,->](0,0)--(-1,-1);
\draw[thick,->](0,0)--(0,-1);
\draw[thick,->](0,0)--(1,-1);
\end{tikzpicture}}
\quad
\underset{\walk{IA.9}}{\begin{tikzpicture}[scale=.4, baseline=(current bounding box.center)]
\foreach \x in {-1,0,1} \foreach \y in {-1,0,1} \fill(\x,\y) circle[radius=2pt];
\draw[thick,->](0,0)--(-1,1);
\draw[thick,->](0,0)--(0,1);
\draw[thick,->](0,0)--(1,1);
\draw[thick,->](0,0)--(-1,0);
\draw[thick,->](0,0)--(1,0);
\draw[thick,->](0,0)--(0,-1);
\draw[thick,->](0,0)--(1,-1);
\end{tikzpicture}}
$$
\vskip 3 pt
\begin{center}
\end{center}
\vskip 3 pt
$$
\underset{\walk{IB.1}}{\begin{tikzpicture}[scale=.4, baseline=(current bounding box.center)]
\foreach \x in {-1,0,1} \foreach \y in {-1,0,1} \fill(\x,\y) circle[radius=2pt];
\draw[thick,->](0,0)--(-1,1);
\draw[thick,->](0,0)--(1,1);
\draw[thick,->](0,0)--(-1,-1);
\draw[thick,->](0,0)--(0,-1);
\end{tikzpicture}}
\quad
\underset{\walk{IB.2}}{\begin{tikzpicture}[scale=.4, baseline=(current bounding box.center)]
\foreach \x in {-1,0,1} \foreach \y in {-1,0,1} \fill(\x,\y) circle[radius=2pt];
\draw[thick,->](0,0)--(-1,1);
\draw[thick,->](0,0)--(1,1);
\draw[thick,->](0,0)--(-1,0);
\draw[thick,->](0,0)--(0,-1);
\end{tikzpicture}}
\quad
\underset{\walk{IB.3}}{\begin{tikzpicture}[scale=.4, baseline=(current bounding box.center)]
\foreach \x in {-1,0,1} \foreach \y in {-1,0,1} \fill(\x,\y) circle[radius=2pt];
\draw[thick,->](0,0)--(-1,1);
\draw[thick,->](0,0)--(0,1);
\draw[thick,->](0,0)--(1,1);
\draw[thick,->](0,0)--(-1,-1);
\draw[thick,->](0,0)--(0,-1);
\end{tikzpicture}}
\quad
\underset{\walk{IB.4}}{\begin{tikzpicture}[scale=.4, baseline=(current bounding box.center)]
\foreach \x in {-1,0,1} \foreach \y in {-1,0,1} \fill(\x,\y) circle[radius=2pt];
\draw[thick,->](0,0)--(-1,1);
\draw[thick,->](0,0)--(0,1);
\draw[thick,->](0,0)--(1,1);
\draw[thick,->](0,0)--(-1,0);
\draw[thick,->](0,0)--(0,-1);
\end{tikzpicture}}
\quad
\underset{\walk{IB.5}}{\begin{tikzpicture}[scale=.4, baseline=(current bounding box.center)]
\foreach \x in {-1,0,1} \foreach \y in {-1,0,1} \fill(\x,\y) circle[radius=2pt];
\draw[thick,->](0,0)--(-1,1);
\draw[thick,->](0,0)--(1,1);
\draw[thick,->](0,0)--(-1,0);
\draw[thick,->](0,0)--(-1,-1);
\draw[thick,->](0,0)--(0,-1);
\end{tikzpicture}}
\quad
\underset{\walk{IB.6}}{\begin{tikzpicture}[scale=.4, baseline=(current bounding box.center)]
\foreach \x in {-1,0,1} \foreach \y in {-1,0,1} \fill(\x,\y) circle[radius=2pt];
\draw[thick,->](0,0)--(-1,1);
\draw[thick,->](0,0)--(0,1);
\draw[thick,->](0,0)--(1,1);
\draw[thick,->](0,0)--(-1,0);
\draw[thick,->](0,0)--(-1,-1);
\draw[thick,->](0,0)--(0,-1);
\end{tikzpicture}}
$$
\vskip 3 pt
\begin{center}
\end{center}
\vskip 3 pt
$$
\underset{\walk{IC.1}}{\begin{tikzpicture}[scale=.4, baseline=(current bounding box.center)]
\foreach \x in {-1,0,1} \foreach \y in {-1,0,1} \fill(\x,\y) circle[radius=2pt];
\draw[thick,->](0,0)--(-1,1);
\draw[thick,->](0,0)--(0,1);
\draw[thick,->](0,0)--(1,1);
\draw[thick,->](0,0)--(1,0);
\draw[thick,->](0,0)--(-1,-1);
\end{tikzpicture}}
\quad
\underset{\walk{IC.2}}{\begin{tikzpicture}[scale=.4, baseline=(current bounding box.center)]
\foreach \x in {-1,0,1} \foreach \y in {-1,0,1} \fill(\x,\y) circle[radius=2pt];
\draw[thick,->](0,0)--(-1,1);
\draw[thick,->](0,0)--(0,1);
\draw[thick,->](0,0)--(1,1);
\draw[thick,->](0,0)--(1,0);
\draw[thick,->](0,0)--(0,-1);
\end{tikzpicture}}
\quad
\underset{\walk{IC.3}}{\begin{tikzpicture}[scale=.4, baseline=(current bounding box.center)]
\foreach \x in {-1,0,1} \foreach \y in {-1,0,1} \fill(\x,\y) circle[radius=2pt];
\draw[thick,->](0,0)--(-1,1);
\draw[thick,->](0,0)--(0,1);
\draw[thick,->](0,0)--(1,1);
\draw[thick,->](0,0)--(-1,0);
\draw[thick,->](0,0)--(1,0);
\draw[thick,->](0,0)--(-1,-1);
\draw[thick,->](0,0)--(0,-1);
\end{tikzpicture}}
$$
\vskip 3 pt
\begin{center}
\end{center}
\vskip 3 pt
$$\underset{\walk{IIA.1}}{\begin{tikzpicture}[scale=.4, baseline=(current bounding box.center)]
\foreach \x in {-1,0,1} \foreach \y in {-1,0,1} \fill(\x,\y) circle[radius=2pt];
\draw[thick,->](0,0)--(0,1);
\draw[thick,->](0,0)--(1,1);
\draw[thick,->](0,0)--(-1,-1);
\draw[thick,->](0,0)--(1,-1);
\end{tikzpicture}}
\quad
\underset{\walk{IIA.2}}{\begin{tikzpicture}[scale=.4, baseline=(current bounding box.center)]
\foreach \x in {-1,0,1} \foreach \y in {-1,0,1} \fill(\x,\y) circle[radius=2pt];
\draw[thick,->](0,0)--(0,1);
\draw[thick,->](0,0)--(1,1);
\draw[thick,->](0,0)--(-1,0);
\draw[thick,->](0,0)--(1,-1);
\end{tikzpicture}}
\quad
\underset{\walk{IIA.3}}{\begin{tikzpicture}[scale=.4, baseline=(current bounding box.center)]
\foreach \x in {-1,0,1} \foreach \y in {-1,0,1} \fill(\x,\y) circle[radius=2pt];
\draw[thick,->](0,0)--(0,1);
\draw[thick,->](0,0)--(1,1);
\draw[thick,->](0,0)--(-1,0);
\draw[thick,->](0,0)--(-1,-1);
\draw[thick,->](0,0)--(1,-1);
\end{tikzpicture}}
\quad
\underset{\walk{IIA.4}}{\begin{tikzpicture}[scale=.4, baseline=(current bounding box.center)]
\foreach \x in {-1,0,1} \foreach \y in {-1,0,1} \fill(\x,\y) circle[radius=2pt];
\draw[thick,->](0,0)--(0,1);
\draw[thick,->](0,0)--(1,1);
\draw[thick,->](0,0)--(-1,-1);
\draw[thick,->](0,0)--(0,-1);
\draw[thick,->](0,0)--(1,-1);
\end{tikzpicture}}
\quad
\underset{\walk{IIA.5}}{\begin{tikzpicture}[scale=.4, baseline=(current bounding box.center)]
\foreach \x in {-1,0,1} \foreach \y in {-1,0,1} \fill(\x,\y) circle[radius=2pt];
\draw[thick,->](0,0)--(0,1);
\draw[thick,->](0,0)--(1,1);
\draw[thick,->](0,0)--(1,0);
\draw[thick,->](0,0)--(-1,-1);
\draw[thick,->](0,0)--(0,-1);
\draw[thick,->](0,0)--(1,-1);
\end{tikzpicture}}
\quad
\underset{\walk{IIA.6}}{\begin{tikzpicture}[scale=.4, baseline=(current bounding box.center)]
\foreach \x in {-1,0,1} \foreach \y in {-1,0,1} \fill(\x,\y) circle[radius=2pt];
\draw[thick,->](0,0)--(0,1);
\draw[thick,->](0,0)--(1,1);
\draw[thick,->](0,0)--(-1,0);
\draw[thick,->](0,0)--(1,0);
\draw[thick,->](0,0)--(-1,-1);
\draw[thick,->](0,0)--(1,-1);
\end{tikzpicture}}
\quad 
\underset{\walk{IIA.7}}{\begin{tikzpicture}[scale=.4, baseline=(current bounding box.center)]
\foreach \x in {-1,0,1} \foreach \y in {-1,0,1} \fill(\x,\y) circle[radius=2pt];
\draw[thick,->](0,0)--(0,1);
\draw[thick,->](0,0)--(1,1);
\draw[thick,->](0,0)--(-1,0);
\draw[thick,->](0,0)--(-1,-1);
\draw[thick,->](0,0)--(0,-1);
\draw[thick,->](0,0)--(1,-1);
\end{tikzpicture}}
$$
\vskip 3 pt
\begin{center}
\end{center}
\vskip 3 pt
$$
\underset{\walk{IIB.1}}{\begin{tikzpicture}[scale=.4, baseline=(current bounding box.center)]
\foreach \x in {-1,0,1} \foreach \y in {-1,0,1} \fill(\x,\y) circle[radius=2pt];
\draw[thick,->](0,0)--(0,1);
\draw[thick,->](0,0)--(1,0);
\draw[thick,->](0,0)--(-1,-1);
\draw[thick,->](0,0)--(0,-1);
\end{tikzpicture}}
\quad 
\underset{\walk{IIB.2}}{\begin{tikzpicture}[scale=.4, baseline=(current bounding box.center)]
\foreach \x in {-1,0,1} \foreach \y in {-1,0,1} \fill(\x,\y) circle[radius=2pt];
\draw[thick,->](0,0)--(0,1);
\draw[thick,->](0,0)--(1,0);
\draw[thick,->](0,0)--(-1,-1);
\draw[thick,->](0,0)--(1,-1);
\end{tikzpicture}}
\quad
\underset{\walk{IIB.3}}{\begin{tikzpicture}[scale=.4, baseline=(current bounding box.center)]
\foreach \x in {-1,0,1} \foreach \y in {-1,0,1} \fill(\x,\y) circle[radius=2pt];
\draw[thick,->](0,0)--(0,1);
\draw[thick,->](0,0)--(-1,0);
\draw[thick,->](0,0)--(1,0);
\draw[thick,->](0,0)--(1,-1);
\end{tikzpicture}}
\quad
\underset{\walk{IIB.4}}{\begin{tikzpicture}[scale=.4, baseline=(current bounding box.center)]
\foreach \x in {-1,0,1} \foreach \y in {-1,0,1} \fill(\x,\y) circle[radius=2pt];
\draw[thick,->](0,0)--(0,1);
\draw[thick,->](0,0)--(-1,0);
\draw[thick,->](0,0)--(1,0);
\draw[thick,->](0,0)--(0,-1);
\draw[thick,->](0,0)--(1,-1);
\end{tikzpicture}}
\quad
\underset{\walk{IIB.5}}{\begin{tikzpicture}[scale=.4, baseline=(current bounding box.center)]
\foreach \x in {-1,0,1} \foreach \y in {-1,0,1} \fill(\x,\y) circle[radius=2pt];
\draw[thick,->](0,0)--(0,1);
\draw[thick,->](0,0)--(-1,0);
\draw[thick,->](0,0)--(1,0);
\draw[thick,->](0,0)--(-1,-1);
\draw[thick,->](0,0)--(0,-1);
\end{tikzpicture}}
\quad
\underset{\walk{IIB.6}}{\begin{tikzpicture}[scale=.4, baseline=(current bounding box.center)]
\foreach \x in {-1,0,1} \foreach \y in {-1,0,1} \fill(\x,\y) circle[radius=2pt];
\draw[thick,->](0,0)--(0,1);
\draw[thick,->](0,0)--(1,0);
\draw[thick,->](0,0)--(-1,-1);
\draw[thick,->](0,0)--(0,-1);
\draw[thick,->](0,0)--(1,-1);
\end{tikzpicture}}
\quad
\underset{\walk{IIB.7}}{\begin{tikzpicture}[scale=.4, baseline=(current bounding box.center)]
\foreach \x in {-1,0,1} \foreach \y in {-1,0,1} \fill(\x,\y) circle[radius=2pt];
\draw[thick,->](0,0)--(-1,1);
\draw[thick,->](0,0)--(0,1);
\draw[thick,->](0,0)--(1,0);
\draw[thick,->](0,0)--(0,-1);
\draw[thick,->](0,0)--(1,-1);
\end{tikzpicture}}
\quad
\underset{\walk{IIB.8}}{\begin{tikzpicture}[scale=.4, baseline=(current bounding box.center)]
\foreach \x in {-1,0,1} \foreach \y in {-1,0,1} \fill(\x,\y) circle[radius=2pt];
\draw[thick,->](0,0)--(-1,1);
\draw[thick,->](0,0)--(0,1);
\draw[thick,->](0,0)--(1,0);
\draw[thick,->](0,0)--(-1,-1);
\draw[thick,->](0,0)--(1,-1);
\end{tikzpicture}}
\quad
\underset{\walk{IIB.9}}{\begin{tikzpicture}[scale=.4, baseline=(current bounding box.center)]
\foreach \x in {-1,0,1} \foreach \y in {-1,0,1} \fill(\x,\y) circle[radius=2pt];
\draw[thick,->](0,0)--(-1,1);
\draw[thick,->](0,0)--(0,1);
\draw[thick,->](0,0)--(-1,0);
\draw[thick,->](0,0)--(1,0);
\draw[thick,->](0,0)--(-1,-1);
\draw[thick,->](0,0)--(1,-1);
\end{tikzpicture}}
\quad
\underset{\walk{IIB.10}}{\begin{tikzpicture}[scale=.4, baseline=(current bounding box.center)]
\foreach \x in {-1,0,1} \foreach \y in {-1,0,1} \fill(\x,\y) circle[radius=2pt];
\draw[thick,->](0,0)--(-1,1);
\draw[thick,->](0,0)--(0,1);
\draw[thick,->](0,0)--(-1,0);
\draw[thick,->](0,0)--(1,0);
\draw[thick,->](0,0)--(-1,-1);
\draw[thick,->](0,0)--(0,-1);
\draw[thick,->](0,0)--(1,-1);
\end{tikzpicture}} \quad
$$
\vskip 3 pt
\begin{center}
\end{center}
\vskip 3 pt
$$
\underset{\walk{IIC.1}}{\begin{tikzpicture}[scale=.4, baseline=(current bounding box.center)]
\foreach \x in {-1,0,1} \foreach \y in {-1,0,1} \fill(\x,\y) circle[radius=2pt];
\draw[thick,->](0,0)--(0,1);
\draw[thick,->](0,0)--(1,1);
\draw[thick,->](0,0)--(-1,0);
\draw[thick,->](0,0)--(0,-1);
\end{tikzpicture}}
\quad
\underset{\walk{IIC.2}}{\begin{tikzpicture}[scale=.4, baseline=(current bounding box.center)]
\foreach \x in {-1,0,1} \foreach \y in {-1,0,1} \fill(\x,\y) circle[radius=2pt];
\draw[thick,->](0,0)--(0,1);
\draw[thick,->](0,0)--(1,1);
\draw[thick,->](0,0)--(-1,0);
\draw[thick,->](0,0)--(-1,-1);
\draw[thick,->](0,0)--(0,-1);
\end{tikzpicture}}
\quad
\underset{\walk{IIC.3}}{\begin{tikzpicture}[scale=.4, baseline=(current bounding box.center)]
\foreach \x in {-1,0,1} \foreach \y in {-1,0,1} \fill(\x,\y) circle[radius=2pt];
\draw[thick,->](0,0)--(0,1);
\draw[thick,->](0,0)--(1,1);
\draw[thick,->](0,0)--(1,0);
\draw[thick,->](0,0)--(-1,-1);
\end{tikzpicture}} 
\quad 
\underset{\walk{IIC.4}}{\begin{tikzpicture}[scale=.4, baseline=(current bounding box.center)]
\foreach \x in {-1,0,1} \foreach \y in {-1,0,1} \fill(\x,\y) circle[radius=2pt];
\draw[thick,->](0,0)--(0,1);
\draw[thick,->](0,0)--(1,1);
\draw[thick,->](0,0)--(-1,0);
\draw[thick,->](0,0)--(1,0);
\draw[thick,->](0,0)--(-1,-1);
\end{tikzpicture}} 
\quad 
\underset{\walk{IIC.5}}{\begin{tikzpicture}[scale=.4, baseline=(current bounding box.center)]
\foreach \x in {-1,0,1} \foreach \y in {-1,0,1} \fill(\x,\y) circle[radius=2pt];
\draw[thick,->](0,0)--(0,1);
\draw[thick,->](0,0)--(1,1);
\draw[thick,->](0,0)--(-1,0);
\draw[thick,->](0,0)--(1,0);
\draw[thick,->](0,0)--(0,-1);
\end{tikzpicture}}
$$
\vskip 3 pt
\begin{center}
\end{center}
\vskip 3 pt
$$
\underset{\walk{IID.1}}{\begin{tikzpicture}[scale=.4, baseline=(current bounding box.center)]
\foreach \x in {-1,0,1} \foreach \y in {-1,0,1} \fill(\x,\y) circle[radius=2pt];
\draw[thick,->](0,0)--(0,1);
\draw[thick,->](0,0)--(-1,0);
\draw[thick,->](0,0)--(0,-1);
\draw[thick,->](0,0)--(1,-1);
\end{tikzpicture}}
\quad
\underset{\walk{IID.2}}{\begin{tikzpicture}[scale=.4, baseline=(current bounding box.center)]
\foreach \x in {-1,0,1} \foreach \y in {-1,0,1} \fill(\x,\y) circle[radius=2pt];
\draw[thick,->](0,0)--(0,1);
\draw[thick,->](0,0)--(-1,0);
\draw[thick,->](0,0)--(-1,-1);
\draw[thick,->](0,0)--(1,-1);
\end{tikzpicture}}
\quad
\underset{\walk{IID.3}}{\begin{tikzpicture}[scale=.4, baseline=(current bounding box.center)]
\foreach \x in {-1,0,1} \foreach \y in {-1,0,1} \fill(\x,\y) circle[radius=2pt];
\draw[thick,->](0,0)--(-1,1);
\draw[thick,->](0,0)--(0,1);
\draw[thick,->](0,0)--(-1,-1);
\draw[thick,->](0,0)--(1,-1);
\end{tikzpicture}}
\quad
\underset{\walk{IID.4}}{\begin{tikzpicture}[scale=.4, baseline=(current bounding box.center)]
\foreach \x in {-1,0,1} \foreach \y in {-1,0,1} \fill(\x,\y) circle[radius=2pt];
\draw[thick,->](0,0)--(-1,1);
\draw[thick,->](0,0)--(0,1);
\draw[thick,->](0,0)--(-1,0);
\draw[thick,->](0,0)--(1,-1);
\end{tikzpicture}}
\quad
\underset{\walk{IID.5}}{\begin{tikzpicture}[scale=.4, baseline=(current bounding box.center)]
\foreach \x in {-1,0,1} \foreach \y in {-1,0,1} \fill(\x,\y) circle[radius=2pt];
\draw[thick,->](0,0)--(0,1);
\draw[thick,->](0,0)--(-1,0);
\draw[thick,->](0,0)--(-1,-1);
\draw[thick,->](0,0)--(0,-1);
\draw[thick,->](0,0)--(1,-1);
\end{tikzpicture}}
\quad
\underset{\walk{IID.6}}{\begin{tikzpicture}[scale=.4, baseline=(current bounding box.center)]
\foreach \x in {-1,0,1} \foreach \y in {-1,0,1} \fill(\x,\y) circle[radius=2pt];
\draw[thick,->](0,0)--(-1,1);
\draw[thick,->](0,0)--(0,1);
\draw[thick,->](0,0)--(-1,-1);
\draw[thick,->](0,0)--(0,-1);
\draw[thick,->](0,0)--(1,-1);
\end{tikzpicture}}
\quad 
\underset{\walk{IID.7}}{\begin{tikzpicture}[scale=.4, baseline=(current bounding box.center)]
\foreach \x in {-1,0,1} \foreach \y in {-1,0,1} \fill(\x,\y) circle[radius=2pt];
\draw[thick,->](0,0)--(-1,1);
\draw[thick,->](0,0)--(0,1);
\draw[thick,->](0,0)--(-1,0);
\draw[thick,->](0,0)--(0,-1);
\draw[thick,->](0,0)--(1,-1);
\end{tikzpicture}}
\quad 
\underset{\walk{IID.8}}{\begin{tikzpicture}[scale=.4, baseline=(current bounding box.center)]
\foreach \x in {-1,0,1} \foreach \y in {-1,0,1} \fill(\x,\y) circle[radius=2pt];
\draw[thick,->](0,0)--(-1,1);
\draw[thick,->](0,0)--(0,1);
\draw[thick,->](0,0)--(-1,0);
\draw[thick,->](0,0)--(-1,-1);
\draw[thick,->](0,0)--(1,-1);
\end{tikzpicture}}
\quad
\underset{\walk{IID.9}}{\begin{tikzpicture}[scale=.4, baseline=(current bounding box.center)]
\foreach \x in {-1,0,1} \foreach \y in {-1,0,1} \fill(\x,\y) circle[radius=2pt];
\draw[thick,->](0,0)--(-1,1);
\draw[thick,->](0,0)--(0,1);
\draw[thick,->](0,0)--(-1,0);
\draw[thick,->](0,0)--(-1,-1);
\draw[thick,->](0,0)--(0,-1);
\draw[thick,->](0,0)--(1,-1);
\end{tikzpicture}}
$$
\vskip 3 pt
\begin{center}
\end{center}
\vskip 3 pt
$$
\underset{\walk{III}}{\begin{tikzpicture}[scale=.4, baseline=(current bounding box.center)]
\foreach \x in {-1,0,1} \foreach \y in {-1,0,1} \fill(\x,\y) circle[radius=2pt];
\draw[thick,->](0,0)--(1,1);
\draw[thick,->](0,0)--(-1,0);
\draw[thick,->](0,0)--(-1,-1);
\draw[thick,->](0,0)--(0,-1);
\end{tikzpicture}}
$$
\caption{\ju{The 51 elements of the set $\walks$ of nonsingular walks with small steps and infinite group}}
\label{figure:nonsing}
\end{figure}

\begin{figure}
$$
\begin{tabular}{|c|c|}
  \hline
\cite[Tab 2]{BBMR16} & Figure 1. \\
  \hline
  1 & $\walk{IIB.1}$ (after $x\leftrightarrow y$)  \\
  2 & $\walk{IIB.2}$ (after $x\leftrightarrow y$)  \\
  3 & $\walk{IIC.1}$  \\
  4 & $\walk{IIB.3}$  \\
  5 & $\walk{IIC.4}$ \\
  6 & $\walk{IIC.2}$  \\
  7 & $\walk{IIB.6}$ (after $x\leftrightarrow y$)  \\
  8 & $\walk{IIC.5}$  \\
  9 & $\walk{IIB.7}$  \\
  \hline
\end{tabular}
$$
$$
\underset{\walk{IIB.1}}{\begin{tikzpicture}[scale=.4, baseline=(current bounding box.center)]
\foreach \x in {-1,0,1} \foreach \y in {-1,0,1} \fill(\x,\y) circle[radius=2pt];
\draw[thick,->](0,0)--(0,1);
\draw[thick,->](0,0)--(1,0);
\draw[thick,->](0,0)--(-1,-1);
\draw[thick,->](0,0)--(0,-1);
\end{tikzpicture}}
\quad 
\underset{\walk{IIB.2}}{\begin{tikzpicture}[scale=.4, baseline=(current bounding box.center)]
\foreach \x in {-1,0,1} \foreach \y in {-1,0,1} \fill(\x,\y) circle[radius=2pt];
\draw[thick,->](0,0)--(0,1);
\draw[thick,->](0,0)--(1,0);
\draw[thick,->](0,0)--(-1,-1);
\draw[thick,->](0,0)--(1,-1);
\end{tikzpicture}
}
\quad
\underset{\walk{IIC.1}}{\begin{tikzpicture}[scale=.4, baseline=(current bounding box.center)]
\foreach \x in {-1,0,1} \foreach \y in {-1,0,1} \fill(\x,\y) circle[radius=2pt];
\draw[thick,->](0,0)--(0,1);
\draw[thick,->](0,0)--(1,1);
\draw[thick,->](0,0)--(-1,0);
\draw[thick,->](0,0)--(0,-1);
\end{tikzpicture}
}
\quad
\underset{\walk{IIB.3}}{\begin{tikzpicture}[scale=.4, baseline=(current bounding box.center)]
\foreach \x in {-1,0,1} \foreach \y in {-1,0,1} \fill(\x,\y) circle[radius=2pt];
\draw[thick,->](0,0)--(0,1);
\draw[thick,->](0,0)--(-1,0);
\draw[thick,->](0,0)--(1,0);
\draw[thick,->](0,0)--(1,-1);
\end{tikzpicture}}
\quad
\underset{\walk{IIC.4}}{\begin{tikzpicture}[scale=.4, baseline=(current bounding box.center)]
\foreach \x in {-1,0,1} \foreach \y in {-1,0,1} \fill(\x,\y) circle[radius=2pt];
\draw[thick,->](0,0)--(0,1);
\draw[thick,->](0,0)--(1,1);
\draw[thick,->](0,0)--(-1,0);
\draw[thick,->](0,0)--(1,0);
\draw[thick,->](0,0)--(-1,-1);
\end{tikzpicture}}
\quad
\underset{\walk{IIC.2}}{\begin{tikzpicture}[scale=.4, baseline=(current bounding box.center)]
\foreach \x in {-1,0,1} \foreach \y in {-1,0,1} \fill(\x,\y) circle[radius=2pt];
\draw[thick,->](0,0)--(0,1);
\draw[thick,->](0,0)--(1,1);
\draw[thick,->](0,0)--(-1,0);
\draw[thick,->](0,0)--(-1,-1);
\draw[thick,->](0,0)--(0,-1);
\end{tikzpicture}}
\quad
\underset{\walk{IIB.6}}{\begin{tikzpicture}[scale=.4, baseline=(current bounding box.center)]
\foreach \x in {-1,0,1} \foreach \y in {-1,0,1} \fill(\x,\y) circle[radius=2pt];
\draw[thick,->](0,0)--(0,1);
\draw[thick,->](0,0)--(1,0);
\draw[thick,->](0,0)--(-1,-1);
\draw[thick,->](0,0)--(0,-1);
\draw[thick,->](0,0)--(1,-1);
\end{tikzpicture}}
\quad 
\underset{\walk{IIC.5}}{\begin{tikzpicture}[scale=.4, baseline=(current bounding box.center)]
\foreach \x in {-1,0,1} \foreach \y in {-1,0,1} \fill(\x,\y) circle[radius=2pt];
\draw[thick,->](0,0)--(0,1);
\draw[thick,->](0,0)--(1,1);
\draw[thick,->](0,0)--(-1,0);
\draw[thick,->](0,0)--(1,0);
\draw[thick,->](0,0)--(0,-1);
\end{tikzpicture}}
\quad
\underset{\walk{IIB.7}}{\begin{tikzpicture}[scale=.4, baseline=(current bounding box.center)]
\foreach \x in {-1,0,1} \foreach \y in {-1,0,1} \fill(\x,\y) circle[radius=2pt];
\draw[thick,->](0,0)--(-1,1);
\draw[thick,->](0,0)--(0,1);
\draw[thick,->](0,0)--(1,0);
\draw[thick,->](0,0)--(0,-1);
\draw[thick,->](0,0)--(1,-1);
\end{tikzpicture}}
$$
\caption{\ju{The 9 elements of the set $\walkse \subset \walks$ of exceptional nonsingular walks with small steps and infinite group}}
\label{figure:ex}
\end{figure}
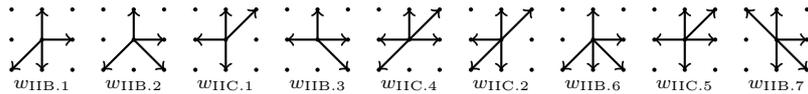

\subsection{The algebraic curve $\Etproj$ defined by the kernel $K(x,y,t)$}\label{sec:etproj}

{We fix $\mathcal{D}$ a set of \ju{authorized steps} and }$0<t< 1/|\mathcal{D}|$. \ju{We will omit the subscript $\mathcal{D}$ in our notations; for instance, the kernel $K_{\mathcal{D}}$ will be denoted by $K$.}  We \ju{denote by $\overline{K}$ the} {\mfs homogenized polynomial} 
\begin{equation}\label{eq:kernelwalk}
\overline{K}(x_0,x_1,y_0,y_1,t)
\ju{= (x_{1}y_{1})^{2}K(x_{0}/x_{1},y_{0}/y_{1},t)}
= x_0x_1y_0y_1 -t \sum_{i,j=0}^2 d_{i-1,j-1} x_0^i x_1^{2-i}y_0^j y_1^{2-j}. 
 \end{equation}
We let  $\Etproj$ be the algebraic curve in $\P1(\C) \times \P1(\C)$ defined by $\overline{K}$, {\it i.e.}, 
$$
\Etproj=\{(x,y)=([x_0:x_1], [y_0:y_1]) \in \P1(\C) \times \P1(\C) \ \vert \ \overline{K}(x_0,x_1,y_0,y_1,t)=0\}. 
$$ 
The intersection of $\Etproj$ with $(\P1(\C) \setminus \{\infty\})\times (\P1(\C) \setminus \{\infty\})$, where $\infty=[1:0]$, will be denoted by $\Et$ and identified with a subset of $\C \times \C$, {\it i.e.}, 
$$
\Et=\{(x,y) \in \C \times \C \ \vert \ K(x,y,t)=0\}. 
$$
Such curves have been studied in detail by Duistermaat in \cite[Chapter 2]{DuistQRT}. 

\subsection{Smoothness and genus of $\Etproj$}

We fix $0<t< 1/|\mathcal{D}|$. For any $[x_0:x_1]$ and $[y_0:y_1]$ in $\P1(\C)$, we denote by $\Delta^x_{[x_0:x_1]}$ and $\Delta^y_{[y_0:y_1]}$ the discriminants of the degree~$2$ homogeneous polynomials given by $y \mapsto \overline{K}(x_0,x_1,y,t)$ and  $x \mapsto \overline{K}(x,y_0,y_1,t)$ respectively, {\it i.e.},
\begin{multline}\label{eq:ydiscriminant}
\Delta^x_{[x_0:x_1]}=t^2 \Big[ (d_{-1,0} x_1^2 -  \frac{1}{t} x_0x_1 + d_{1,0}x_0^2)^2  \\
- 4(d_{-1,1} x_1^2 + d_{0,1} x_0x_1 + d_{1,1}x_0^2)(d_{-1,-1} x_1^2 + d_{0,-1} x_0x_1 + d_{1,-1}x_0^2) \Big] \nonumber
\end{multline}
and 
\begin{multline}
\Delta^y_{[y_0:y_1]}=t^2 \Big[ (d_{0,-1} y_1^2 -  \frac{1}{t} y_0y_1 + d_{0,1}y_0^2)^2  \\ 
 - 4(d_{1,-1} y_1^2 + d_{1,0} y_0y_1 + d_{1,1}y_0^2)(d_{-1,-1} y_1^2 + d_{-1,0} y_0y_1 + d_{-1,1}y_0^2) \Big].\nonumber
\end{multline}

\begin{prop}[{\cite[\S  2.4.1, especially Proposition 2.4.3]{DuistQRT}}]\label{prop:genuscurvewalk}
\ju{The following properties are equivalent~:
\begin{itemize}
\item the curve $\Etproj$ is smooth; 
\item the discriminant $\Delta^x_{[x_0:x_1]}$ has simple roots in $\C^{2} \setminus \{(0,0)\}$; 
\item the discriminant $\Delta^y_{[y_0: y_1]}$ has simple roots in $\C^{2} \setminus \{(0,0)\}$.
\end{itemize} }
Moreover, if $\Etproj$ is smooth, then it has genus $1$, {\it i.e.}, it is an elliptic curve. 
\end{prop}

\begin{rem}
In \cite[Section 2.3.2]{FIM},  other explicit equivalent conditions on the $d_{i,j}$ are given ensuring that the discriminants have simple roots.
\end{rem}

\begin{cor}\label{cor:genuscurvewalk}
For any walk in \ju{$\walks$}, the algebraic curve $\Etproj$ is an elliptic curve. 
\end{cor}

\begin{proof}
This is a direct consequence of Proposition \ref{prop:genuscurvewalk} in virtue of \cite[Section~2.1]{KurkRasch} and \cite[Part 2.3]{FIM}. 
\end{proof}

\begin{rem}
We work in $\P1(\C)\times \P1(\C)$ instead of the projective plane over $\C$ in order to get a smooth curve.{  Moreover the smoothness criteria of Proposition~\ref{prop:genuscurvewalk} can be encoded as follows. Following \cite[Proposition 2.4.3]{DuistQRT}, the curve \ju{$\Etproj$} is  smooth if and only if  the Eisenstein invariants $F_x$ (resp. $F_y$) of $\Delta^x_{[x_0:x_1]}$ (resp. $\Delta^y_{[y_0: y_1]}$)  are non zero (see \cite[Section 2.3.5]{DuistQRT}).  The Eisenstein invariants $F_x$ and $F_y$ are  given by explicit polynomial formulas in $t$ and the $d_{i,j}$. Furthermore, if \ju{$\Etproj$} is smooth then it is an elliptic curve with modulus $J$ that 
can be explicitly computed thanks to the Eisenstein invariants (see \cite[2.3.23]{DuistQRT}).}

\end{rem}

\subsection{The involutions $\iota_{1}$ and $\iota_{2}$ and the QRT mapping $\tau$ of $\Etproj$}\label{subsec:group of the walk}

In this section, we fix $0<t< 1/|\mathcal{D}|$ and we assume that $\Etproj$ is an elliptic curve. 

We let $\iota_{1}$ and $\iota_{2}$ be the involutions of $\Etproj$ induced by $i_{1}$ and $i_{2}$, {\it i.e.},       
$$
\iota_1(x,y) =\left(\frac{x_{0}}{x_{1}}, \frac{A_{-1}(\frac{x_{0}}{x_{1}}) }{A_1(\frac{x_{0}}{x_{1}})\frac{y_{0}}{y_{1}}}\right) 
\text{ and } 
\iota_2(x,y)=\left(\frac{B_{-1}(\frac{y_{0}}{y_{1}})}{B_1(\frac{y_{0}}{y_{1}})\frac{x_{0}}{x_{1}}},\frac{y_{0}}{y_{1}}\right).
$$
These formulas define $\iota_{1}$ and $\iota_{2}$ as rational maps from $\Etproj$ to itself, but, since $\Etproj$ is a smooth projective curve, they are actually endomorphisms of $\Etproj$. They are nothing but the vertical and horizontal  switches of $\Etproj$, {\it i.e.}, for any $P=(x,y) \in \Etproj$, we have 
$$
\{P,\iota_1(P)\} = \Etproj \cap (\{x\} \times \P1(\C))
\text{ and }
\{P,\iota_2(P)\} = \Etproj \cap (\P1(\C) \times \{y\}).
$$


\begin{lemma}\label{lemma:fixedpointinvolution}
A point $P=([x_0: x_1],[y_0:y_1]) \in \Etproj$ is fixed by $\iota_1$ (resp. $\iota_2$)
if and only if $\Delta^x_{[x_0: x_1]}=0$ (resp. $\Delta^y_{[y_0: y_1]}=0$).
\end{lemma}

\ju{
\begin{proof}
Since $\iota_{1}$ is the vertical switch of $\Etproj$, $P$ is fixed by $\iota_{1}$ if and only if $y \mapsto \overline{K}(x_0,x_1,y,t)$ has a unique root in $\P1(\C)$. The last property is equivalent to $\Delta^x_{[x_0: x_1]}=0$. The proof for $\iota_{2}$ is similar. 
\end{proof}
}

The automorphism $\tau$ of $\Etproj$ given by 
$$ 
\tau= \iota_2 \circ \iota_1,
$$
which will play a central role in this paper, is called the QRT mapping of $\Etproj$. 
According to \cite[Proposition 2.5.2]{DuistQRT}, $\tau$ is the addition by a point of the elliptic curve $\Etproj$.

We denote by $G_t$ the subgroup of $\Aut(\Etproj)$, that is the group formed of the automorphisms of $\Etproj$, generated by $\iota_1$ and $\iota_2$. \ju{Note that $G_{t}$ has finite order if and only if $\tau$ has finite order.}

If the group of the walk is finite then $G_t$ is finite as well. However, if the group of the walk is infinite, there are {\it a priori} no reason why $G_t$ should be infinite for all $0<t<1/\vert \mathcal{D}\vert$ (and this is false in general). However, we have the following result.  

\begin{prop}\label{prop:groupwalkrestrEt}
{For any walk with infinite group}, the set of $t \in \C$
such that $G_{t}$ is infinite, {\it i.e.}, such that $\tau$ has infinite order, has
 denumerable complement in  $\C$.
\end{prop}

\begin{proof}
This follows from \cite[Remark 6, Proposition 14]{KurkRasch}. However, we give an alternate more elementary proof \ju{here}below. \\[0.1in]
\ju{So, we assume that the group of the walk under consideration is infinite} and we \ju{must} show that the set of $t \in \C$
such that $\tau$ has finite order on $\Etproj$ is denumerable.
 To do this it suffices to show that, for each positive integer $n$, the set of $t \in\C$ 
  such that $\tau^n$ is the identity on $\Etproj$ is finite.
 We know that the walks under consideration have an infinite group, {\it i.e.}, that the birational transformation of $\C^{2}$ given by 
 $$
 f= \iota_{2} \circ \iota_{1}
 $$ 
 has infinite order.  
Fix a  value of $n>0$. Using the formulas for $\iota_1$ and $\iota_2$ in Section~\ref{sec:listofwalkssmallsteps}, one sees that 
\[f^n(x,y) = \left(\frac{p_1(x,y)}{p_2(x,y)}, \frac{q_1(x,y)}{q_2(x,y)}\right)\] 
where $p_1,p_2,q_1,q_2 \in \Q[x,y]$ are independent of $t$. This birational map is well defined on the complement of the curve $Z \subset \C^{2}$ defined by $p_2(x,y)q_2(x,y)=0$.   For any $t \in \C$, B\'ezout's theorem ensures that either $\Et$ and $Z$ have an irreducible component in common or $\Et \cap Z$ is finite. 
Let $\mathcal{S}$ be the set of $t \in \C$ such that $\Et$ and $Z$ have an irreducible component in common. We claim that $\mathcal{S}$ is finite.  Assume not.  For $t \in  \mathcal{S}$, $K(x,y,t)=xy(1-t S(x,y))$ and $p_2(x,y)q_2(x,y)$ have a nonconstant factor in common. Since $\mathcal{S}$ is infinite, there are two values of $t$, $t_1$ and $t_2$, such that $K(x,y,t_{1})=xy(1-t_1S(x,y))$ and $K(x,y,t_{2})=xy(1-t_2S(x,y))$ will both be divisible by the same nonconstant factor $d(x,y) \in \C[x,y]$ of $p_2(x,y)q_2(x,y)$. This implies that $d(x,y)$ is a factor of $xy$ and $xyS(x,y)$, and, hence, of $K(x,y,t)$. But, for the walks we consider, \ju{none of the factors of $xy$ is a factor of} $K(x,y,t)$ \ju{({\it i.e.}, neither $x$ nor $y$ is a factor of $K(x,y,t)$ in $\C[x,y,t]$), whence a contradiction}. \ju{Therefore, $\mathcal{S}$ is finite.}

 Let $X_n$ be the set of $(x,y,t) \in \C^3$ such that \begin{itemize}
\item $t\notin \mathcal{S}$,
\item $(x,y) \in \Et$,
\item $p_2(x,y)q_2(x,y) \neq 0$,
\item $f^n(x,y) = \tau^{n}(x,y) = (x,y)$.
\end{itemize}

One sees that $X_n$ is a constructible set\footnote{A subset of $\C^m$ is constructible if it lies in the boolean algebra generated by the Zariski closed sets.}. The projection of a constructible set onto a subset of its coordinates is again constructible \cite[Ch.5.6, Cor.2, Ex.1]{CLO2005} so the set
\[Y_n = \{t \in \C \ | \ \exists (x,y) \in \C^2 \mbox{ s.t. } (x,y,t) \in X_n\}\]
is also constructible. If $t \in Y_n$, then there is a point $(x,y) \in  \Et$ such that $f^n$ is defined at this point and $f^{n}(x,y)=\tau^n(x,y) = (x,y)$. Since $\tau^n$ is given by the addition of a  point on  $\overline{E_t}$, it must leave all of $\overline{E_t}$ fixed. Conversely, if $t\notin \mathcal{S}$ and $\tau^n$ is the identity on $\overline{E_{t}}$ then for some $(x,y) \in \C^{2}$,  $(x,y,t)\in X_n$ and so $t \in Y_n$.  We will have completed the proof once we show that $Y_n$ is finite.  If $Y_n$ is not finite then, since  constructible subsets of $\C$ are either finite or cofinite, it must contain an open set $U$. The set of points \ju{$\cup_{t\in Y_n} \Et$} will then contain an open subset $V$ of $\C\times \C$ such that $f^n$ leaves $V$ pointwise fixed. Therefore,  $f^n$ would be the identity on $\C \times \C$, \ju{and this contradicts the fact that $f$ has infinite order}. Therefore $Y_n$ is finite. 
\end{proof}

\subsection{Functional equations satisfied by ${F^{1}(x,t)}$ and $F^{2}(y,t)$}\label{subsec:func equ elliptic curve}
In this section, following \cite{KurkRasch}, we describe how one identifies the formal series $F^{1}(x,t)$ and $F^{2}(y,t)$ with meromorphic functions on a suitable domain and  give the functional equations satisfied by these functions for the walks in \ju{$\walks$}. 

\ju{In this section, we consider a walk in $\walks$.} We recall that the corresponding generating series $Q_{\mathcal{D}}(x,y,t)$ satisfies the functional equation \eqref{eq:funcequ}.
This equation is formal yet, but for $|x|,|y| < 1$ and $0<t<\frac{1}{|\mathcal{D}|}$, the series $Q_{\mathcal{D}}(x,y,t)$, $F_{\mathcal{D}}^{1}(x,t)$ and $F_{\mathcal{D}}^{2}(y,t)$ are convergent.  Therefore we have 

\begin{equation} 
0=xy-F_{\mathcal{D}}^{1}(x,t) -F_{\mathcal{D}}^{2}(y,t)+td_{-1,-1} Q_{\mathcal{D}}(0,0,t)
\end{equation}
for all $x,y \in V := E_t\cap \{|x|, |y| < 1 \}$. This $V$ is a non empty open subset of $E_{t}$ as explained in \cite[Section 4.1]{KurkRasch}.
 In particular, $F^{1}(x,t)$ and $F^{2}(y,t)$ yield  analytic functions on some small pieces of the elliptic curve $E_t$.  Thanks to uniformization, we can identify {\mfs $\overline{E_t}$ }with $\C/\Z\omega_1 + \Z\omega_2$ via a map  
$$
\begin{array}{llll}
\ju{\mathfrak{q} :} &\C& \rightarrow &\overline{E}_t\\
& \omega &\mapsto& (\mathfrak{q}_1(\omega), \mathfrak{q}_2(\omega)),
\end{array}
 $$ 
 where $\mathfrak{q}_1, \mathfrak{q}_2$ are rational functions of $\mathfrak{p}$ and its derivative $d\mathfrak{p}/d\omega$, $\mathfrak{p}$ the Weirestrass function associated with the lattice $\Z\omega_1 + \Z\omega_2$ (cf. \cite[Section 3.2]{KurkRasch}). Therefore one can lift the functions $F^{1}(x,t)$ and $F^{2}(y,t)$ to functions $r_x(\omega) = F^{1}(\mathfrak{q}_1(\omega),t)$ and $r_y(\omega) = F^{2}(\mathfrak{q}_2(\omega),t)$, each defined on a suitable open subset of $\C$. 
\ju{We summarize our constructions with the following diagrams :
$$
  \xymatrix{ \mathfrak{q}^{-1}(V) \ar[r]^{\mathfrak{q}} \ar@/_1pc/[rrr]_{r_{x}} & V \ar[r]^{pr_{1}} & pr_{1}(V) \ar[r]^{F_{\mathcal{D}}^{1}(\cdot,t)} & \C}
$$
and
$$
  \xymatrix{ \mathfrak{q}^{-1}(V) \ar[r]^{\mathfrak{q}} \ar@/_1pc/[rrr]_{r_{y}} & V \ar[r]^{pr_{2}} & pr_{2}(V) \ar[r]^{F_{\mathcal{D}}^{2}(\cdot,t)} & \C},
$$
where $pr_{1}$ and $pr_{2}$ denote the projections on the first and second coordinates respectively. }
 
 Note that the involutions $\iota_1, \iota_2$ and the map $\tau$ can also be lifted to the universal cover $\C$ of $\overline{E}_t$. We shall abuse notation and again denote these functions as $\iota_1(\omega), {\iota}_2(\omega)$ and ${\tau}(\omega)$. Furthermore, we have that ${\tau}(\omega)= \omega + \omega_3$  on $\C$ where no nonzero integer multiple of $\omega_3$ belongs to the lattice $\Z\omega_1 + \Z\omega_2$ if $\tau$ has infinite order.  \ju{For any meromorphic function $f$ on a Riemann surface $X$ and for any automorphism $\sigma$ of $X$, we set $\sigma(f)=f \circ \sigma$. This notation is widely used in the theory of difference equations. The notation $\sigma^{*}(f)$ is also classical but will not be used. Be careful, if $\sigma,\beta$ are automorphisms of $X$, then $(\sigma \circ \beta)(f)=\beta (\sigma (f))$.}

\begin{rmk}
In \cite[Section 3.2]{KurkRasch}, explicit expressions for $\omega_{1},\omega_{2}$, and $\omega_{3}$ are given. Furthermore, $\omega_{1}$ is a purely imaginary number, whereas $\omega_{2}$ and $\omega_{3}$ are real numbers.
\end{rmk}

One can be deduce from \cite[Theorems 3 and 4]{KurkRasch}, that the functions $r_x(\omega)$ and $r_y(\omega)$ can be continued meromorphically as univalent functions on the universal cover $\C$.  Furthermore, for any $\omega \in \C$, we have
 \begin{eqnarray}
  {\tau}(r_x(\omega)) -r_x(\omega) & = & b_1, \mbox{ where } b_1 =\iota_1(\mathfrak{q}_2(\omega))(\tau(\mathfrak{q}_1(\omega)) - \mathfrak{q}_1(\omega)) \label{eq:KRThm1a}\\
  && \ \ \ \ \ \ \ \ \ \ \ \ \ \ \ \ \ \ju{=\tau(\mathfrak{q}_2(\omega))(\tau(\mathfrak{q}_1(\omega)) - \mathfrak{q}_1(\omega))} \nonumber \\
 {\tau}(r_y(\omega)) -r_y(\omega) & = & b_2, \mbox{ where } b_2 =\mathfrak{q}_1(\omega)({\iota}_1(\mathfrak{q}_2(\omega)) - \mathfrak{q}_2(\omega)) \label{eq:KRThm1}
 \\  && \ \ \ \ \ \ \ \ \ \ \ \ \ \ \ \ \ \ju{=\mathfrak{q}_1(\omega) (\tau(\mathfrak{q}_2(\omega)) - \mathfrak{q}_2(\omega))} \nonumber \\
r_x(\omega+\omega_1) & = &  r_x(\omega)\label{eq:KRThm2}\\
r_y(\omega+\omega_1) & = & r_y(\omega).\label{eq:KRThm3}
 \end{eqnarray}
 
 \begin{rmk}  1.~In the statement of \cite[Theorem 4]{KurkRasch} Kurkova and Raschel   give equations that are different from equations~\eqref{eq:KRThm1} and~\eqref{eq:KRThm1a}.  The above equations are presented in the proof of \cite[Theorem 4]{KurkRasch} and are shown to be equivalent to those in the statement of their theorem.  
 
 2.~The functions $r_x$ and $r_y$ are not $\omega_2$ periodic and therefore only define multivalued functions on $\bar{E}_t$.

 3.~Equation \eqref{eq:KRThm1a} gives meaning to the formula $$\toto{\tau(F^1(x,t)) - F^1(x,t) = \iota_1(y)(\tau(x) - x)\ju{=\tau(y)(\tau(x) - x)},}$$ and \eqref{eq:KRThm1} \ju{gives} meaning to the formula $$\toto{\tau(F^2(y,t)) - F^2(y,t) = x(\iota_1(y) - y)=\ju{x(\tau(y) - y)}.}$$ Although the right hand sides of these equations are well defined on $\Etproj$, further analysis is needed to give meaning to ``$\tau(F^1(x,t))$'' and ``$\tau(F^2(y,t))$''. This is one reason for lifting the various functions to the universal cover of this curve.
 \end{rmk}

\ju{From now on and until the end of the paper, we shall make the following assumptions. 
\begin{assumption}\label{assumption:generalcase t and D}
The walk is in $\walks$ and ${0<t< 1/|\mathcal{D}|}$ is such that~:
\begin{enumerate}
\item the group $G_t$ is infinite, 
\item $t$ is not algebraic over $\Q$. 
\end{enumerate}
\end{assumption}}

\ju{
Using the fact that the group of any walk in $\walks$ is infinite (see Section~\ref{sec:listofwalkssmallsteps}) and Proposition \ref{prop:groupwalkrestrEt}, we see that, once we have fixed a walk in $\walks$, the set of $t$ such that the assumptions~\ref{assumption:generalcase t and D} are satisfied has 
 denumerable complement in  $]0,1/|\mathcal{D}|[$.} 
 

\ju{
\begin{rem}\label{rem:remfundpropused}
Actually, the crucial properties used in the rest of the paper are the following~:  
\begin{enumerate}
\item the curve \ju{$\Etproj$} is an elliptic curve,
\item the group $G_t$ is infinite, 
\item $t$ is not algebraic over $\Q$, 
\item the functions $x \mapsto F^1(x,t)$ and $y \mapsto F^2(y,t)$, each analytic on some open subset of $\Etproj$, can be lifted and continued to functions $r_x$ and $r_y$ meromorphic on the universal cover of $\Etproj$ such that these functions satisfy equations \eqref{eq:KRThm1a},  \eqref{eq:KRThm1}, \eqref{eq:KRThm2}, and \eqref{eq:KRThm3}.
\end{enumerate}
These properties are automatically satisfied under the assumptions~\ref{assumption:generalcase t and D} (indeed, for (1) see Corollary \ref{cor:genuscurvewalk} and for (4) see the beginning of the present Section). 
\end{rem}
}

\section{Hypertranscendancy Criteria}\label{sec:hyper}

In this section, we derive hypertranscendency criteria for $F^{1}(x,t)$ and $F^{2}(y,t)$.  The related functions $r_x$ and $r_y$ satisfy  ``difference'' equations of the form ${\tau(Y) - Y = b}$. Galois theoretic methods to study the differential properties of such functions have been developed in \cite{HS} and \cite{DHR} (see also \cite{HAR16}).  In this section we describe a consequence of this latter theory and how it will be used to show that \ju{the walks in $\walksg$ have hypertranscendental generating series}.  We will begin by making precise the differential situation.

\subsection{A derivation on $\Etproj$ commuting with $\tau$}

We assume that $\Etproj$ is an elliptic curve. We denote its function field by $\C(\Etproj)$.  The space of meromorphic differential forms on $\Etproj$ forms a one dimensional vector space over $\C(\Etproj)$ and the space of regular differential forms is a one dimensional vector space over $\C$.   We let  $\Omega$ be a nonzero regular differential form on $\Etproj$. \ju{In the following Lemma, we prove the existence of a derivation $\delta$ on $\Etproj$ commuting with the automorphism $\tau$ of $\Etproj$.}

\begin{lemma}\label{lem:derivellcurvecommtau}
The derivation $\delta$ of $\C(\Etproj)$ such that $d(f)=\delta(f) \Omega$ commutes with $\tau$, that is  
$$
\tau \circ \delta =\delta \circ \tau.
$$
\end{lemma}

\begin{proof}
According to \cite[Lemma 2.5.1, Proposition 2.5.2]{DuistQRT} or to \cite[Proposition III.5.1]{Silverman}, we have $\tau^*(\Omega)=\Omega$, {\mfs where $\tau^*$ is the map induced by $\tau$ on the space of  regular differential forms on $\Etproj$}. 
It follows that, for any  $f$ in $\C(\Etproj)$, we have 
$$
\delta(\tau(f)) \Omega = d(\tau(f))= \tau^*(df)= \tau^{*}(\delta(f)\Omega)= \tau(\delta(f))\tau^*(\Omega)=\tau(\delta(f))\Omega.
$$
Whence the equality $\delta(\tau(f))=\tau(\delta(f))$. 
\end{proof}

\begin{lemma}\label{lemma:derivationvalutaion}
Let $P \in \Etproj$ and let $v_P$ be the associated valuation on $\C(\Etproj)$. Then, for any $f \in \C(\Etproj)$, we have \begin{itemize}
\item if $v_P(f) \geq 0$ then $v_P( \delta(f)) \geq  0$;
\item if $v_P(f)<0$ then $v_P( \delta(f))=  v_P(f)-1$.
\end{itemize}
\end{lemma}

\begin{proof}
We recall that
$\omega$ has valuation $0$ at any point of $\Etproj$ \ju{(see \cite[Proposition III.1.5]{Silverman};  see \cite[\S~II.4]{Silverman} for the definition of the valuation of differential forms)}. Let $u$ be a local parameter of $\Etproj$ at $P$. 
Since $d(u)=\delta(u) \Omega$ and since both $du$ and $\Omega$ have valuation $0$ at $P$, we get that $v_P(\delta(u))=0$. \ju{The result follows clearly from this and from the fact that, for $f= \sum_{i=v_u(f)}^{+\infty} a_i u^i \in \C(\Etproj)$, we have ${\delta(f)=\sum_{i=v_u(f)}^{+\infty} a_{i}i\delta(u)u^{i-1}}$.}
\end{proof}

\begin{rmk}\label{rmk:derivation} In Section~\ref{subsec:func equ elliptic curve}, we discussed the universal covering space map ${\C \rightarrow \Etproj}$ and \ju{used} it to lift functions on $\Etproj$ to $\C$.  In particular the elements of $\C(\Etproj)$ lift to doubly periodic meromorphic functions on $\C$ and so we can consider $\C(\Etproj) \subset \calM(\C)$ where $\calM(\C)$ is the field of meromorphic functions on  $\C$.  One then sees that $\tau$ corresponds to the map $\omega \mapsto \omega + \omega_3$, $\Omega$ corresponds  (up to constant multiple) to the  regular differential form $d\omega$ and the $\tau$-invariant derivation $\delta$ corresponds to  $\frac{d}{d\omega}$.  
\end{rmk}

\subsection{Hypertranscendancy criteria}\label{sec:gene hyp crit} We will reduce questions concerning the hypertranscendence of $F^{1}(x,t)$ and $F^{2}(y,t)$ to questions about the differential behavior of elements of 
$\C(\Etproj)$. Criteria derived by using the {Galois theory of difference equations} allow us to do this. 
We will start with an abstract formulation but quickly specialize to the present situation.

\begin{defi}\label{def:tauconstselliptic}
A {\it $\delta\tau$-field} is a triple $(K,\delta, \tau)$ where $K$ is a field, $\delta$ is a derivation on $K$, $\tau$ is an automorphism of $K$ and $\delta$ and $\tau$ commute on $K$. The {\rm $\tau$-constants $K^\tau$} of $K$ is the set $\{c \in K \ | \ \tau(c) = c\}$.
\end{defi}  

The triples $(\C(\Etproj), \delta ,\tau)$ and  $(\calM(\C),  \frac{d}{d\omega},\tau: \omega\rightarrow \omega+\omega_3)$ are examples \ju{(see Lemma \ref{lem:derivellcurvecommtau} and Remark~\ref{rmk:derivation})}. We say that a $\delta\tau$-field is a subfield of another $\delta\tau$-field if the derivation and automorphism of the smaller field are just the restrictions of the derivation and automorphism of the larger field.  If $\tau$ has infinite order as an automorphism of $\C(\Etproj)$, then $\C(\Etproj)^\tau = \C$. This follows from the fact that $\tau(f(X)) = f(X \oplus P)$ where $P$ is a point of infinite order. If $\tau(f) = f$, then $f(Q) -f(Q\oplus nP)=0$ for all $n\in \Z$ and $Q$ a regular point of $f$.  This implies that $f(X) = f(Q) $ on a \ju{Zariski-}dense subset of $\Etproj$ and so $f(X)$ must be constant on $\Etproj$. 

 The following formalizes the notion of  holonomic, hyperalgebraic and hypertranscendental.  

\begin{defi} Let $(E,\delta)\subset (F,\delta)$ be $\delta$-fields 
We say that $f\in F$ is {\rm  hyperalgebraic} over $E$ if it satisfies a non trivial algebraic differential equation with coefficients in $E$, {\it i.e.}, if for some $m$ there exists a nonzero polynomial ${P(y_0, \ldots , y_m) \in E[y_0, \ldots , y_m]}$ such that
$$
P(f,\delta(f), \ldots, \delta^m(f)) = 0.
$$
We say that $f$ is {\rm honolomic} over $E$ if in addition, the equation is linear. We say that  $f$  is {\rm hypertranscendental} over $E$ if it is not hyperalgebraic.  
\end{defi}
Other terms have been used for the above concepts: hypotranscendental or differentially algebraic or $\delta$-algebraic for  hyperalgebraic and differentially transcendental or transcendentally transcendental for hypertranscendental.

Proposition~2.6 of \cite{DHR} gives criteria for hypertranscendence in this general setting.  Here, we only state  this result  in our situation.   \ju{As noted above, in Remark~\ref{rmk:derivation},} we may consider $(\C(\Etproj), \delta ,\tau)$ as a subfield of $(\calM(\C),  \frac{d}{d\omega},\tau: \omega\rightarrow \omega+\omega_3)$. Given $f \in \calM(\C)$ we denote by $\C(\Etproj)<f>_{\delta\tau}$ the smallest subfield of $\calM(\C)$ containing $\C(\Etproj)$ and $\{ \tau^i(\delta^j(f)) \ | \ i\in \Z ,j \in \Z_{\geq 0}\}$.  Note that this is a $\delta\tau$-field.

\begin{prop}\label{prop:caract hyperalg} Let $b\in \C(\Etproj)$ and $f \in \calM(\C)$ and assume that 
\[\tau(f) - f = b.\]
If $f$ is hyperalgebraic over $\C(\Etproj)$, then, there exist an integer $n \geq 0$, ${c_0,\ldots,c_{n-1} \in \C}$ and $g \in \C(\Etproj)$ such that 
 \begin{equation}\label{eq:lindiffb}
  \delta^n(b)+c_{n-1}\delta^{n-1}(b)+\cdots+c_{1}\delta(b)+c_{0} b = \tau(g)-g. 
 \end{equation}
Conversely, if $b$ satisfies such an equality and if $(\C(\Etproj)<f>_{\delta\tau})^\tau = \C$, then $f$ is holonomic over $\C(\Etproj)$.
\end{prop}

\ju{For the convenience of the reader, Appendix~\ref{diffgalois} contains a brief introduction to the Galois theory of difference equations and a self-contained proof of Proposition~\ref{prop:caract hyperalg}.}

The condition $(\C(\Etproj)<f>_{\delta\tau})^\tau = \C$ is not superfluous.  We will reconfirm \ju{(cf. \cite[Theorem~11]{BBMR16})}  in Section~\ref{sec:hyperalg} that for the nine exceptional \ju{walks in $\walkse$}, specializations of the  generating series are hyperalgebraic but we know that they are nonetheless not holonomic.  This situation arises because when one tries to apply the last part of Proposition~\ref{prop:caract hyperalg} one is forced to add new $\tau$-constants. 
{\mfs This is precisely the situation that occurs in the proof of Proposition~\ref{prop:hypalg}}

In the appendix, we give necessary and sufficient conditions on the poles of $b$ to guarantee the existence of an equation of the form of \eqref{eq:lindiffb}.  \ju{For instance, the following result is a particular case of Corollary~\ref{cor1} from the Appendix}.

\begin{cor}\label{cor:prop:caract hyperalg}
 Assume that $b$ has a pole $P \in \Etproj$ of order $m \geq 1$ such that none of the $\tau^k(P)$ with $k \in \Z \setminus \{0\}$ is a pole of order $\geq m$ of $b$. Then, $f$ is hypertranscendental.  
\end{cor}

In Sections~\ref{sec:hypertrF1F2} and ~\ref{sec:hyperalg}, we will verify conditions like this to show that various generating series are hypertranscendental.

\subsection{Applications to $F^{1}(x,t)$ and $F^{2}(y,t)$}\label{sec:applichyptrF1F2}

\ju{We assume that assumptions \ref{assumption:generalcase t and D} are satisfied.} We shall now apply the results of Section \ref{sec:gene hyp crit} to $F^{1}(x,t)$ and $F^{2}(y,t)$.  

We begin by considering $F^{1}(x,t)$ as a formal power series with coefficients in $\C$.  We wish to show that $F^{1}(x,t)$ does not satisfy a polynomial differential equation 

\[P(x,F^{1}(x,t), \frac{dF^{1}(x,t)}{dx}, \ldots , \frac{d^nF^{1}(x,t)}{dx^n}) = 0,\]
where $P \in \C[x, Y_0, \ldots , Y_n]$.  Let us assume that such an equation existed.  The derivation $\frac{d}{dx}$ extends uniquely to a derivation on $\C(\Etproj)$ which we again denote by $\frac{d}{dx}$. \ju{So, we have the following commutative diagram:
$$
\xymatrix{
    \C(\Etproj) \ar[r]^{\frac{d}{dx}}  & \C(\Etproj)  \\
    \C(x) \ar[r]_{\frac{d}{dx}} \ar[u]^{(pr_{1})_{*}} & \C(x) \ar[u]_{(pr_{1})_{*}}
  }
$$
where $(pr_{1})_{*}$ is the map induced on the function fields by the first projection $pr_{1} : \Etproj \rightarrow \P1(\C)$}. The derivations on $\C(\Etproj)$ form a one dimensional vector space over $\C(\Etproj)$.  Therefore  $\delta$ on  $\Etproj$ can be written as $\delta = h \frac{d}{dx}$ for some $h \in \C(\Etproj)$.  In particular this implies that when we consider $F^1(x,t)$ as an analytic function on an open set of $\Etproj$ it will satisfy a polynomial differential equation $\tilde{P}(F^{1}(x,t), \delta(F^{1}(x,t)), \ldots , \delta^n(F^{1}(x,t)))= 0$ where $\tilde{P}$ has coefficients in $\C(\Etproj)$. When we  lift this equation to the universal cover, we see that $r_x \in \calM(\C)$ is hyperalgebraic over $\C(\Etproj)$. By assumption~\ref{assumption:generalcase t and D} \ju{and Section \ref{subsec:func equ elliptic curve}}, $r_x$ satisfies $\tau(r_x) - r_x = b_1$ where $b_1 \in \C(\Etproj)$.  We can therefore apply Proposition~\ref{prop:caract hyperalg} and conclude that there exist an integer $n \geq 0$, $c_0,\ldots,c_{n-1} \in \C$ and $g \in \C(\Etproj)$ such that 
\begin{eqnarray}\label{eq:F1hyper}
\delta^n(b_1)+c_{n-1}\delta^{n-1}(b_1)+\cdots+c_{1}\delta(b_1)+c_{0} b_1 = \tau(g)-g.
\end{eqnarray}  Therefore to show $F^{1}(x,t)$ is  hypertranscendental, it is enough to show that such an equation does not exist.   Notice that this last condition only involves elements in $\C(\Etproj)$.  Similar reasoning (where we replace $\frac{d}{dx}$ with $\frac{d}{dy}$)  shows that $F^{2}(y,t)$ is hypertranscendental over $\C(\Etproj)$ if there is no relation such as \eqref{eq:F1hyper} with $b_1$ replaced by $b_2$. Therefore we have

\begin{prop}\label{prop:caract hyperalg F1}
Let $i \in \{1,2\}$. Assume that there does not  exist an integer $n \geq 0$, $c_0,\ldots,c_{n-1} \in \C$ and $g \in \C(\Etproj)$ such that 
 \begin{equation}\label{eq:eqbiettau}
 \delta^n(b_{i})+c_{n-1}\delta^{n-1}(b_{i})+\cdots+c_{1}\delta(b_{i})+c_{0} b_{i} = \tau(g)-g. 
 \end{equation}
 Then, the function $F^{i}$ is hypertranscendental  over $\C(\Etproj)$. 
\end{prop}
 Once again, using results from the appendix \ju{(namely Corollary \ref{cor1})}, we have
\begin{cor}\label{coro: hyperalg F}
Let $i \in \{1,2\}$. Assume that $b_{i}$ has a pole $P \in \Etproj$ of order $m \geq 1$ such that none of the $\tau^k(P)$ with $k \in \Z \setminus \{0\}$ is a pole of order $\geq m$ of $b_{i}$. Then, $F^{i}$ is hypertranscendental over $\C(\Etproj)$.  
\end{cor}

The following additional result will also  be useful. 

\begin{prop}\label{prop:caract hyperalg F1 and F2}
The function $F^{1}(x,t)$ is hypertranscendental over $\C(\Etproj)$ if and only if $F^{2}(y,t)$ is hypertranscendental  over $\C(\Etproj)$. 
\end{prop}

\begin{proof}
We note that $b_{1}+b_{2}=\tau(xy)-xy$. So, for all $k \in \Z_{\geq 0}$, 
\begin{eqnarray}
\delta^{k}(b_{2})&=&-\delta^{k}(b_{1})+\tau(\delta^k(xy))-\delta^k(xy).\label{eqn:L1L2equiv}
\end{eqnarray}
It follows that an equation of the form (\ref{eq:eqbiettau}) holds true for $i=1$ if and only if such an equation holds true for $i=2$. 
\end{proof}

\section{Preliminary results on the elliptic curve $\Etproj$}\label{sec:prelim}
In this and the next section we will show \ju{that all walks in $\walksg$ have} generating series that are hypertranscendental.  As we have indicated, an examination of the poles of the $b_i, i=1,2$, and of the orbits of these poles under $\tau$ will yield these results.  

To get a sense of our techniques, we will outline  in the following example how we show that $F^2(y,t)$ is hypertranscendental for  the  walk \ju{$\walk{IA.1}$}.

\begin{exa}For the walk \ju{$\walk{IA.1}$}, we have $d_{-1,1} = d_{1,1} = d_{1,-1} = d_{0,-1} = 1$ and all other $d_{i,j} = 0$. The curve $\Etproj$ is defined by
\[ \overline{K}(x_0,x_1, y_0, y_1, t) = x_0x_1y_0y_1 - t(x_1^2y_0^2 + x_0^2y_0^2 + x_0^2 y_1^2+x_0x_1y_1^2).\]

 Recalling that $x = x_0/x_1$ and $y = y_0/y_1$ one sees that the poles of $b_2 = x(\iota_1(y) - y)$ are among the poles of $x$, $y$, and $\iota_1(y)$, that is, among the points \[S_2 = \{P_1, P_2, Q_1, Q_2 , \iota_1(Q_1), \iota_1(Q_2)\} = \{P_1, P_2, Q_1, Q_2 , \tau^{-1}(Q_1), \tau^{-1}(Q_2)\}\] where
 \begin{eqnarray*}
  P_1 = ([1:0],[\sqrt{-1}: 1]), & \ \ \ & P_2 =  ([1:0],[-\sqrt{-1}, 1]),\\
  Q_1 = ([1:\sqrt{-1}],[1:0]), & \ \ \  & Q_2 = ([1:\ju{-}\sqrt{-1}],[1:0]).
  \end{eqnarray*}
  One sees that $b_2$ has a pole at $P_1$.  We claim that $b_2$ has no pole of the form $\tau^k(P_1)$ with $k \in \Z\backslash \{0\}$.  Once this is shown the result will follow from Corollary~\ref{coro: hyperalg F}. Therefore we must show that $\tau^k(P_1) \neq P_2$ for any $k \in \Z \backslash \{0\}$ and $\tau^k\ju{(P_{1})} \neq Q_1, Q_2$ for any $k \in \Z$.  
  
  To verify this, first note that $P_1, P_2, Q_1,Q_2$  all have coordinates in ${L = \Q(t,\sqrt{-1})}$. Let $\sigma$ be the automorphism of $L|\Q(t)$ defined by ${\sigma(\sqrt{-1}) = - \sqrt{-1}}$. A calculation shows that for any point $Q \in \Etproj(L)$ one has that ${\iota_1(Q), \iota_2(Q), \tau(Q) \in \Etproj(L)}$ and that $\iota_k\circ \sigma = \sigma \circ\iota_k, k = 1,2,$ and $\tau\sigma = \sigma \tau$ (cf. Proposition~\ref{prop:galoisactionpointsellipticcurve}). Therefore all points in $S_2$ have coordinates in $L$. Furthermore $P_2 = \sigma(P_1)  $ and $Q_2 = \sigma(Q_1)$.
  
If $P_2 = \tau^k(P_1)$, then $P_1 = \sigma(P_2) = \sigma(\tau^k(P_1)) = \tau^k(\sigma(P_1)) = \tau^k(P_2)=\tau^{2k}(P_1)$.  Since $\tau$ corresponds to addition by a point of infinite order, we get a contradiction if $k \neq 0$.

If $Q_1 = \tau^k(P_1)$ for some $k \in \Z$, then $\tau^k(P_2) = \sigma(\tau^k(P_1)) = \sigma(Q_1) = Q_2$. Since $P_1 = \iota_1(P_2)$ and  $\tau = \iota_2\iota_1$, we have that 
$$
\tau^{-k+1}(P_1) = \tau^{-k+1} \iota_1(P_2) = \iota_2\tau^k(P_2) = Q_1 = \tau^k(P_1).
$$ 
\ju{We have used the fact that $\iota_{1}\tau \iota_{1}=\tau^{-1}$.} 
Therefore $\tau^{-2k+1}(P_1) = P_1$, again a contradiction, since $k \in \Z$. The proof that $Q_2 \neq \tau^k(P_1)$ for any $k \in \Z$ is similar.
\end{exa}
A similar proof combining  arithmetic ({\it e.g.}, Galois theory) with  geometry (the behavior of points under $\iota_1,\iota_2, \tau$) shows that $b_1$ or $b_2$ for  the walks \ju{$\walk{IA.*}$},{\footnote{\toto {Notation such as $\walk{IA.*}$ refers to all the walks $\walk{IA.1}, \ldots,\walk{IA.9}$.}} } \ju{$\walk{IB.*}$}, \ju{$\walk{IC.*}$}, and \ju{$\walk{IIA.*}$} has a pole that is unique in its $\tau$-orbit.  In the remaining cases, the arguments become more complicated. The poles lie in $\Q(t)$ and so we do not have a field automorphism at our disposal.  In addition we may need to look more carefully at which poles lie in which orbits\footnote{We note that given  points $Q_1$ and $Q_2$ on an elliptic curve $\Etproj$, a general procedure is given in \cite{MA88} to determine if there is an integer $n$ such that $Q_1 = Q_2 \oplus nP$ where $\tau(Q) = Q\oplus P$. Our more elementary, direct  approach is independent of  \cite{MA88}.} and consider their orders as well as the expansions at these poles.  Nonetheless, the interplay of arithmetic and geometry will be the key to the arguments. 

In this section we examine points that are possible poles of $b_1$ and $b_2$ and develop the properties needed in the next sections.    In Section~\ref{sec:hypertrF1F2} we will apply these together with results from the appendix and Proposition~\ref{prop:caract hyperalg F1} to conclude hypertranscendence \ju{for any walk in $\walksg$}.   

As before, we suppose that the \ju{assumptions~\ref{assumption:generalcase t and D} are satisfied}.
 
\subsection{On the base points}

\begin{defi}
The points $P=([x_0:x_1], [y_0:y_1])$ of $\Etproj$ such that $x_0x_1y_0y_1=0$ will be called the base points of $\Etproj$.
\end{defi}

\begin{rem}
This terminology comes from the fact that these points are the base points of a natural pencil of elliptic curves.
\end{rem}

Let us recall that the notation $[a:b] \in \P1(\C)$ represents a ray of points of the form ${ \{ (\alpha a, \alpha b) \ | \ 0\neq \alpha \in \C\}}$.  Let $\L$ be a subfield of $\C$. We say that $[a:b] \in \P1(\L)$ if there is some element of this ray with coordinates in $\L$.  A similar notation concerns points in $\Etproj(\L)$, the elements of which will be called the $\L$-points of $\Etproj$.

\begin{lemma}\label{lem:base points are alg}
Any base point of $\Etproj$ belongs to $\Etproj(\Qbar)$.
\end{lemma}

\begin{proof}
Let $P=([x_0:x_1], [y_0:y_1])$ be a base point of $\Etproj$. 
Let us assume for instance that $x_{0}=0$ (the other cases being similar). Then, we obviously have ${[x_0:x_1]=[0:1] \in \P1(\Qbar)}$. Moreover, we have
\begin{multline}
\overline{K}(x_0,x_1,y_0,y_1,t)= x_0x_1y_0y_1 -t \sum_{i,j=0}^2 d_{i-1,j-1} x_0^i x_1^{2-i}y_0^j y_1^{2-j} \\
=  -t x_{1}^{2} \sum_{j=0}^2 d_{-1,j-1} y_0^j y_1^{2-j} \nonumber
\end{multline}
which is equal to $0$ if and only if $\sum_{j=0}^2 d_{-1,j-1} y_0^j y_1^{2-j}=0$. \ju{Since at least one of the $d_{-1,j-1}$ is nonzero, we get}  ${[y_0:y_1] \in \P1(\Qbar)}$. 
\end{proof}

\begin{lemma}\label{lem:fixedbasepointsinvolution}
The following arrays describe precisely what are the possible base points fixed by $\iota_{1}$ or $\iota_{2}$, the corresponding conditions on the $d_{i,j}$ \ju{and the walks in $\walks$ satisfying these conditions}:
{\small
$$ \begin{array}{ |c|c|c|}
\hline
\mbox{Points} & \multicolumn{2}{|c|}{\mbox{ Fixed by } \iota_1}  \\ \hline
 x_0=0 & ([0:1],[1:0])\mbox{ iff } d_{-1,0}=d_{-1,1}=0 & ([0:1],[0:1]) \mbox{ iff } d_{-1,0}=d_{-1,-1}=0 \\ 
  & \ju{ \walk{IIA.1}, \walk{IIA.4}, \walk{IIA.5}, \walk{IIB.1}, \walk{IIB.2}, \walk{IIB.6}, \walk{IIC.3} } & \ju{ \walk{IA.1}, \walk{IA.2}, \walk{IC.2}, \walk{IIB.7}} \\
 \hline
 y_0=0 &  ( [1:0],[0:1])\mbox{ iff } d_{1,0}=d_{1,-1}=0 & ([0:1],[0:1]) \mbox{ iff } d_{-1,-1}=d_{-1,0}=0 \\ 
   & \ju{ \walk{IB.*}, \walk{IIC.1}, \walk{IIC.2}, \walk{III} }  &  \ju{ \walk{IA.1}, \walk{IA.2}, \walk{IC.2}, \walk{IIB.7},  } \\
 \hline
x_1=0 & ([1:0],[1:0]) \mbox{ iff } d_{1,0}=d_{1,1}=0  & ([1:0],[0:1]) \mbox{ iff } d_{1,0}=d_{1,-1}=0 \\ 
   & \ju{ \walk{IID.*} } &  \ju{\walk{IB.*}, \walk{IIC.1}, \walk{IIC.2}, \walk{III} }   \\
\hline
y_1=0 &([1:0],[1:0]) \mbox{ iff } d_{1,0}=d_{1,1}=0  & ([0:1],[1:0]) \mbox{ iff } d_{-1,0}=d_{-1,1}=0 \\ 
& \ju{\walk{IID.*} } & \ju{ \walk{IIA.1}, \walk{IIA.4}, \walk{IIA.5}, \walk{IIB.1}, \walk{IIB.2}, \walk{IIB.6}, \walk{IIC.3}} \\
\hline
\end{array}
$$
$$
\begin{array}{ |c|c|c|}
\hline
\mbox{Points} &   \multicolumn{2}{|c|}{\mbox{ Fixed by } \iota_2}  \\ \hline
x_0=0 &  ( [0:1],[1:0])\mbox{ iff } d_{0,1}=d_{-1,1}=0 & ([0:1],[0:1]) \mbox{ iff } d_{-1,-1}=d_{0,-1}=0  \\ 
& \ju{\walk{III}}  & \ju{ \walk{IA.3}, \walk{IA.5}, \walk{IIA.2}, \walk{IIB.3}, \walk{IID.4} } \\
\hline
y_0 =0 &  ([1:0],[0:1])\mbox{ iff } d_{0,-1}=d_{1,-1}=0 & ([0:1],[0:1]) \mbox{ iff } d_{0,-1}=d_{-1,-1}=0 \\ 
&\ju{ \walk{IC.1}, \walk{IIC.3}, \walk{IIC.4} } & \ju{ \walk{IA.3}, \walk{IA.5}, \walk{IIA.2}, \walk{IIB.3}, \walk{IID.4} } \\
\hline
x_1=0 & ([1:0],[1:0]) \mbox{ iff } d_{0,1}=d_{1,1}=0  & ([1:0],[0:1]) \mbox{ iff } d_{0,-1}=d_{1,-1}=0 \\ 
& \ju{none}  & \ju{ \walk{IC.1}, \walk{IIC.3}, \walk{IIC.4}}  \\
\hline
y_1=0 & ([1:0],[1:0]) \mbox{ iff } d_{0,1}=d_{1,1}=0  & ([0:1],[1:0]) \mbox{ iff } d_{0,1}=d_{-1,1}=0 \\ 
& \ju{none}  & \ju{ \walk{III}}  \\
\hline
\end{array}
$$}

\end{lemma}
\begin{proof}
Taking into consideration obvious symmetries, we see that it is sufficient to prove the Lemma for a base point of the form $P=([0:1],[\beta_0:\beta_1])$. 

Assume that $P$ is fixed by $\iota_1$. By Lemma \ref{lemma:fixedpointinvolution}, this is equivalent to $$\Delta^x_{[0:1]}/t^2=d_{-1,0}^2 -4 d_{-1,1}d_{-1,-1}=0.$$ Since the $d_{i,j}$ belong to $\{ 0,1\}$, the latter condition is equivalent to the equality $d_{-1,0}=0=d_{-1,-1}d_{-1,1}$. 

If $d_{-1,0}=d_{-1,-1}=0$, then we have $d_{-1,1} \neq 0$ and the fact that $P$ belongs to $\Etproj$ simply means that $d_{-1,1}\beta_{1}^{2}=0$. Therefore, we have $P=([0:1],[1:0])$. 

If $d_{-1,0}=d_{-1,1}=0$, then we have $d_{-1,-1} \neq 0$ and the fact that $P$ belongs to $\Etproj$ simply means that $d_{-1,-1}\beta_{0}^{2}=0$. Therefore, we have $P=([0:1],[0:1])$. 

This is precisely the first line of the first array.

Assume now that $P$ is fixed by $\iota_2$. Since $P$ and $\iota_{2}(P)$ belong to the curve $E_{t}$, the $x$-coordinates of both $P$ and $\iota_2(P)$ must satisfy the homogeneous equation of degree $2$ in $x_{0}$ and $x_{1}$ given by 
$x_0^2 A + B x_0x_1 +C x_1^2=0$
with  
\begin{eqnarray*}
A &=& d_{1,-1} \beta_1^2 +d_{1,0} \beta_0\beta_1 +d_{1,1} \beta_0^2,\\
 B  & = &-\beta_0\beta_1/ t  +(d_{0,-1}\beta_1^2 +d_{0,1} \beta_0^2),  \\
C  &=&  d_{-1,-1} \beta_1^2 +d_{-1,0} \beta_0\beta_1 +d_{-1,1} \beta_0^2.
\end{eqnarray*}
Since the $x$-coordinate of $P$ is equal to $[0:1]$, we see that $C=0$. Moreover, the fact that $\iota_{2}(P)=P$ ensures that $[0:1]$ is the only solution in $\P1(\C)$ of the homogeneous equation $x_{0}^{2}A+Bx_{0}x_{1}+Cx_{1}^{2}=x_{0}^{2}A+Bx_{0}x_{1}=0$. This ensures that $B=0$. 
Thus, we have obtained the equalities 
\begin{eqnarray}
B  = -\beta_0\beta_1/ t  +(d_{0,-1}\beta_1^2 +d_{0,1} \beta_0^2) &  = &  0 \label{eq:eqBeq0}\\
C  =  d_{-1,-1} \beta_1^2 +d_{-1,0} \beta_0\beta_1 +d_{-1,1} \beta_0^2 &=& 0.
\end{eqnarray}
But, according to Lemma \ref{lem:base points are alg}, $[\beta_0:\beta_1]$ belongs to $\P1(\Qbar)$ and, by hypothesis, $t$ is transcendental. Therefore, equation \eqref{eq:eqBeq0} is equivalent to $$\beta_0\beta_1 =0= d_{0,-1}\beta_1^2 +d_{0,1} \beta_0^2,$$ {\it i.e.}, $\beta_{0}=0$ and $d_{0,-1}=0$ or $\beta_{1}=0$ and $d_{0,1}=0$. 

If $\beta_{0}=0$ and $d_{0,-1}=0$, then $P=([0:1],[0:1])$. This point belongs to $\Etproj$ if and only if $d_{-1,-1}=0$. 

If $\beta_{1}=0$ and $d_{0,1}=0$, then $P=([0:1],[1:0])$. This point belongs to $\Etproj$ if and only if $d_{-1,1}=0$. 

This gives the first line of the second array.
 \end{proof}

\begin{lemma}\label{lemma:fixedpointinvolutionnotrational}
The following properties hold true : 
\begin{itemize}
\item if $\iota_1(P)=P$ then $P=([0:1],[0:1])$ and  $d_{-1,0}=d_{-1,-1}=0$; \ju{the last condition corresponds to the walks $\walk{IA.1}, \walk{IA.2}, \walk{IC.2}, \walk{IIB.7}$};
\item if $\iota_2(P)=P$ then $P=([0:1],[0:1])$ and $d_{0,-1}=d_{-1,-1}=0$; \ju{the last condition corresponds to the walks} {\mfs $\walk{IA.3},\walk{IA.5},\walk{IIA.2},\walk{IB.3},\walk{IID.4}$.}
\end{itemize}
\end{lemma}

\begin{proof}
Taking into consideration obvious symmetries, we see that it is sufficient to prove the first statement. Since $t$ is transcendental, we can and will identify $\Q(t)$ with a field of rational functions. Assume that $P=([\alpha_0:1],[\beta_0:1]) \in \Etproj(\Q(t))$ is fixed by $\iota_1$. By Lemma \ref{lemma:fixedpointinvolution}, we must have $\Delta^x_{[\alpha_0:1]}=0$, {\it i.e.},
\begin{equation}\label{eq:alpha0}
(d_{-1,0} -\frac{\alpha_0}{t} +d_{1,0} \alpha_0^2)^2= 4(d_{-1,1} +d_{0,1} \alpha_0 +d_{1,1} \alpha_0^2)(d_{-1,-1} +d_{0,-1} \alpha_0 +d_{1,-1}\alpha_0^2).
\end{equation}

\vskip 5 pt
\begin{center}
Step $1$: Case $\alpha_0 \in \Q$.
\end{center}

If $\alpha_0 \in \Q$, then, comparing the $t$-adic valuations of both sides of \eqref{eq:alpha0}, we get $\alpha_0=0$, {\it i.e.}, $P=([0:1],[\beta_0:1])$. In this case, equation \eqref{eq:alpha0} is simply $d_{-1,0}^2=4d_{-1,1}d_{-1,-1}$. Since $d_{i,j} \in \{0,1\}$, we get $d_{-1,0}=0=d_{-1,1}d_{-1,-1}$. This leads us to consider the two cases $d_{-1,0}=d_{-1,-1}=0$ and $d_{-1,0}=d_{-1,1}=0$. 

If $d_{-1,0}=d_{-1,-1}=0$, then we have $d_{-1,1} \neq 0$ and the fact that $P$ belongs to $\Etproj$ simply means that $d_{-1,1}\beta_{0}^{2}=0$. Therefore, we have $P=([0:1],[0:1])$, as desired.

If $d_{-1,0}=d_{-1,1}=0$, the fact that $P$ belongs to $\Etproj$ simply means that $d_{-1,-1}=0$. But for every walks under consideration, $d_{-1,-1}=d_{-1,0}=d_{-1,1}=0$ never occur \ju{so this is impossible}.

\vskip 5 pt

We shall now prove that  $\alpha_0 \in \Q(t) \setminus \Q$ is impossible. A key fact that is used several times in the proof is that, \ju{for any walks in $\walks$}, each set of steps contain elements lying on both sides of each of the lines: ${i=0, i+j=0, j=0, i-j =0}$. For example, the following condition ${d_{-1,0}=d_{-1,1} = d_{0,1} =0}$ is never realized since this condition would imply all steps lie on or below the line $i-j=0$.
We argue by contradiction. Equation \eqref{eq:alpha0} ensures that $\alpha_0$ must have either a pole
or a zero at $t=0$. Indeed, otherwise, the left hand side of \eqref{eq:alpha0} would have a pole at $t=0$ but not the right hand side. 

Let us write $\alpha_0=t^l P(t)$ with $l \in \Z^*$ and $P(t) \in \Q(t)$ without zero and pole at $t=0$.

\vskip 5 pt
\begin{center}
Step $2$: Case $P(t)$ constant.
\end{center}

 If $P(t) =c \in \Q^*$ then \eqref{eq:alpha0} becomes 
\begin{multline}\label{eq:alpha0monomial}
(d_{-1,0} -ct^{l-1} +d_{1,0}c^2t^{2l})^2=\\ 4(d_{-1,1} +d_{0,1}ct^l +d_{1,1}c^2t^{2l})(d_{-1,-1,} +d_{0,-1} ct^l +d_{1,-1}c^2t^{2l}).
\end{multline}
If $l < 0$, then, equating the coefficients of $t^{0}$ in \eqref{eq:alpha0monomial}, we find the equality ${d_{-1,0}^2=4d_{-1,1}d_{-1,-1}}$. This gives $d_{-1,0}=0$ and $d_{-1,1}d_{-1,-1}=0$. 

If $d_{-1,0}=0$ and $d_{-1,1}=0$, then $d_{-1,-1} \neq 0$. Moreover, equation \eqref{eq:alpha0monomial} simplifies as follows. 
\begin{equation}\label{eq:alpha0monomialb}
(-ct^{l-1} +d_{1,0}c^2t^{2l})^2= 4(d_{0,1}ct^l +d_{1,1}c^2t^{2l})(d_{-1,-1,} +d_{0,-1} ct^l +d_{1,-1}c^2t^{2l}).
\end{equation}
On the right hand side of \eqref{eq:alpha0monomialb}, we find the monomial $4cd_{0,1}d_{-1,-1}t^{l}$, but the non trivial monomials appearing on the left hand side have degree $2l-2$, $3l-1$ or $4l$ and none of them is equal to $l$ since $l<0$. Therefore, we must have $d_{0,1}d_{-1,-1}=0$ and, hence, $d_{0,1}=0$. \ju{However, the condition $d_{-1,0}=d_{-1,1}=d_{0,1}= 0$ is never realized  so this is impossible}. 

If $d_{-1,0}=0$ and $d_{-1,-1}=0$, then $d_{-1,1} \neq 0$. Moreover, equation \eqref{eq:alpha0monomial} simplifies as follows. 
\begin{equation}\label{eq:alpha0monomialbb}
(-ct^{l-1} +d_{1,0}c^2t^{2l})^2= 4(d_{-1,1} +d_{0,1}ct^l +d_{1,1}c^2t^{2l})(d_{0,-1} ct^l +d_{1,-1}c^2t^{2l}).
\end{equation}
On the right hand side of \eqref{eq:alpha0monomialbb}, we find the monomial $4cd_{-1,1}d_{0,-1}t^{l}$, but the non trivial monomials appearing on the left hand side have degree $2l-2$, $3l-1$ or $4l$ and none of them is equal to $l$ since $l<0$. Therefore, we must have $d_{-1,1}d_{0,-1}=0$ and, hence, $d_{0,-1}=0$. \ju{However, the condition $d_{-1,0}=d_{-1,-1}=d_{0,-1} = 0$ is never realized so this is impossible. }

If $l > 0$, then, equating the coefficients of $t^{4l}$ in \eqref{eq:alpha0monomial}, we find the equality ${d_{1,0}^2=4d_{1,1}d_{1,-1}}$. This gives $d_{1,0}=0$ and $d_{1,1}d_{1,-1}=0$. We claim that this is impossible.

We first assume that $d_{1,0}=0$ and $d_{1,-1}=0$. In every walks under consideration, we must have $d_{1,1}\neq 0$. Equation \eqref{eq:alpha0monomial} simplifies as follows:
\begin{multline}\label{eq:alpha0monomialbisbis}
(d_{-1,0} -ct^{l-1})^2=d_{-1,0}^{2} -2cd_{-1,0} t^{l-1}+c^{2}t^{2l-2}\\= 4(d_{-1,1} +d_{0,1}ct^l +d_{1,1}c^2t^{2l})(d_{-1,-1,} +d_{0,-1} ct^l).
\end{multline}
On the right side  of \eqref{eq:alpha0monomialbisbis}, we find the monomial $4c^{3}d_{1,1}d_{0,-1}t^{3l}$. But, since $l>0$, we have $3l \neq 0,l-1, 2l-2$. It follows that  $d_{1,1}d_{0,-1}=0$
and, hence, $d_{0,-1}=0$. \ju{However, the condition $d_{1,0}=d_{1,-1}=d_{0,-1}= 0$ is never realized so this is impossible.}

We now assume that $d_{1,0}=0$ and $d_{1,1}=0$. In every walks under consideration, we must have $d_{1,-1}\neq 0$. Equation \eqref{eq:alpha0monomial} simplifies as follows:
\begin{multline}\label{eq:alpha0monomialbis}
(d_{-1,0} -ct^{l-1})^2=d_{-1,0}^{2} -2cd_{-1,0} t^{l-1}+c^{2}t^{2l-2}\\ = 4(d_{-1,1} +d_{0,1}ct^l)(d_{-1,-1,} +d_{0,-1} ct^l +d_{1,-1}c^2t^{2l}).
\end{multline}
On the right side  of \eqref{eq:alpha0monomialbis}, we find the monomial $4c^{3}d_{0,1}d_{1,-1}t^{3l}$. But, since $l>0$, we have $3l \neq 0,l-1,2l-2$. It follows that  $d_{0,1}d_{1,-1}=0$ and, hence,  
$d_{0,1}=0$. \ju{However, the condition $d_{1,0}=d_{1,1}=d_{0,1}=0$ is never realized so this is impossible.}

\vskip 5 pt

So, $P(t)$ is not constant and, hence, has at least one zero or pole at some $t_{0} \in \C^{*}$. 

\begin{center}
Step $3$: Poles of $P(t)$.
\end{center}
Assume that $P(t)$ has a pole $t_{0}\in \C^{*}$ of order $\kappa \geq 1$. Multiplying both sides of \eqref{eq:alpha0} by $(t-t_{0})^{4\kappa}$ and evaluating at $t_{0}$, we find that $d_{1,0}^{2}=4d_{1,1}d_{1,-1}$, which implies that $d_{1,0}=0=d_{1,1}d_{1,-1}$. 

We first assume that $d_{1,0}=0$ and $d_{1,1}=0$. Then, \eqref{eq:alpha0} simplifies as follows: 
\begin{equation}\label{eq:alpha0pole}
(d_{-1,0} -\frac{\alpha_0}{t})^2= 4(d_{-1,1} +d_{0,1} \alpha_0 )(d_{-1,-1} +d_{0,-1} \alpha_0 +d_{1,-1}\alpha_0^2).
\end{equation}
Multiplying both sides of \eqref{eq:alpha0pole} by $(t-t_{0})^{3\kappa}$ and evaluating at $t_{0}$, we find that
the term in $\alpha_{0}^{3}$ of the right hand side of \eqref{eq:alpha0pole} must be equal to $0$, {\it i.e.}, 
  $d_{0,1}d_{1,-1}=0$. So, we have either $d_{1,0}=d_{1,1}=d_{0,1}=0$ or $d_{1,0}=d_{1,1}=d_{1,-1}=0$. \ju{However, these conditions are never realized so this is impossible.}  

We now assume that $d_{1,0}=0$ and $d_{1,-1}=0$. Then, \eqref{eq:alpha0} simplifies as follows: 
\begin{equation}\label{eq:alpha0polebis}
(d_{-1,0} -\frac{\alpha_0}{t})^2= 4(d_{-1,1} +d_{0,1} \alpha_0 +d_{1,1} \alpha_0^2)(d_{-1,-1} +d_{0,-1} \alpha_0).
\end{equation}

Multiplying both sides of \eqref{eq:alpha0polebis} by $(t-t_{0})^{3\kappa}$ and evaluating at $t_{0}$, we find that
the term in $\alpha_{0}^{3}$ of the right hand side of \eqref{eq:alpha0polebis} must be equal to $0$, {\it i.e.},  $d_{1,1}d_{0,-1}=0$. So, we have either $d_{1,0}=d_{1,-1}=d_{1,1}=0$ or $d_{1,0}=d_{1,-1}=d_{0,-1}=0$. \ju{However, these conditions are never realized so this is impossible}.  So $P(t)$ has no poles in $\C$.
\begin{center}
Step $4$: Zeros of $P(t)$.
\end{center}

Assume that $P(t)$ has a zero $t_{0}\in \C^{*}$ of order $\nu \geq 1$.  Evaluating \eqref{eq:alpha0} at $t_{0}$, we get  
$d_{-1,0}^{2}=4d_{-1,1}d_{-1,-1}$, 
which implies $d_{-1,0}=0=d_{-1,1}d_{-1,-1}$. 

We first assume that $d_{-1,0}=0$ and $d_{-1,1}=0$. Then \eqref{eq:alpha0} simplifies as follows: 
\begin{equation}\label{eq:alpha0zero}
(-\frac{\alpha_0}{t} +d_{1,0} \alpha_0^2)^2= 4(d_{0,1} \alpha_0 +d_{1,1} \alpha_0^2)(d_{-1,-1} +d_{0,-1} \alpha_0 +d_{1,-1}\alpha_0^2).
\end{equation}
Since the walk is non degenerate, we have $d_{-1,-1}\neq 0$.
Let $\alpha \neq 0$ be the value of $(t-t_{0})^{-\nu}\alpha_{0}$ at $t_{0}$. Dividing both sides of \eqref{eq:alpha0zero} by $(t-t_{0})^{\nu}$ and evaluating at $t_{0}$ we find $0=4\alpha d_{0,1}d_{-1,-1}$, which implies $d_{0,1}=0$. So, we have obtained $d_{-1,0}=d_{-1,1}=d_{0,1}=0$. \ju{However, these conditions are never realized  so this is impossible}.    

We now assume that $d_{-1,0}=0$ and $d_{-1,-1}=0$. Then \eqref{eq:alpha0} simplifies as follows: 
\begin{equation}\label{eq:alpha0one}
(-\frac{\alpha_0}{t} +d_{1,0} \alpha_0^2)^2= 4(d_{-1,1} +d_{0,1} \alpha_0 +d_{1,1} \alpha_0^2)(d_{0,-1} \alpha_0 +d_{1,-1}\alpha_0^2).
\end{equation}
Since the walk is non degenerate, we have $d_{-1,1}\neq 0$.
Let $\alpha \neq 0$ be the value of $(t-t_{0})^{-\nu}\alpha_{0}$ at $t_{0}$. Dividing both sides of {\eqref{eq:alpha0one}} by $(t-t_{0})^{\nu}$ and evaluating at $t_{0}$ we find $0=4\alpha d_{-1,1}d_{0,-1}$, which implies $d_{0,-1}=0$. So, we have obtained $d_{-1,0}=d_{-1,-1}=d_{0,-1}=0$. \ju{However, these conditions are never realized so this is impossible}.   So $P(t)$ has no zeros in $\C$. Since it is not constant and it has no poles in $\C$, we find that $\alpha_0 \in \Q(t) \setminus \Q$ is impossible. This concludes the proof \ju{of Lemma~\ref{lemma:fixedpointinvolutionnotrational}}.
\end{proof}

In the following lemma, we focus our attention on the base points ${([x_{0}:x_{1}], [y_0:y_1])}$ of $\Etproj$ corresponding to the equation $x_1y_1=0$, namely:
\begin{eqnarray*}
P_1 =([1:0], [\beta_0:\beta_1]), & P_2 =\iota_1(P_1)=([1:0], [\beta_0':\beta_1']), \\ 
Q_1=([\alpha_0:\alpha_1],[1:0]), & Q_2=\iota_2(Q_1)=([\alpha_0':\alpha_1'],[1:0]).
\end{eqnarray*}

We will use the following notations :  
$$
L_x= \Q\left(\sqrt{\Delta^x_{[1:0]}/t^2}\right) \text{ and } L_y=\Q\left(\sqrt{\Delta^y_{[1:0]}/t^2}\right). 
$$

\begin{lemma}\label{lem:basefieldbasepoints}
The points $P_1$ and $P_2$ (resp. $Q_1$ and $Q_2$)  are  
  $L_x$-points (resp. $L_y$-points) of $\Etproj$.  They
  are $\Q$-points of $\Etproj$ if and only if $\Delta^x_{[1:0]}/t^2$ (resp. $\Delta^y_{[1:0]}/t^2$)  is a square in $\Q$. \ju{Moreover,} the following properties hold true~: 
  \begin{itemize}
  \item $\Delta^x_{[1:0]}/t^2$  is a square in $\Q$ if and only if $d_{1,-1}d_{1,1}=0$; moreover :
  \begin{itemize} 
  \item $d_{1,1}=0$ if and only if there exist $i,j \in \{1,2\}$ such that ${P_i=Q_j=([1:0],[1:0])}$; \ju{this corresponds to the walks $\walk{IIB.*}$, $\walk{IID.*}$;} 
  \item $d_{1,-1}=0$ if and only if there exists $i \in \{1,2\}$ such that ${P_{i}=([1:0],[0:1])}$; \ju{this corresponds to the walks $\walk{IB.*}$, $\walk{IC.*}$, $\walk{IIC.*}$, $\walk{III}$;}
  \end{itemize}
  \item $\Delta^y_{[1:0]}/t^2$  is a square in $\Q$ if and only if $d_{-1,1}d_{1,1}=0$; moreover :
  \begin{itemize}
  \item $d_{1,1}=0$ if and only if there exist $i,j \in \{1,2\}$ such that ${P_i=Q_j=([1:0],[1:0])}$; \ju{this corresponds to the walks $\walk{IIB.*}$, $\walk{IID.*}$;}
  \item $d_{-1,1}=0$ if and only if there exists $j \in \{1,2\}$ such that ${Q_{j}=([0:1],[1:0])}$; \ju{this corresponds to the walks $\walk{IIA.*}$, $\walk{IIB.1-6}$, $\walk{IIC.*}$, $\walk{IID.1}$, $\walk{IID.2}$, $\walk{IID.5},\walk{III}$.}
  \end{itemize}
  \end{itemize}
\end{lemma}

\begin{proof}
Taking into consideration the obvious symmetry in $x$ and $y$, we see that it is sufficient to prove the result for $P_{1}$ and $P_{2}$ and the first of the last two statements of the Lemma.
 
The $y$-coordinates $[\beta_0:\beta_1]$ and $[\beta'_0:\beta'_1]$ of $P_1$ and $P_2$ are the roots in $\P1(\C)$ of the homogeneous polynomial in $y_{0}$ and $y_{1}$ given by  
\begin{equation}\label{eq:equaycoords}
d_{1,-1}y_1^2 +d_{1,0}y_0y_1 +d_{1,1}y_0^2=0.
\end{equation}  
Therefore, $[\beta_0:\beta_1]$ and $[\beta'_0:\beta'_1]$ belong to $\P1(L_x)$.  Moreover, we see that they belong to $\P1(\Q)$ if and only if $\Delta^x_{[1:0]}/t^2=d_{1,0}^2-4d_{1,-1}d_{1,1}$ is a square in $\Q$. Since the $d_{i,j}$ are in $\{0,1\}$, we have that $d_{1,0}^2-4d_{1,-1}d_{1,1}$ is a square in $\Q$ if and only if $d_{1,-1}d_{1,1}=0$.

The fact that $d_{1,1}=0$ is equivalent to the fact that $[1:0]$ is a root of equation \eqref{eq:equaycoords} and this is equivalent to the fact that there exist $i,j \in \{1,2\}$ such that $P_i=Q_j=([1:0],[1:0])$. 

Similarly, the fact that $d_{1,-1}=0$ is equivalent to the fact that $[0:1]$ is a root of equation \eqref{eq:equaycoords} and this is equivalent to the fact that  $P_{1}=([1:0],[0:1])$ or $P_{2}=([1:0],[0:1])$. 
\end{proof}

\subsection{Galois action on $\Etproj$}

 Let $\Q(t) \subset L \subset \C$ be a field extension. For any $L$-point $P=([x_0:x_1],[y_0:y_1])$ of $\Etproj$, with $x_{0},x_{1},y_{0},y_{1} \in L$, and any $\sigma \in \Aut(L/\Q(t))$, we set  
 $$
 \sigma(P)=([\sigma(x_0):\sigma(x_1)],[\sigma(y_0):\sigma(y_1)]).
 $$ 
 Since $\Etproj$ is defined over $\Q(t)$,  $\sigma(P)$ is an $L$-point of $\Etproj$.
 
 \begin{prop}\label{prop:galoisactionpointsellipticcurve}
Let $\Q(t) \subset L \subset \C$ be a field extension and let $\sigma \in \Aut(L/\Q(t))$. Let $P$ be a $L$-point of $\Etproj$. Then, the following properties hold true :
\begin{itemize}
\item $\iota_1(P), \iota_2(P)$ and, hence, $\tau^{n}(P)$ for any $n \in \Z$, are $L$-points of $\Etproj$;
\item for any $k\in \{1,2\}$, $\iota_k \circ \sigma = \sigma \circ \iota_k$ on $\Etproj(L)$ and, hence, $\tau \circ \sigma=\sigma \circ \tau$ on $\Etproj(L)$. 
\end{itemize}
 \end{prop}
 
 \begin{proof}
We only prove the assertions concerning $\iota_1$. The proofs for $\iota_2$ are similar and the assertions concerning $\tau$ follow from those about $\iota_{1}$ and $\iota_{2}$ since $\tau=\iota_{2} \circ \iota_{1}$.

 We set $P=([a_0:a_1],[b_0:b_1]) \in \Etproj(L)$ (with $a_{0},a_{1},b_{0},b_{1} \in L$) and ${\iota_1(P)= ([a_0:a_1],[b_0':b_1'])}$. The point $[b_0':b_1']$ is characterized by the fact that $[b_0:b_1]$ and $[b_0':b_1']$ are the roots in $\P1(\C)$ of the homogeneous polynomial in $y_{0}$ and $y_{1}$ given by 
 \begin{equation}\label{eq:iota1}
A(a_0,a_1) y_0^2 + B(a_0,a_1) y_0y_1 +  C(a_0,a_1) y_1^2
 \end{equation}
where $A(a_0,a_1) =  d_{-1,1} a_1^2 + d_{0,1} a_0a_1 + d_{1,1}a_0^2$, $B(a_0,a_1) = d_{-1,0} a_1^2 -  \frac{1}{t} a_0a_1 + d_{1,0}a_0^2$ and  $ C(a_0,a_1) = d_{-1,-1} a_1^2 + d_{0,-1} a_0a_1 + d_{1,-1}a_0^2$ (these are not all $0$). 
Since \eqref{eq:iota1} has coefficients in $L$  and $b_0, b_1 \in L$, we can assume that $b_0',b_1' \in L$ as well. Hence, $[b_0':b_1'] \in \P1(L)$ and $\iota_{1}(P) \in \Etproj(L)$, as desired. 

Moreover, $[\sigma(b_0):\sigma(b_1)]$ and $[\sigma(b_0'):\sigma(b_1')]$ are the roots in $\P1(\C)$ of 
$$
A(\sigma(a_0),\sigma(a_1)) y_0^2+B(\sigma(a_0),\sigma(a_1)) y_0y_1 + C(\sigma(a_0),\sigma(a_1)) y_1^2.
$$ 
Therefore, $\iota_{1}(\sigma(P))=([\sigma(a_0):\sigma(a_1)],[\sigma(b_0'):\sigma(b_1')])=\sigma(\iota_{1}(P))$. 
 \end{proof}

 \subsection{On the $\tau$-orbits}

 \begin{defn} 
 We define an equivalence relation $\sim$ on $\Etproj$ by 
 $$
 P \sim Q \Leftrightarrow  \exists n \in \Z,\tau^n(P)=Q.
 $$ 
 If $P \sim Q$ is not true, we shall write $P \nsim Q$.
 An equivalence class for $\sim$ will be called a $\tau$-orbit. 
 \end{defn}
 
 \begin{lemma}\label{lemma:genericorbits}
\ju{We set $M_{x}=\Q(t)\left(\sqrt{\Delta^x_{[1:0]}}\right)$ and ${M_{y}=\Q(t)\left(\sqrt{\Delta^y_{[1:0]}}\right)}$.}
The following properties hold true~:
 \begin{itemize}
 \item if $\Q(t) \subsetneq \ju{M_{x}}$ or $\Q(t) \subsetneq \ju{M_{y}}$ then, for all 
 $i,j \in \{1,2\}$, we have $P_i \nsim Q_j$;
 \item if  $\Q(t) \subsetneq \ju{M_{x}}$ (resp. $\Q(t) \subsetneq  \ju{M_{y}}$)
 then $P_i \nsim P_j$ (resp. $Q_i \nsim Q_j$) for $i \neq j$.
 \end{itemize}
  \end{lemma}
  
 \begin{proof}
We recall that due to the assumption \ref{assumption:generalcase t and D}, $\tau$ has infinite order (so that it corresponds to a translation by a non torsion point). Let us prove the first assertion. Suppose  to the contrary that, for instance, $P_1 \sim Q_1$ and that 
 $\Q(t) \subsetneq \ju{M_{x}}$, the other cases being similar. So, there exists $n \in \Z$ such that $\tau^n(P_1)=Q_1$. The fact that $P_{1}=([1:0], [\beta_0:\beta_1])$ belongs to $\Etproj$ means that  
 \begin{equation}
 d_{1,-1} \beta_1^{2} + d_{1,0} \beta_{0} \beta_1 + d_{1,1} \beta_0^{2}=0. 
 \end{equation}
 Since $\Q(t) \subsetneq \ju{M_{x}}$, we have that $\Delta^x_{[1:0]}/t^2=d_{1,0}^2-4d_{1,1}d_{1,-1}$ is not a square in $\Q(t)$. It follows that $d_{1,1}d_{1,-1} \neq 0$ and that $P_1 \in \Etproj(\ju{M_{x}}) \setminus \Etproj(\Q(t))$. On the other hand, the fact that $Q_1=([\alpha_0:\alpha_1],[1:0])$ belongs to $\Etproj$ means that
$$
d_{-1,1} \alpha_1^{2}+d_{0,1} \alpha_0 \alpha_1+d_{1,1} \alpha_0^2=0. 
$$ 
So, $Q_{1}$ belongs to $\Etproj(\ju{M_{y}})$. Since $\tau^{-n}(Q_1)=P_1$, Proposition \ref{prop:galoisactionpointsellipticcurve} ensures that $P_1 \in \Etproj(\ju{M_{y}})$ as well. Therefore, $P_1 \in (\Etproj(\ju{M_{x}}) \setminus \Etproj(\Q(t))) \cap \Etproj(\ju{M_{y}})$. In particular, $\ju{M_{x}} \cap \ju{M_{y}}$ is not reduced to $\Q(t)$. Since $\ju{M_{x}}$ and $\ju{M_{y}}$ are fields extensions of degree at most $2$ of $\Q(t)$, we get $\ju{M_{x}}=\ju{M_{y}}$. Let ${\sigma \in \Gal(\ju{M_{x}}/\Q(t))=\Gal(\ju{M_{y}}/\Q(t))}$ be an element of order $2$. We obviously have $\sigma(P_1)=P_2$ and $\sigma(Q_1)=Q_2$. Using Proposition \ref{prop:galoisactionpointsellipticcurve}, it follows that ${\tau^n(P_2)=\tau^n(\sigma(P_1))=\sigma(\tau^n(P_1))=\sigma(Q_1)=Q_{2}}$. Therefore, ${\iota_2 \tau^n(P_2)= \iota_2(Q_2)=Q_1}$. But, we have ${\iota_2 \tau^n =\tau^{-n +1} \iota_1}$ (because $\tau=\iota_2 \iota_1$ and the $\iota_k$ are involutions). Then, we find ${\tau^{-n +1}(P_1)=\tau^{-n +1} \iota_1(P_2)=\iota_2 \tau^n(P_2)=Q_1=\tau^n(P_1)}$. 
 This gives $\tau^{-2n+1}(P_1)=P_1$. Since  $\tau$ is a translation by a non torsion point of $\Etproj$, this implies that $-2n+1=0$. This yields a contradiction because $n \in \Z$. 
 
 We shall now prove the second assertion. Assume that $\Q(t) \subsetneq \ju{M_{x}}$. In particular $\Delta^x_{[1:0]} \neq 0$ and, hence,   $P_1 \neq \iota_{1}(P_{1})=P_2$. Suppose to the contrary that $P_1 \sim P_2$, that is, that there exists $n \in \Z^*$ such that $\tau^n(P_1)=P_2$.  Let $\sigma \in \Gal(\ju{M_{x}}/\Q(t))$ be an element of order $2$. We obviously have $\sigma(P_1)=P_2$. Using Proposition \ref{prop:galoisactionpointsellipticcurve}, we get $\tau^{n}(P_{2})=\tau^n(\sigma(P_1))=\sigma(\tau^n(P_1))=\sigma(P_{2})=P_1$. Therefore, $\tau^{2n}(P_1)=P_1$. Since $\tau$ is a translation by a non torsion point of $\Etproj$, this implies that $n=0$ and hence, $P_{1}=P_{2}$. This yields a contradiction. 
 \end{proof}

\subsection{The poles of $b_{1}$ and $b_{2}$}\label{subsec:setpoles}

\begin{lemma}\label{lemma:divisorsecondmember}
The set of poles of $b_{1}= \iota_1(y)\left(\tau (x)-x\right)$ in $\Etproj$ is contained in 
$$\mathcal{S}_{1}=\{\iota_{1}(Q_1),\iota_{1}(Q_2), P_1,P_2,\tau^{-1}(P_1),\tau^{-1}(P_2)\}.$$
Similarly, the set of poles of $b_{2}=x(\iota_1(y)-y)$ in $\Etproj$ is contained in 
$$\mathcal{S}_{2}=\{P_1,P_2, Q_1,Q_2,\iota_1(Q_1),\iota_1(Q_2)\}=\{P_1,P_2,Q_1,Q_2,\tau^{-1}(Q_1),\tau^{-1}(Q_2)\}.$$
Moreover, we have
\begin{equation}\label{equ:secondmembreform}
(b_{2})^2=\frac{ x_0^2 \Delta^x_{[x_0:x_1]}}{x_1^2(\sum_{i=0}^2  x_0^i x_1^{2-i}td_{i-1,1} )^2 }.
\end{equation}  
\end{lemma}

 \begin{proof}
The proof of the assertions about the localization of the poles of $b_{1}$ and $b_{2}$ are straightforward. Let us prove  \eqref{equ:secondmembreform}. By definition, $ \iota_1 ( \frac{y_0}{y_1} )$ and $\frac{y_0}{y_1} $ are the two roots of the polynomial $y\mapsto \overline{K}(x_0,x_1,y,t)$. The square of their difference equals to the discriminant divided by the square of the leading term. Then, we have 
$$
\left(  \iota_1 ( \frac{y_0}{y_1} ) -  \frac{y_0}{y_1} \right)^2= \frac{ \Delta^x_{[x_0:x_1]}}{(\sum_i  x_0^i x_1^{2-i}td_{i-1,1} )^2 }.
$$
Therefore, we find
$$
b_{2}(\frac{x_0}{x_1},\frac{y_0}{y_1})^2=\frac{ x_0^2 \Delta^x_{[x_0:x_1]}}{x_1^2(\sum_i  x_0^i x_1^{2-i}td_{i-1,1} )^2 }.
$$
\end{proof}

\section{\ju{Hypertranscendance of generating series of the walks in $\walksg$}}\label{sec:hypertrF1F2}

\ju{The treatments of the walks in $\walksg$ are finalized in the following subsections; more precisely :
\begin{itemize}
\item for $\walk{IA.*}$, \ju{$\walk{IB.*}$}, $\walk{IC.*}$ and $\walk{IIA.*}$, see Section \ref{sec:gencase}, Theorem~\ref{theo:genericcaseunweighted2};
\item for \ju{$\walk{IIB.4}$}, \ju{$\walk{IIB.5}$}, \ju{$\walk{IIB.8}$}, \ju{$\walk{IIB.9}$}, \ju{$\walk{IIB.10}$}, see Section \ref{sec:caseIIB}, Theorem~\ref{thm:caseIIB}; 
\item for \ju{$\walk{IID.*}$}, see Section \ref{sec:caseIID}, Theorem~\ref{thm:caseIID}; 
\item for \ju{$\walk{III}$}, see Section \ref{sec:caseIII}, Theorem~\ref{thm:caseIII}; 
\item For \ju{$\walk{IIC.3}$}, see Section \ref{sec:caseIIC3}, Theorem~\ref{thm:caseIIC3}. 
\end{itemize}
}
\subsection{Generic cases}\label{sec:gencase}
The following proposition gives a diophantine criteria for the hypertranscendency of $F^{1}(x,t)$ and $F^{2}(y,t)$.

\begin{prop}\label{prop:genericcaseunweighted} 
\ju{We assume that the assumptions \ref{assumption:generalcase t and D} are satisfied.}
If $\Delta^x_{[1:0]}/t^2=d_{1,0}^2-4d_{1,-1}d_{1,1}$ or $\Delta^y_{[1:0]}/t^2=d_{0,1}^2-4d_{-1,1}d_{1,1}$ is not a square in $\Q$ then $F^{1}(x,t)$ and $F^{2}(y,t)$ are hypertranscendental over $\C(\Etproj)$.
\end{prop}
\begin{proof}
 Assume for instance that $\Delta^x_{[1:0]}/t^2$ is not a square in $\Q$, the other case being similar. Combining Corollary \ref{coro: hyperalg F} and Proposition \ref{prop:caract hyperalg F1 and F2}, we see that it is sufficient to prove that $P_{1}$ is a pole of $b_{2}$ and that it is the only pole of $b_{2}$ of the form $\tau^{n}(P_{1})$ with $n \in \Z$. The fact that $P_{1}$ is a pole of $b_{2}$ is clear (indeed, on the one hand, $P_1$ is a pole of $x$ and, on the other hand, the $y$-coordinates of $P_1$ and $\iota_1(P_1)$ are distinct because $\Delta^x_{[1:0]} \neq 0$, and, hence, $P_{1}$ is not a zero of $\iota_{1}(y)-y$). Moreover, Lemma~\ref{lemma:genericorbits} implies that $P_1 \nsim P_2$ and $P_1 \nsim Q_i$ for $i=1,2$. The latter also implies that $P_1 \nsim \tau^{-1}(Q_i)$ for $i=1,2$. But,  Lemma \ref{lemma:divisorsecondmember} ensures that the set of poles of $b_{2}$ is included in 
$\{P_1,P_2,Q_1,Q_2,\tau^{-1}(Q_1),\tau^{-1}(Q_2)\}$. So, $P_{1}$ is the only pole of $b_{2}$ of the form $\tau^{n}(P_{1})$ with $n \in \Z$, as desired. 
\end{proof}

\begin{thm}\label{theo:genericcaseunweighted2}
For any of the walks \ju{$\walk{IA.*}$}, \ju{$\walk{IB.*}$}, \ju{$\walk{IC.*}$} or \ju{$\walk{IIA.*}$} and  for any $t \in]0, 1/|\mathcal{D}|[ \setminus \Qbar$ such that $G_t$ is infinite, the generating series 
 $F^{1}(x,t)$ and $F^{2}(y,t)$ are hypertranscendental over $\C(\Etproj)$.
\end{thm}

\begin{proof}
\ju{According to Proposition \ref{prop:genericcaseunweighted}, it is sufficient to prove that either the discriminant $\Delta^x_{[1:0]}/t^2=d_{1,0}^2-4d_{1,-1}d_{1,1}$ or $\Delta^y_{[1:0]}/t^2=d_{0,1}^2-4d_{-1,1}d_{1,1}$ is not a square in $\Q$. This is true because, according to Lemma \ref{lem:basefieldbasepoints},  they are both squares in $\Q$ if and only if the walk we consider is among $\walk{IIB.*},\walk{IIC.*},\walk{IID.*},\walk{III}$.}
\end{proof}

\subsection{Non generic Cases}\label{sec:the non gen cases}

In this subsection, we shall focus our attention on the walks \ju{in $\walksg$} \ju{which are not covered by Theorem~\ref{theo:genericcaseunweighted2}, {\it i.e.}, on the walks in \ju{$\walksg$} among \ju{$\walk{IIB.*}$}, \ju{$\walk{IIC.*}$}, \ju{$\walk{IID.*}$} and \ju{$\walk{III}$}}. \ju{This gives 16 walks, namely   
{\mfs $\walk{IIB.4}$}, 
$\walk{IIB.5}$,
  $\walk{IIB.8}$,
  $\walk{IIB.9}$,
  $\walk{IIB.10}$,
  $\walk{IIC.3}$,
 $\walk{IID.1}$, 
 $\walk{IID.2}$,
  $\walk{IID.3}$,  
  $\walk{IID.4}$,  
  $\walk{IID.5}$,
  $\walk{IID.6}$,
  $\walk{IID.7}$,
  $\walk{IID.8}$,
  $\walk{IID.9}$,
   $\walk{III}$.}  
   In each of these cases, we will show that the corresponding generating series is hypertranscendental.

\ju{For the convenience of the reader, we have included in Figure~\ref{figure:exbis} a table of the the values of the $d_{i,j}$ for the various walks considered in this section.}
\begin{figure}
$$
\ju{\begin{tabular}{|c|c|}
  \hline
Walks & The nonzero $d_{i,j}$ \\
  \hline
 {\mfs $\walk{IIB.4}$} & {\mfs $d_{1,0}=d_{0,1}=d_{-1,-1}=d_{0,-1}=d_{-1,0}=1$ } \\
   \hline
 $\walk{IIB.5}$ & $d_{1,0}=d_{0,1}=d_{-1,0}=d_{-1,-1}=d_{0,-1}=1$ \\
 \hline
  $\walk{IIB.8}$ & $d_{1,0}=d_{0,1}=d_{-1,1}=d_{-1,-1}=d_{1,-1}=1$  \\
   \hline
  $\walk{IIB.9}$ & $d_{1,0}=d_{0,1}=d_{-1,1}=d_{-1,0}=d_{-1,-1}=d_{1,-1}=1$   \\
   \hline
  $\walk{IIB.10}$ & $d_{1,0}=d_{0,1}=d_{-1,1}=d_{-1,0}=d_{-1,-1}=d_{0,-1}=d_{1,-1}=1$ \\
   \hline
  $\walk{IIC.3}$ & $d_{1,0}=d_{1,1}=d_{0,1}=d_{-1,-1}=1$   \\
   \hline
 $\walk{IID.1}$ & $d_{0,1}=d_{-1,0}=d_{0,-1}=d_{1,-1}=1$ \\
  \hline
 $\walk{IID.2}$ & $d_{0,1}=d_{-1,0}=d_{-1,-1}=d_{1,-1}=1$ \\
  \hline
  $\walk{IID.3}$ & $d_{0,1}=d_{-1,1}=d_{-1,-1}=d_{1,-1}=1$  \\
   \hline
  $\walk{IID.4}$ & $d_{0,1}=d_{-1,1}=d_{-1,0}=d_{1,-1}=1$ \\
   \hline
  $\walk{IID.5}$ & $d_{0,1}=d_{-1,0}=d_{-1,-1}=d_{0,-1}=d_{1,-1}=1$ \\
   \hline
  $\walk{IID.6}$ &  $d_{0,1}=d_{-1,1}=d_{-1,-1}=d_{0,-1}=d_{1,-1}=1$ \\
   \hline
  $\walk{IID.7}$ & $d_{0,1}=d_{-1,1}=d_{-1,0}=d_{0,-1}=d_{1,-1}=1$ \\
   \hline
  $\walk{IID.8}$ & $d_{0,1}=d_{-1,1}=d_{-1,0}=d_{-1,-1}=d_{1,-1}=1$  \\
   \hline
  $\walk{IID.9}$ & $d_{0,1}=d_{-1,1}=d_{-1,0}=d_{-1,-1}=d_{0,-1}=d_{1,-1}=1$  \\
   \hline
   $\walk{III}$ & $d_{1,1}=d_{-1,0}=d_{-1,-1}=d_{0,-1}=1$  \\
  \hline
\end{tabular}}
$$
\caption{The nonzero $d_{i,j}$ for the walks considered in Section~\ref{sec:the non gen cases}}
\label{figure:exbis}
\end{figure}


\subsubsection{The walks \ju{$\walk{IIB.*}$} }\label{sec:caseIIB}

In that situation, we have 
$$
d_{1,1}=0, \ d_{1,0}\neq 0 \text{ and } d_{0,1}\neq 0.
$$ 
The following properties hold :
\begin{itemize}
\item according to Lemma \ref{lem:basefieldbasepoints}, there exist $i,j \in \{1,2\}$ such that $P_{i}=Q_{j}$; up to renumbering, we can and will assume that $P_{1}=Q_{1}=([1:0],[1:0])$; 
\item $P_{1} \neq P_{2}=\iota_{1}(P_{1})$ because $\Delta^x_{[1:0]}=d_{1,0}^{2}-4d_{1,1}d_{1,-1}=d_{1,0}^{2} \neq 0$; 
\item $P_{1}=Q_{1}=([1:0],[1:0]) \neq Q_{2} = \iota_{2}(Q_{1})$ in virtue of Lemma \ref{lem:fixedbasepointsinvolution} (or simply because $\Delta^y_{[1:0]} \neq 0$);
\item  $Q_{1} \neq \iota_{1}(Q_{2})$ because $Q_{1} \neq Q_{2}$ so $Q_{1}$ and $Q_{2}$ do not have the same $x$-coordinates.
\end{itemize}

In particular, we see that 

\begin{itemize}
\item $P_{1}=Q_{1}=([1:0],[1:0])$;
\item $\iota_1(P_1)=P_2=([1:0],[\beta_1':\beta_2'])$ with $[\beta_1':\beta_2'] \neq [1:0]$; 
\item $\iota_1(P_2)=P_1=([1:0],[1:0])$; 
\item $Q_{2}=([\alpha'_0:\alpha'_1],[1:0])$ with $[\alpha'_0:\alpha'_1] \neq [1:0]$; 
\item $\iota_1(Q_{2})\neq P_{1},P_{2}$. 
\end{itemize}

\begin{thm}\label{thm:caseIIB} 
For any of the walks   \ju{$\walk{IIB.4}$}, \ju{$\walk{IIB.5}$}, \ju{$\walk{IIB.8}$}, \ju{$\walk{IIB.9}$}, \ju{$\walk{IIB.10}$}
and any ${t \in ] 0,1 /\vert \mathcal{D} \vert [ \ \setminus \Qbar}$ such that $G_t$ is infinite, $F^1(x,t)$ and $F^2(y,t)$ are hypertranscendental.
\end{thm}

\begin{proof}
\ju{Using Lemma \ref{lemma:divisorsecondmember}, we see that the set of poles of $b_2$ is included in $\{P_1,P_2, Q_1,Q_2,\iota_1(Q_1),\iota_1(Q_2)\}=\{P_1,P_2,Q_2,\iota_1(Q_2)\}$ and that: }

\begin{itemize}
\item $P_{1}$ is a pole of order $2$ of  $b_2$ because
\begin{itemize}
\item $P_{1}$ is a pole of order $1$ of $x_{0}/x_{1}$;
\item $P_{1}$ is a pole of order $1$ of $y_{0}/y_{1}$;
\item $P_{1}$ is not a pole of $\iota_{1}(y_{0}/y_{1})$. 
\end{itemize}
\item $P_{2}$ is a pole of order $2$ of  $b_2$ because
\begin{itemize}
\item $P_{2}$ is a pole of order $1$ of $x_{0}/x_{1}$;
\item $P_{2}$ is not a pole of $y_{0}/y_{1}$;
\item $P_{2}$ is a pole of order $1$ of $\iota_{1}(y_{0}/y_{1})$. 
\end{itemize}
\end{itemize}
There are at least two double poles. Since $x_{0}/x_{1}$, $y_{0}/y_{1}$ and $\iota_{1}(y_{0}/y_{1})$ have at most two poles counted with multiplicities, we find that $b_{2}=x_{0}/x_{1}(\iota_{1}(y_{0}/y_{1})- y_{0}/y_{1})$ has at most $6$ poles, counted with multiplicities. So there are at most $3$ double poles (in fact $P_{1}$ and $P_{2}$ are the only double poles in this situation but this fact will not be used). \par 

If $P_{1}$ and $P_{2}$ are the only double poles, combining Corollary \ref{coro: hyperalg F} and Proposition~\ref{prop:caract hyperalg F1 and F2}, we see that, in order to conclude, it is sufficient to show that 
$P_1 \nsim P_2$. Assume that there exists a third double pole $P_{3}$ and that $P_1 \not \sim P_2$. Then there should exists $j\in \{1,2\}$ such that 
$P_{3}\not \sim P_{j}$. Combining Corollary \ref{coro: hyperalg F} and Proposition \ref{prop:caract hyperalg F1 and F2}, we see that we have the conclusion in this case. So it is sufficient to prove that $P_{1}\not\sim P_{2}$.

Suppose to the contrary that $P_1 \sim P_2$, {\it i.e.}, that there exists $n \in \Z$ such that $\tau^n(P_1)=P_2$. 

On the one hand, since $P_1 \in \Etproj(\Q) \subset \Etproj(\Q(t))$, Proposition \ref{prop:galoisactionpointsellipticcurve} implies that $\tau^j(P_1) \in \Etproj(\Q(t))$ for any $j \in \Z$. Moreover, Proposition \ref{prop:galoisactionpointsellipticcurve} implies that ${P_2 = \iota_1(P_1) \in \Etproj(\Q(t))}$ is such that, for any $j \in \Z$,  $\tau^j(\iota_2(P_2)) \in \Etproj(\Q(t))$.

On the other hand, it is easily seen that the equality $\tau^n(P_1)=P_2$, together with the fact that $\tau=\iota_2 \circ \iota_1$ is the composition of two involutions, imply, 
\begin{itemize}
\item if $n=2k$, then $\tau^{k}(P_{1})=\tau^{-k}(P_{2})$ and thus $$\iota_{1}\tau^{k}(P_{1})=\iota_{1}\tau^{-k}(P_{2})=\iota_{1}\tau^{-k}\iota_{1}(P_{1})=\tau^k(P_1).$$
\item if $n=2k+1$, then $\tau^{k+1}(P_{1})=\tau^{-k}(P_{2})$ and thus $$\tau^{k}\iota_{2}(P_{2})=\tau^{k+1}(P_{1})=\tau^{-k}(P_{2})=\iota_{2}\tau^{k}\iota_{2} P_{2}.$$
\end{itemize}
 For $n=2k$, we get that $\tau^k(P_1)$ is fixed by the involution $\iota_1$. 
 Lemma \ref{lem:fixedbasepointsinvolution} ensures that for the walks under considerations none of the base points of $\Etproj$ is fixed by $\iota_1$. Therefore, $\tau^k(P_1)$ is not a base point. By Lemma \ref{lemma:fixedpointinvolutionnotrational}, we conclude that $\tau^k(P_1) \notin \Etproj(\Q(t))$. This yields a contradiction. For $n=2k+1$, we get that $\tau^k(\iota_2(P_2))$ is fixed by the involution $\iota_2$. Lemma \ref{lem:fixedbasepointsinvolution} ensures that for the walks under consideration none of the base points of $\Etproj$ is fixed by $\iota_2$. Therefore, $\tau^k(\iota_2(P_2))$ is not a base point. By Lemma~\ref{lemma:fixedpointinvolutionnotrational}, we conclude that $\tau^k(\iota_2(P_2)) \notin \Etproj(\Q(t))$. This yields a contradiction.
\end{proof}

\subsubsection{The walks \ju{$\walk{IID.*}$}}\label{sec:caseIID}

In that situation, we have 
$$
d_{1,1}=0, \ d_{1,0}=0 \text{ and } d_{0,1}\neq 0.
$$
In particular, $\Delta^x_{[1:0]}=0$ and, hence, $P_{1}=P_{2}$.  Moreover, Lemma \ref{lem:basefieldbasepoints} ensures that $P_{i}=Q_{j}=([1:0],[1:0])$ for some $i,j \in \{1,2\}$. Up to renumbering, we can and will assume that $P_{1}=Q_{1}$. Since $\Delta^y_{[1:0]}=d_{0,1}^{2} \neq 0$, we have $P_{1}=Q_{1} \neq Q_{2}$.  Since $P_1=\iota_1(P_1)$, we also have $P_1=Q_{1} \neq \iota_1(Q_2)$. Last, using Lemma \ref{lem:fixedbasepointsinvolution}, we see that $Q_{2} \neq \iota_{1}(Q_{2})$.

\begin{thm}\label{thm:caseIID}
For any of the walks \ju{$\walk{IID.*}$} and any 
 $t \in \,]0,1/\vert \mathcal{D} \vert[ \,\setminus \Qbar$  such that $G_t$ is infinite, $F^1(x,t)$ and $F^2(y,t)$ are hypertranscendental.
\end{thm}

\begin{proof}
We recall the formula \eqref{equ:secondmembreform}
$$
(b_{2})^2=\frac{ x_0^2 \Delta^x_{[x_0:x_1]}}{x_1^2(\sum_i  x_0^i x_1^{2-i}td_{i-1,1} )^2 }=\frac{x_0^2}{x_1^4}\frac{\Delta^x_{[x_0:x_1]}}{t^2(x_1d_{-1,1}+ x_0 d_{0,1})^2}. 
$$
Since the curve $\Etproj$ is nonsingular, the point $P_1$ is a simple zero of $\Delta^x_{[x_0:x_1]}$ seen as a rational function on $\P1(\C)$ (cf Proposition \ref{prop:genuscurvewalk} and Corollary \ref{cor:genuscurvewalk}).
Using additionally $d_{0,1}\neq 1$, it is easily seen that the polar divisor of $b_{2}$ is of the form $$3[P_{1}]+i[Q_{2}]+j[\iota_{1}(Q_{2})],$$ for some $i, j \in \{0,1,2\}$. 
Thus, $P_{1}$ is the only pole of $b_{2}$ of order $\geq 3$. The result is now a consequence of Corollary \ref{coro: hyperalg F} and Proposition~\ref{prop:caract hyperalg F1 and F2}. 
\end{proof}
 
 \subsubsection{The walk \ju{$\walk{III}$}}\label{sec:caseIII}
 
 This walk is symmetric in the sense of the following definition. 

\begin{defi}\label{def:symmwalks}
We say that a walk is symmetric if $d_{i,j}=d_{j,i}$ for all $i,j \in \{0,\pm 1\}$. 
\end{defi}

Note that the walk is symmetric if and only if  
$$
\overline{K}(x_{0},x_{1},y_{0},y_{1},t)=\overline{K}(y_{0},y_{1},x_{0},x_{1},t). 
$$
Therefore, the involutive morphism $s$ of $\P1 \times \P1$ defined by 
$$
s(x,y)=(y,x)
$$
induces an involutive morphism of $\Etproj$ in the symmetric case, still denoted by $s$. Note that   
\begin{equation}\label{eq:siotaiotas}
s \circ \iota_{1}=\iota_{2} \circ s. 
\end{equation}
Indeed, on the one hand, for any point $P=(x,y) \in \Etproj$, we have that
$
{\{P,\iota_1(P)\} = \Etproj \cap (\{x\} \times \P1(\C))}.
$
and, hence, ${\{s(P),s(\iota_1(P))\} = \Etproj \cap (\P1(\C) \times \{x\})}$. On the other hand, we find
${\{s(P),\iota_2(s(P))\} = \Etproj \cap (\P1(\C) \times \{x\})}$. Whence the desired equality $s(\iota_1(P))=\iota_2(s(P))$. 

Similarly, we have 
\begin{equation}\label{eq:iotassiota}
 \iota_{1} \circ s = s  \circ \iota_{2}. 
\end{equation}
It follows that 
\begin{equation}\label{eq:tauiotaiotatau}
s \circ \tau=\tau^{-1} \circ s. 
\end{equation}

\begin{lemma}\label{lem:fixedsandsim}
Assume that the walk under consideration is symmetric. Consider  $R_{1},R_{2}\in \Etproj(\Q(t))$ such that $s(R_{1})=R_{2}$. If $R_{1}\sim R_{2}$ then there exists ${R_{3}\in \Etproj(\Q(t))}$ such that $s(R_{3})=R_{3}$. 
\end{lemma}

\begin{proof}
We have to prove that, if there exists $\ell\in  \Z$ such that $\tau^{\ell}(R_{1})=R_{2}$ for some $R_{1},R_{2} \in \Etproj(\Q(t))$ {such that $s(R_1)=R_2$}, then there exists $R_{3}\in \Etproj(\Q(t))$ such that $s(R_{3})=R_{3}$. Up to interchanging $R_{1}$ and $R_{2}$, we can assume that $\ell \geq 0$. We argue by induction on $\ell \geq 0$. The result is obvious for $\ell=0$. Let us assume that the result is true for some $\ell \geq 0$.  
 The equality $\tau^{\ell}(R_{1})=R_{2}$ ensures that $s\tau^{\ell}(R_{1})=s(R_{2})=R_{1}$. Using \eqref{eq:tauiotaiotatau}, we get $\tau^{-\ell}s(R_{1})=R_{1}$. Using the equality $\tau^{-\ell}=\iota_{1} \tau^{\ell-1} \iota_{2}$, we get { $\iota_{1}\tau^{\ell-1}\iota_{2}s(R_{1})=\tau^{-\ell}s(R_{1})= R_{1}$}. Apply $\iota_{1}$ in the both sides of the equality gives  $\tau^{\ell-1}\iota_{2}s(R_{1})=\iota_{1}(R_{1})$. With \eqref{eq:siotaiotas}, we obtain 
 $\tau^{\ell-1}s\iota_{1}(R_{1})=\iota_{1}(R_{1})$. Note that $\iota_{1}(R_{1})$ belongs to $\Etproj(\Q(t))$ in virtue of Proposition~\ref{prop:galoisactionpointsellipticcurve}. The induction hypothesis leads to the desired result. 
\end{proof}

\begin{lemma}\label{lem:IIIfixeds}
We assume that $t \in \C \setminus \Qbar$. 
For the walk \ju{$\walk{III}$}, 
there are no $R_{1},R_{2}\in \Etproj(\Q(t))$ such that $s(R_{1})=R_{2}$ and $R_{1}\sim R_{2}$. 
\end{lemma}

\begin{proof}
According to Lemma \ref{lem:fixedsandsim}, it is sufficient to prove that $s$ does not have fixed points in $\Etproj(\Q(t))$.  Suppose to  the contrary that there exists $P \in \Etproj(\Q(t))$ such that $s(P)=P$. So,  $P=(x,x)$ for some $x=[x_{0}:x_{1}] \in \P1(\Q(t))$ such that   
\begin{equation}\label{eq:eqseqsbisetbla}
x_{0}^{2}x_{1}^{2}-t(x_{1}^{4}+2x_{1}^{3}x_{0}+x_{0}^{4})=0.
\end{equation}
If $x_{0}=0$, then we can assume that $x_{1}=1$ and we see that \eqref{eq:eqseqsbisetbla} is impossible. 
Assume that $x_{0} \neq 0$. Then, we can and will assume that $x_{0}=1$ and we have  
\begin{equation}\label{eq:eqseqsbis}
x_{1}^{2}=t(1+2x_{1}^{3}+x_{1}^{4}).
\end{equation}
Since $t$ is transcendental, we can and will identify $\Q(t)$ with a field of rational functions.
If $x_{1}$ has a pole of order $\mu \geq 1$ at some $t=t_{0} \in \Qbar$, then $t(1+2x_{1}^{3}+x_{1}^{4})$ has a pole of order $4\mu$ or $4\mu -1$ at $t_{0}$ (depending on whether $t_{0}$ is equal to $0$) whereas $x_{1}^{2}$ has a pole of order $2\mu$ at $t_{0}$. Equation \eqref{eq:eqseqsbis} yields $2 \mu = 4\mu$ or $4\mu-1$, whence a contradiction. 

If $x_{1}$ vanishes at some $t=t_{0} \in \Qbar$, then the equality \eqref{eq:eqseqsbis} specialized at $t=t_{0}$ gives $0=t_{0}$. 

Therefore, we have proved that $x_{1} = ct^{m}$ for some $c \in \Qbar^{\times}$ and $m \in \Z_{\geq 0}$. Equation \eqref{eq:eqseqsbis} becomes 
$$
c^{2}t^{2m}=t(1+2c^{3}t^{3m}+c^{4}t^{4m}).
$$
Equating the degrees of both sides of this equation, we get $m=0$. Now, the equality  
$c^{2}=t(1+2c^{3}+c^{4})$ and $t \in \C \setminus \Qbar$ implies that $c=0$, whence a contradiction. 

\end{proof}

\begin{thm}\label{thm:caseIII}
For the walk \ju{$\walk{III}$}  and  for any $t \in ]0,1/\vert \mathcal{D} \vert[ \setminus \Qbar$ such that $G_t$ is infinite, $F^1(x,t)$ and $F^2(y,t)$ are hypertranscendental.
\end{thm}

\begin{proof}
Here, we have \ju{(using Lemma~\ref{lem:fixedbasepointsinvolution} for instance)}:
$$
\begin{array}{lll}
P_1=P_2= ([1:0],[0:1]) &\text{ and } &Q_1=Q_2=([0:1],[1:0]), \\
\iota_1(Q_1) = ([0:1],[-1:1])&\text{ and }  &\iota_2(P_1)= ([-1:1],[0:1]).
\end{array}
$$ 
The formula \eqref{equ:secondmembreform} applied in this setting gives 
\begin{equation}
(b_{2})^2=\frac{x_0^2}{x_1^2}\frac{\Delta^x_{[x_0:x_1]}}{(x_0^2t)^2}.
\end{equation}
Since the curve $\Etproj$ is nonsingular, the point $P_1$ is a simple zero of $\Delta^x_{[x_0:x_1]}$ seen as a rational function on $\P1(\C)$ (see Proposition \ref{prop:genuscurvewalk} and Corollary \ref{cor:genuscurvewalk}). Then, it is easily seen that the polar divisor of $b_{2}$ is $[P_1] +[Q_1] +[\iota_1(Q_1)]=[P_1]+[Q_1]+[\tau(Q_1)]$. Using Corollary \ref{coro: hyperalg F} and Proposition \ref{prop:caract hyperalg F1 and F2}, we see that, in order to conclude the proof, it is sufficient to prove that $P_1 \nsim Q_1$. Since $s(P_1)=Q_1$, this follows from Lemma \ref{lem:IIIfixeds}.
\end{proof}

\subsubsection{The walk \ju{$\walk{IIC.3}$}}\label{sec:caseIIC3}

We shall exploit the fact that this walk is symmetric in the sense of Definition \ref{def:symmwalks}.

\begin{lemma}\label{lem:3IICfixeds}
We assume that $t \in \C \setminus \Qbar$. 
For the walk \ju{$\walk{IIC.3}$},  
there are no $R_{1},R_{2}\in \Etproj(\Q(t))$ such that $s(R_{1})=R_{2}$ and $R_{1}\sim R_{2}$. 
\end{lemma}

\begin{proof}
According to Lemma \ref{lem:fixedsandsim}, it is sufficient to prove that $s$ does not have fixed points in $\Etproj(\Q(t))$.  Suppose to the contrary that there exists $P \in \Etproj(\Q(t))$ such that $s(P)=P$. So,  $P=(x,x)$ for some $x=[x_{0}:x_{1}] \in \P1(\Q(t))$ such that   
$$
x_{0}^{2}x_{1}^{2}-t(x_{1}^{4}+2x_{0}^{3}x_{1}+x_{0}^{4})=0.
$$
If $x_{1}=0$, then we can assume that $x_{0}=1$ and we see that \eqref{eq:eqseqs} is impossible. 
Assume that $x_{1} \neq 0$. Then, we can and will assume that $x_{1}=1$ and we have  
\begin{equation}\label{eq:eqseqs}
x_{0}^{2}=t(1+2x_{0}^{3}+x_{0}^{4}).
\end{equation}

As we can see in the proof of Lemma \ref{lem:IIIfixeds}, there are no such $x_{0}\in \Q(t)$, whence a contradiction.
\end{proof}

\begin{thm}\label{thm:caseIIC3}
For the  walk \ju{$\walk{IIC.3}$} and for any  $t \in ]0,1/\vert \mathcal{D} \vert[ \setminus \Qbar$  such that $G_t$ is infinite, we have that  $F^1(x,t)$ and $F^2(y,t)$ are hypertranscendental.
\end{thm}

\begin{proof}
In this case, we have 
\begin{itemize}
\item  $P_{1} = ([1:0],[0:1])$;

\item $P_{2} = ([1:0],[-1:1])=\iota_1(P_1)$;
\item $\iota_{2}(P_{1})=P_{1}$;
\item $ Q_{1}=([0:1],[1:0])$;
\item $Q_{2} = ([-1:1],[1:0])=\iota_2(Q_1)$;
\item $\iota_{1}(Q_{1}) = ([0:1],[1:0])=Q_{1}$;
\item $\iota_{1}(Q_{2}) = ([-1:1],[-t:t+1])$.
\end{itemize}
The polar divisor of $b_{2}$ is 
$$
[P_{1}]+[P_{2}]+[Q_{2}]+[\iota_{1}(Q_{2})].
$$ 
The poles $Q_{2}$ and $\iota_{1}(Q_{2})$ belong to the same $\tau$-orbit: 

$$\iota_{1}(Q_{2})\xrightarrow{\iota_{1}} Q_{2} \xrightarrow{\iota_{2}} Q_{1}\xrightarrow{\iota_{1}} Q_{1} \xrightarrow{\iota_{2}} Q_{2}.$$
Similarly, the poles $P_{2}$ and {$P_1$} belong to the same $\tau$-orbit:
$$P_{2} \xrightarrow{\iota_{1}} P_{1}\xrightarrow{\iota_{2}} P_{1}.$$
Since $s(Q_{1})=P_{1}$, Lemma \ref{lem:3IICfixeds} implies that the $\tau$-orbits of $\iota_{1}(Q_{2})$ and $P_{2}$ are distinct. We refer to Figure \ref{fig3} for a summary of these facts.

Since $\iota_1(b_2)=-b_2$, Lemma \ref{lem1} ensures that the residue of $b_2 \omega$ at $P_{1}$ and $P_{2}$ are equal (and they are non zero). Therefore, the sum of the residues of $b_{2} (\omega)$ on the $\tau$-orbit of $P_{2}$ is nonzero. Lemma \ref{lem:lemhregandresj} together with Remark \ref{rmk:ores1vanishiffResvanish} and Proposition~\ref{Prop0} lead to the desired result. 
\end{proof}

\begin{figure}
\begin{tabular}{|l|ccc|}
\hline
\hbox{Walk} && \ju{$\walk{IIC.3}$} \begin{tikzpicture}[scale=.4, baseline=(current bounding box.center)]
\foreach \x in {-1,0,1} \foreach \y in {-1,0,1} \fill(\x,\y) circle[radius=2pt];
\draw[thick,->](0,0)--(1,1);
\draw[thick,->](0,0)--(0,1);
\draw[thick,->](0,0)--(1,0);
\draw[thick,->](0,0)--(-1,-1);
\end{tikzpicture}&\\ \hline
\hbox{Polar divisor of } $b_{2}$ &$([-1:1],[-t:t+1])$&&$([1:0],[-1:1])$\\
&$+([-1:1],[1:0])$&&$+([1:0],[0:1])$\\
\hline 
$\tau$-\hbox{Orbit of the poles of } $b_{2}$ &$([-1:1],[-t:t+1])$&&$([1:0],[-1:1])$\\
&$\downarrow \tau$&&$\downarrow \tau$\\
&$([0:1],[1:0])$&$\not\sim$&$([1:0],[0:1])$\\
&$\downarrow \tau$&&\\
&$([-1:1],[1:0])$&& \\ \hline
\end{tabular}
\caption{$\tau$-Orbit of the poles for \ju{$\walk{IIC.3}$}}\label{fig3}
\end{figure}

In conclusion, we have shown that the series $F^1(x,t)$ and $F^2(y,t)$ for the walks in \ju{$\walksg$} are hypertranscendental. \toto{The next section is devoted to the study of the walks in $\walkse$.}

\section{\ju{Hyperalgebraicity of the generating series of the walks in $\walkse$}}\label{sec:hyperalg} \ju{We will show below
 that the following is true for any walk in $\walkse$~:}

\begin{condition}\label{condition:linde}For $i=1$ or $i=2$ there exist  an integer $n \geq 0$, $c_0,\ldots,c_{n-1} \in \C$ and $g \in \C(\Etproj)$ \ju{ such that }
 \begin{equation}\label{eq:lindiffc}
  \delta^n(b_i)+c_{n-1}\delta^{n-1}(b_i)+\cdots+c_{1}\delta(b_{i})+c_{0} b_i = \tau(g)-g. 
 \end{equation}
\end{condition}

As we have noticed in the proof Proposition~\ref{prop:caract hyperalg F1 and F2}, if Condition~\ref{condition:linde} is satisfied for $i=1$ or $i=2$, then it is satisfied for both values. Condition~\ref{condition:linde} precludes the  possibility of using Proposition~\ref{prop:caract hyperalg F1} to show that the corresponding generating series are hypertranscendental.  As we have already mentioned, this does not immediately   imply that the series $F^1$ and $F^2$ are hyperalgebraic.  Nonetheless, we can use this data together with properties of the  related $r_x$ and $r_y$ to show that they are indeed hyperalgebraic.  \ju{We are going to prove} the following result.

\begin{prop}\label{prop:hypalg} If a walk satisfies Assumption~\ref{assumption:generalcase t and D} and Condition~\ref{condition:linde} then the corresponding $F^1(x,t)$ and $F^2(y,t)$ are hyperalgebraic over $\C$.
\end{prop}

\ju{Using Proposition~\ref{prop:hypalg}, we are now able to prove Theorem~\ref{thm2}.}

\begin{thm2proof} We will show below, \ju{in Subsections \ref{subsec:iic1}, \ref{subsec:iib1} and \ref{subsec:iic5}}, that \ju{the nine walks in $\walkse$ satisfy} Condition~\ref{condition:linde}. Proposition~\ref{prop:hypalg} \ju{together with Assumption~\ref{assumption:generalcase t and D} then imply} that for these walks the corresponding series $F^1(x,t) = Q_{\mathcal D}(x,0,t)$ and $F^2(y,t)= Q_{\mathcal D}(0,y,t)$ are hyperalgebraic with respect to $x$ and $y$.  The functional equation (\ref{eq:funcequ}) for $Q_{\mathcal D}(x,y,t)$ can be rewritten as 
\[Q_{\mathcal{D}}(x,y,t)= \frac{1}{K_{\mathcal{D}}(x,y,t)}[xy-F_{\mathcal{D}}^{1}(x,t) -F_{\mathcal{D}}^{2}(y,t)+td_{-1,-1} Q_{\mathcal{D}}(0,0,t)].\]
Each term on the right hand side of this equation is both $x$- and $y$-hyperalgebraic. Since the property of hyperalgebraicity is closed under field operations, $Q_{\mathcal{D}}(x,y,t)$ is also $x$- and $y$-hyperalgebraic.\end{thm2proof}

 We will need the following lemmas to prove Proposition~\ref{prop:hypalg}. 

\begin{lem}\label{lem:composit} Let $U,V$ be open subsets of $\C$ and $f:V\rightarrow \C,  g:U\rightarrow V$ be functions analytic in their domains.  If $f(x)$ and $g(x)$ are hyperalgebraic over $\C$, then so is $f(g(x))$.
\end{lem}

\begin{proof} One easily checks that a function $h$ is hyperalgebraic over $\C$ if and only if  the field $\C(h(x), h'(x),  \ldots , h^{(n)}(x), \ldots)$ has finite transcendence degree over $\C$. Since $f$ is hyperalgebraic, the field $\C(f(g(x)), f'(g(x)), \ldots, f^{(n)}(g(x)), \ldots)$ has finite transcendence degree over $\C$. Fa\`a di Bruno's formula {\mfs \cite[p.~33]{Jordan}} for the derivative of a composite function shows that for all $m$ $$(f(g(x))^{(m)} \in \C(f(g(x)), f'(g(x)), \ldots, g(x), g'(x), \ldots)$$ which is of finite transcendence degree over $\C$.  \end{proof}

We note that similar techniques show that if $f$ and $g$ are hyperalgebraic over $\C$ then so is any element of $\C(f, f', \ldots , g, g', \ldots)$.

\begin{lem}\label{lem:reciprochyperalg}
Let $h : U \rightarrow V$ be a biholomorphism between open subsets of $\C$. If $h$ is hyperalgebraic over $\C$, then $h^{-1}$ is hyperalgebraic over $\C$. 
\end{lem}

\begin{proof}
We know that $\C(h,h',\ldots,h^{(k)},\ldots)$ has finite transcendence degree over $\C$, so $\C(h \circ h^{-1},h'\circ h^{-1},\ldots,h^{(k)}\circ h^{-1},\ldots)$ has finite transcendence degree over $\C$ as well. But, the successive derivatives of $h^{-1}$ are rational fractions in the $h^{(k)} \circ h^{-1}$ ($k \in \Z_{\geq 0}$), so $\C((h^{-1})',\ldots,(h^{-1})^{(k)},\ldots)$ is a subfield of $\C(h'\circ h^{-1},\ldots,h^{(k)}\circ h^{-1},\ldots)$ and, hence, has finite transcendence degree over $\C$. 
\end{proof}

\begin{prf}{Proof of Proposition~\ref{prop:hypalg}}
 We will prove this for $F^1(x,t)$; the other case is similar. We set 
$$
L = \delta^n +c_{n-1}\delta^{n-1} +\cdots+c_{1}\delta +c_{0}. 
$$
Let $g \in \C(\Etproj)$ be as in Condition~\ref{condition:linde}. We then have 
 \begin{eqnarray*}
 \tau(L(r_x )-g) - (L(r_x) - g) &=& L(\tau(r_x) -r_x) - (\tau(g) - g)\\
  & = & L(b_1) - (\tau(g) - g)\\
  & = & 0. 
  \end{eqnarray*}
  Therefore $L(r_x )-g$ is $\tau$-invariant.  From  Assumption~\ref{assumption:generalcase t and D}, we have that $r_x$ is invariant under $\omega \mapsto \omega + \omega_1$ and since $g \in \C(\Etproj)$ the same is true for $g$. Therefore ${L(r_x )-g = R(\mathfrak{p}_{1,3},\mathfrak{p}'_{1,3})}$ where $R$ is a rational function of two variables and $\mathfrak{p}_{1,3}$ is the Weierstrass $\mathfrak{p}$-function with periods $\omega_1$ and $\omega_3$.  Since $\mathfrak{p}_{1,3}$ is hyperalgebraic over $\C$, we have that $L(r_x )-g$ is hyperalgebraic over $\C$.  The element $g$ is a rational function of  a  Weierstass $\mathfrak{p}$-function $\mathfrak{p}_{1,2}$ with periods $\omega_1$ and $\omega_2$ and its derivative, so it is also hyperalgebraic over $\C$. Therefore $L(r_x)$ is hyperalgebraic over $\C$ and thus the same holds for $r_x$.  By definition, for some open set $U$, we have that $r_x(\omega) = F^1(\mathfrak{q}(\omega),t)$ where $\mathfrak{q}$ is a rational function of $\mathfrak{p}_{1,2}$ and $\mathfrak{p}_{1,2}'$. 
Let $U' \subset U$ and $V'$ be nonempty open subsets of $\C$ such that $\mathfrak{p}_{1,2}$ induces a biholomorphism $U \rightarrow V$. 
Then, on $V$, we have $F^{1}(x,t)=r_{x}(\mathfrak{p}_{1,2}^{-1}(x))$ and we deduce from Lemma~\ref{lem:reciprochyperalg} and Lemma \ref{lem:composit} that $F^{1}(x,t)$ is hyperalgebraic over $\C$. 
\end{prf} 
  
  The remainder of this section is devoted to showing that the \ju{walks in $\walkse$ satisfy}  Condition~\ref{condition:linde} for $b_2$. These proofs rely heavily on the alternate characterizations of Condition~\ref{condition:linde} for $b_2$ given in the appendix \ju{({\it i.e.}, Proposition~\ref{Prop0}, Corollary~\ref{cor2}, Lemma \ref{lem:lemhregandresj} and Remark~\ref{rmk:ores1vanishiffResvanish})}. These characterizations are just in terms of the $\tau$-orbits of the poles of $b_2$.
 
 \ju{We list here the sections where the treatments of the walks in $\walkse$ are finalized:
\begin{itemize}
\item for \ju{$\walk{IIC.1}$}, \ju{$\walk{IIC.2}$} and \ju{$\walk{IIC.4}$}, see Section \ref{subsec:iic1}, Theorem~\ref{thm:caseIIC};
\item for \ju{$\walk{IIB.1}$}, \ju{$\walk{IIB.2}$}, \ju{$\walk{IIB.3}$}, \ju{$\walk{IIB.6}$} and \ju{$\walk{IIB.7}$}, see Section~\ref{subsec:iib1}, Theorem~\ref{thm:caseIIBbis}; 
\item for \ju{$\walk{IIC.5}$}, see Section \ref{subsec:iic5}, Theorem~\ref{thm:caseIIC5}.
\end{itemize}}
  
In all cases we assume that $t$ has been chosen such that  Assumption~\ref{assumption:generalcase t and D} holds. 
  
\subsection{The walks \ju{$\walk{IIC.1}$}, \ju{$\walk{IIC.2}$} and \ju{$\walk{IIC.4}$} }\label{subsec:iic1}

\begin{thm}\label{thm:caseIIC}
The walks \ju{$\walk{IIC.1}$}, \ju{$\walk{IIC.2}$} and \ju{$\walk{IIC.4}$} satisfy Condition~\ref{condition:linde}.
\end{thm}

\begin{proof}
We shall first give a detailed proof for \ju{$\walk{IIC.1}$}. 
In this case, we have {
$$
\overline{K}(x_{0},x_{1},y_{0},y_{1},t) = - ty_0y_1x_1^2-tx_0x_1y_1^2+x_0x_1y_0y_1-tx_0x_1y_0^2-tx_0^2y_0^2
$$}
and 
\begin{eqnarray*}
P_{1}=([1:0],[0:1]), & P_2 =\iota_{1}(P_{1})=P_{1}=([1:0],[0:1]), \\ 
Q_1=([-1:1],[1:0]), & Q_2=\iota_{2}(Q_{1})=([0:1],[1:0]) \neq Q_{1}.
\end{eqnarray*} 
We see that 
\begin{itemize}
\item $P_{1}$ is the only pole of $x$ and it has order $2$;
\item the poles of $y$ are $Q_{1}$ and $Q_{2}$ and they have order $1$;  
\item the poles of $\iota_1(y)$  are $\iota_1(Q_1)=([-1:1],[\frac{t}{t+1}:1]),\iota_1(Q_2)=([0:1],[0:1]),$ and they have order $1$;
\item since $Q_{1},Q_{2},\iota_1(Q_1),\iota_1(Q_2)$ are two by two distinct, these four points are the poles of $\iota_{1}(y)-y$ and they all have order $1$;
\item we have, see \eqref{equ:secondmembreform}, $(b_{2})^2=\frac{\Delta^x_{[x_0:x_1]}}{x_1^2(x_{0}+x_{1})^2 }$, but $P_{1}$ is a zero of order $2$ of $\frac{\Delta^x_{[x_0:x_1]}}{x_{0}^{4}}$ and a zero of order $4$ of $\left(\frac{x_{1}}{x_{0}}\right)^{2}$, so $P_{1}$ is a pole of order $2$ of $(b_{2})^{2}$, and hence of order $1$ of $b_{2}$;
\item $Q_2$ and $\iota_1(Q_2)$ are zeros of $x$;
\item $Q_1$, and $\iota_1(Q_1)$ are not  zeros of  $x$.
\end{itemize}
Finally, the polar divisor of $b_{2}$ is $[P_1]+[Q_1]+[\iota_1(Q_1)]$ and $P_1,Q_1,\iota_1(Q_1)$ are two by two distinct. 

Moreover, the first $4$ elements of the orbit of $\iota_{1}(Q_{1})$ by the iterated action of $\tau$ are given by :
$$
\xymatrix{
    & \iota_{1}(Q_{1}) = ([-1:1],[\frac{t}{t+1}:1])  \ar[ld]^{\iota_{1}} \ar[d]_{\tau} \\
    ([-1:1],[1:0]) \ar[r]_{\iota_{2}} & Q_{2}= ([0:1],[1:0]) \ar[ld]^{\iota_{1}} \ar[d]_{\tau}\\
    ([0:1],[0:1]) \ar[r]_{\iota_{2}}  & P_{1}= ([1:0],[0:1]) \ar[ld]^{\iota_{1}} \ar[d]_{\tau} \\
   ([1:0],[0:1])  \ar[r]_{\iota_{2}}  & \iota_{1}(Q_{2})=([0:1],[0:1]) \ar[ld]^{\iota_{1}} \ar[d]_{\tau}  \\
  ([0:1],[1:0]) \ar[r]_{\iota_{2}}  &  Q_{1} = ([-1:1],[1:0]) 
  }.
  $$
Therefore, all the poles of $b_{2}$  belong to the same $\tau$-orbit. In summary, $b_{2}$ has only simple poles, and they all belong to the same $\tau$-orbit. The result is now a direct consequence of Corollary \ref{cor2}.  

The other cases are similar. The polar divisor of $b_{2}$ and the first few terms of the $\tau$-orbit of one of the poles of $b_{2}$ in the remaining cases are listed in Figure \ref{fig:listwalksdivorb}. \end{proof}

\begin{figure}
\begin{tabular}{|l|c|c|c|c|}
\hline 
\hbox{Walk}&
\ju{$\walk{IIC.1}$}
\begin{tikzpicture}[scale=.4, baseline=(current bounding box.center)]
\foreach \x in {-1,0,1} \foreach \y in {-1,0,1} \fill(\x,\y) circle[radius=2pt];
\draw[thick,->](0,0)--(1,1);
\draw[thick,->](0,0)--(0,1);
\draw[thick,->](0,0)--(0,1);
\draw[thick,->](0,0)--(-1,0);
\draw[thick,->](0,0)--(0,-1);
\end{tikzpicture}&
\ju{$\walk{IIC.2}$}
\begin{tikzpicture}[scale=.4, baseline=(current bounding box.center)]
\foreach \x in {-1,0,1} \foreach \y in {-1,0,1} \fill(\x,\y) circle[radius=2pt];
\draw[thick,->](0,0)--(1,1);
\draw[thick,->](0,0)--(0,1);
\draw[thick,->](0,0)--(-1,0);
\draw[thick,->](0,0)--(-1,-1);
\draw[thick,->](0,0)--(0,-1);
\end{tikzpicture}&
\ju{$\walk{IIC.4}$}
\begin{tikzpicture}[scale=.4, baseline=(current bounding box.center)]
\foreach \x in {-1,0,1} \foreach \y in {-1,0,1} \fill(\x,\y) circle[radius=2pt];
\draw[thick,->](0,0)--(1,1);
\draw[thick,->](0,0)--(0,1);
\draw[thick,->](0,0)--(1,0);
\draw[thick,->](0,0)--(-1,-1);
\draw[thick,->](0,0)--(-1,0);
\end{tikzpicture}&
\ju{$\walk{IIB.7}$}
\begin{tikzpicture}[scale=.4, baseline=(current bounding box.center)]
\foreach \x in {-1,0,1} \foreach \y in {-1,0,1} \fill(\x,\y) circle[radius=2pt];
\draw[thick,->](0,0)--(0,1);
\draw[thick,->](0,0)--(-1,1);
\draw[thick,->](0,0)--(1,-1);
\draw[thick,->](0,0)--(1,0);
\draw[thick,->](0,0)--(0,-1);
\end{tikzpicture}\\ \hline 

\hbox{Polar divisor of $b_{2}$} &
$([-1:1],[\frac{t}{t+1}:1])$&
$([-1:1],[1:0])$&
$([-1:1],[\frac{-t}{2t+1}:1])$&
$2([1:0],[-1:1])$ \\
&$+([1:0],[0:1])$
&$+([1:0],[0:1])$
&$+([1:0],[-1:1])$
&$+([-1:1],[1:0])$\\
&$+([-1:1],[1:0])$
&$+([-1:1],[0:1])$
&$+([1:0],[0:1])$
&$+([-1:1],[0:1])$\\
&&&$+([-1:1],[1:0])$&$+2([1:0],[1:0])$\\ \hline 
\hbox{$\tau$-Orbit of one of}
&$([-1:1],[\frac{t}{t+1}:1])$
&$([-1:1],[1:0])$
&$([-1:1],[\frac{-t}{2t+1}:1])$
&$([1:0],[-1:1])$\\
\hbox{the poles of $b_{2}$} &$\downarrow \tau$&$\downarrow \tau$&$\downarrow \tau$&$\downarrow \tau$\\
&$([0:1],[1:0])$
&$([1:0],[0:1])$
&$([0:1],[1:0])$
&$([-1:1],[1:0])$\\
&$\downarrow \tau$&$\downarrow \tau$&$\downarrow \tau$&$\downarrow \tau$\\
&$([1:0],[0:1])$
&$([-1:1],[0:1])$
&$([1:0],[-1:1])$
&$([0:1],[0:1])$\\
&$\downarrow \tau$&&$\downarrow \tau$&$\downarrow \tau$\\
&$([0:1],[0:1])$
&&$([1:0],[0:1])$
&$([-1:1],[0:1])$\\
&$\downarrow \tau$&&$\downarrow \tau$&$\downarrow \tau$\\
&$([-1:1],[1:0])$
&&$([0:1],[-1:1])$
&$([1:0],[1:0])$\\
&&&$\downarrow \tau$&\\
&&&$([-1:1],[1:0])$& \\ \hline
\hbox{Walk}&
\ju{$\walk{IIB.1}$}
\begin{tikzpicture}[scale=.4, baseline=(current bounding box.center)]
\foreach \x in {-1,0,1} \foreach \y in {-1,0,1} \fill(\x,\y) circle[radius=2pt];
\draw[thick,->](0,0)--(1,0);
\draw[thick,->](0,0)--(0,1);
\draw[thick,->](0,0)--(0,-1);
\draw[thick,->](0,0)--(-1,-1);
\end{tikzpicture}&
\ju{$\walk{IIB.2}$}
\begin{tikzpicture}[scale=.4, baseline=(current bounding box.center)]
\foreach \x in {-1,0,1} \foreach \y in {-1,0,1} \fill(\x,\y) circle[radius=2pt];
\draw[thick,->](0,0)--(1,0);
\draw[thick,->](0,0)--(0,1);
\draw[thick,->](0,0)--(0,1);
\draw[thick,->](0,0)--(-1,-1);
\draw[thick,->](0,0)--(1,-1);
\end{tikzpicture}&
\ju{$\walk{IIB.3}$}
\begin{tikzpicture}[scale=.4, baseline=(current bounding box.center)]
\foreach \x in {-1,0,1} \foreach \y in {-1,0,1} \fill(\x,\y) circle[radius=2pt];
\draw[thick,->](0,0)--(1,0);
\draw[thick,->](0,0)--(0,1);
\draw[thick,->](0,0)--(-1,0);
\draw[thick,->](0,0)--(1,-1);
\end{tikzpicture}&
\ju{$\walk{IIB.6}$}
\begin{tikzpicture}[scale=.4, baseline=(current bounding box.center)]
\foreach \x in {-1,0,1} \foreach \y in {-1,0,1} \fill(\x,\y) circle[radius=2pt];
\draw[thick,->](0,0)--(1,0);
\draw[thick,->](0,0)--(0,1);
\draw[thick,->](0,0)--(1,-1);
\draw[thick,->](0,0)--(-1,-1);
\draw[thick,->](0,0)--(0,-1);
\end{tikzpicture} \\ \hline 
\hbox{Polar divisor of $b_{2}$} 
&$2([1:0],[0:1])$
&$2([1:0],[-1:1])$
&$2([1:0],[-1:1])$
&$2([1:0],[-1:1])$\\
&$+2([1:0],[1:0])$
&$+2([1:0],[1:0])$
&$+2([1:0],[1:0])$
&$+2([1:0],[1:0])$\\
&
&
&$$
&$$\\ \hline 

\hbox{$\tau$-Orbit of one of}
&$([1:0],[0:1])$
&$([1:0],[-1:1])$
&$([1:0],[-1:1])$
&$([1:0],[-1:1])$\\
\hbox{the poles of $b_{2}$}&$\downarrow \tau$&$\downarrow \tau$&$\downarrow \tau$&$\downarrow \tau$\\
&$([0:1],[1:0])$
&$([0:1],[1:0])$
&$([0:1],[1:0])$
&$([0:1],[1:0])$\\
&$\downarrow \tau$&$\downarrow \tau$&$\downarrow \tau$&$\downarrow \tau$\\
&$([1:0],[1:0])$
&$([1:0],[1:0])$
&$([1:0],[1:0])$
&$([1:0],[1:0])$\\
 \hline 
\end{tabular}
\caption{Polar divisors of $b_{2}$ in cases  when all poles belong to  the same $\tau$-orbit \ju{(used in Section~\ref{subsec:iic1})}}
\label{fig:listwalksdivorb}
\end{figure}

\subsection{The walks \ju{$\walk{IIB.1}$}, \ju{$\walk{IIB.2}$}, \ju{$\walk{IIB.3}$}, \ju{$\walk{IIB.6}$} and \ju{$\walk{IIB.7}$}} \label{subsec:iib1}

\begin{thm}\label{thm:caseIIBbis}
The walks \ju{$\walk{IIB.1}$}, \ju{$\walk{IIB.2}$}, \ju{$\walk{IIB.3}$}, \ju{$\walk{IIB.6}$} and \ju{$\walk{IIB.7}$} satisfy Condition~\ref{condition:linde}.
\end{thm}

\begin{proof}
In the 5 cases, 
\begin{itemize}
\item $([1:0],[1:0])$ is a double pole of $b_{2}$,
\item  there exists $\alpha \in \{0,-1\}$ such that $([1:0],[\alpha:1])$ is a double pole of $b_{2}$,
\item $b_{2}$ has at most two  simple poles.\end{itemize}

 Furthermore, every pole belong to the same $\tau$-orbit, see Figure \ref{fig:listwalksdivorb}. We consider a set of analytic local parameters as given by Lemma \ref{lem:horesidues} and we use the same notations. Since $\iota_1(b_2)=-b_2$,  Lemma~\ref{lem:horesidues} ensures that $\ores_{([1:0],[1:0]),2}(b_{2})=0$. Moreover, since every pole of $b_{2}$ belong to the same $\tau$-orbit, Lemma \ref{lem:res0order2} ensures that $\ores_{Q,1}(b_{2})=0$ for all $Q$.  Lemma \ref{lem:lemhregandresj} together with Proposition \ref{Prop0} lead to the desired result.
\end{proof}

\subsection{The walk \ju{$\walk{IIC.5}$}}\label{subsec:iic5} In this case, unlike the previous cases, the poles of $b_2$ form two distinct $\tau$-orbits.  To proceed, we will need the following

\begin{lemma}\label{lemma:iota1andsim}
Let $R_{1},R_{2} \in \Etproj(\Q(t))$ be such that $\iota_{1}(R_{1})=R_{2}$ and $R_{1}\sim R_{2}$. Then, there exist  $j\in \{1,2\}$ and $R\in \Etproj(\Q(t))$ such that $\iota_{j}(R)=R$.
\end{lemma}

\begin{proof}
The reasonning is similar to a part of the proof of Theorem \ref{thm:caseIIB}.

It is easily seen that the equality $\tau^n(R_1)=R_2$, together with the fact that ${\tau=\iota_2 \circ \iota_1}$ is the composition of two involutions, imply, if $n=2k$, that ${\iota_1(\tau^k(R_1))=\tau^k(R_1)}$, and, if $n=2k +1$, that $\iota_2(\tau^k(\iota_2(R_2)))=\tau^k(\iota_2(R_2))$. 
Proposition \ref{prop:galoisactionpointsellipticcurve} ensures that both $\tau^k(R_1)$ and $\tau^k(\iota_2(R_2))$ belong to $\Etproj(\Q(t))$. Whence the desired result. 
\end{proof}

\begin{figure}
\begin{tabular}{|l|ccc|}
\hline
\hbox{Walk} &&\ju{$\walk{IIC.5}$} \begin{tikzpicture}[scale=.4, baseline=(current bounding box.center)]
\foreach \x in {-1,0,1} \foreach \y in {-1,0,1} \fill(\x,\y) circle[radius=2pt];
\draw[thick,->](0,0)--(0,1);
\draw[thick,->](0,0)--(1,1);
\draw[thick,->](0,0)--(-1,0);
\draw[thick,->](0,0)--(1,0);
\draw[thick,->](0,0)--(0,-1);
\end{tikzpicture}&\\ \hline
\hbox{Polar divisor of }$b_{2}$ &$([-1:1],[t:2t+1])$&&$([1:0],[-1:1])$\\
&$+([1:0],[0:1])$&&$+([-1:1],[1:0])$\\ \hline 
$\tau$-\hbox{Orbit of the poles of } $b_{2}$&$([-1:1],[t:2t+1])$&&$([1:0],[-1:1])$\\
&$\downarrow \tau$&&$\downarrow \tau$\\
&$([0:1],[1:0])$&$\not\sim$&$([0:1],[0:1])$\\
&$\downarrow \tau$&&$\downarrow \tau$\\
&$([1:0],[0:1])$&&$([-1:1],[1:0])$ \\ \hline
\end{tabular}
\caption{$\tau$-Orbit of the poles in the case \ju{$\walk{IIC.5}$}}\label{fig2}
\end{figure}

\begin{thm}\label{thm:caseIIC5}
The walk \ju{$\walk{IIC.5}$} satisfies Condition~\ref{condition:linde}.
\end{thm}

\begin{proof}
We have 
\begin{itemize}
\item $P_{1} = ([1:0],[0:1])$;
\item $P_{2} = ([1:0],[-1:1])$;
\item $Q_{1}=([0:1],[1:0])$;
\item $Q_{2} = ([-1:1],[1:0])$;
\item $\iota_{1}(Q_{1}) = ([0:1],[0:1])$;
\item $ \iota_{1}(Q_{2}) = ([-1:1],[t:2t+1])$.
\end{itemize}

The polar divisor of $b_{2}$ is 
$$
[P_{1}]+[P_{2}]+[Q_{2}]+[\iota_{1}(Q_{2})].
$$ 
The poles $P_{1}$ and $\iota_{1}(Q_{2})$ belong to the same $\tau$-orbit: 
$$\iota_{1}(Q_{2})=([-1:1],[t:2t+1]) \xrightarrow{\tau} ([0:1],[1:0]) \xrightarrow{\tau} P_{1}=([1:0],[0:1])$$
Similarly, the poles $P_{2}$ and $Q_{2}$ belong to the same $\tau$-orbit:
$$P_{2}=([1:0],[-1:1]) \xrightarrow{\tau} ([0:1],[0:1])\xrightarrow{\tau} Q_{2}=([-1:1],[1:0]).$$
Moreover, the $\tau$-orbits of $\iota_{1}(Q_{2})$ and $P_{2}$ are distinct. Indeed, otherwise, since we have $$
\iota_{1}([0:1],[1:0])=([0:1],[0:1]),
$$ 
Lemma \ref{lemma:iota1andsim} would imply the existence of $R \in \Etproj(\Q(t))$ such that $\iota_{j}(R)=R$ for some $j \in \{1,2\}$. But, Lemma \ref{lemma:fixedpointinvolutionnotrational} ensures that $R$ is a base point and Lemma \ref{lem:fixedbasepointsinvolution} ensures that none of the base points are fixed by $\iota_{j}$, whence a contradiction.

Lemma \ref{lem1} ensures that the residues of $b_{2}\omega$ at $P_{1}$ and $P_{2}=\iota_{1}(P_{1})$ are equal and will be denoted by $a$. Similarly, the residues of $b_{2}\omega$ at $Q_{2}$ and $\iota_{1}(Q_{2})$ are the same and will be denoted by $b$. Since the sum of the residues over $\Etproj$ of $b_{2}(\omega)$ is equal to $0$, we have $2a+2b=0$, so $a+b=0$. Therefore, the sum of the residues of $b_{2}(\omega)$ on any $\tau$-orbit is equal to $0$. 
Lemma \ref{lem:lemhregandresj} together with Remark \ref{rmk:ores1vanishiffResvanish} and Proposition \ref{Prop0} lead to the desired result. 
\end{proof}

\section{Nonholonomicity in the exceptional cases\label{sec:nonholonomic}}

\begin{thm} For each walk in \ju{$\walkse$} the series $F^1(x,t)$ and $F^2(y,t)$ are not holonomic\ju{, {\it i.e.}, they do not satisfy any nontrivial linear differential equation with coefficients in $\C(x)$ and $\C(y)$ respectively}.\end{thm} 
\begin{proof}  We only present the proof for $F^2(y,t)$, the proof for $F^1(x,t)$ being similar.  Assume that $F^2(y,t)$ is holonomic.  
Then, $F^2(y,t)$ could be analytically continued to a multivalued meromorphic function on $\Etproj \backslash \{\mbox{ a finite set of points } \}$.  This implies that $r_y$ would be a meromorphic function on the universal cover of $\Etproj$ whose singular points form a finite set modulo the lattice $\Z\omega_1 + \Z\omega_2$.  \ju{For each walk in $\walkse$}, we have that $F^2(y,t)$ satisfies Condition~\ref{condition:linde}. Therefore as in the proof of Proposition~\ref{prop:hypalg}, there is a $g \in \C(\Etproj)$ such that $\tau(L(r_y) -g) = L(r_y) -g$. Note that the poles of $ L(r_y) -g$ also form a finite set modulo the lattice $\Z\omega_1 + \Z\omega_2$. Since $L(r_y) - g$ is $\tau$-periodic, this set is left invariant by $\omega \mapsto \omega + \omega_{3}$. Using the fact that the reduction of $\omega_{3}$ modulo $\Z\omega_1 + \Z\omega_2$ has infinite order, we get that the set of poles of $L(r_y) - g$ is empty.  
Using   Condition~\ref{condition:linde} again, we see that $L(r_y) -g$ is $\omega_1$-periodic as well as being $\omega_3$ periodic.  Since it has no poles, we  must have that $L(r_y) -g = c \in \C$.\\[0.05in]
We want to prove that this last fact leads to a contradiction. To do this we will use some notation from \cite[Sec.~4.2]{KurkRasch}.  Let $\Delta_x$ be the set of $\omega$ in $\omega_1\R + ]0,\omega_2[ $ corresponding to points  on the elliptic curve with $|x| < 1$ and let $\Delta_y$ be the set of $\omega$ in $\omega_1\R +\omega_3/2 + ]0,\omega_2[ $ corresponding to points  on the elliptic curve with $|y| < 1$.
Let us state and prove two lemmas.

\begin{lem}\label{lem2}
The function $r_y$ has no poles in $\Delta_x \cup \Delta_y$.
\end{lem}
 
 \begin{proof}[Proof of Lemma \ref{lem2}]
From \cite[Theorem~3]{KurkRasch},  one sees that  $r_y$ has no poles on $\Delta_y$ and that  $r_x$ has no poles on $\Delta_x$.
 \par 
Using formula (4.5) in \cite{KurkRasch} and the fact that $r_x$ has no poles on $\Delta_x$, we find that the poles of $r_y$ on $\Delta_x$  are poles of $xy$. It is therefore sufficient to prove that $xy$ does not have poles for $|x|<1$. Note that a pole of $xy$ with $|x|<1$ is a pole of $y$. \par
We claim that $([0:1],[1:0])$ is the only possible pole of $xy$ with  $|x|<1$.
We recall that the poles of $y$ are $Q_{1}$ and $Q_{2}$. Let $\alpha_{1},\alpha_{2}\in \mathbb{P}^{1}(\C)$ be the $x$-coordinates of $Q_{1}$ and $Q_{2}$ respectively. 
To prove the claim it is sufficient to prove that for $j\in \{1,2\}$, $|\alpha_{j}|<1$ implies $\alpha_{j}=0$. The $x$-coordinates $\{\alpha_{1},\alpha_{2}\}$ are the two roots in $\mathbb{P}^{1}(\C)$, counted with multiplicities, of the polynomial $d_{-1,1}+d_{0,1}X+d_{1,1}X^{2}$. For the \ju{walks in $\walkse$}, we have $d_{0,1}=1$. The following array summarizes the possible values of the pair  $(\alpha_{1},\alpha_{2})$ in the four situations: 
 $$\begin{array}{|l|l|l|}\hline 
 &d_{-1,1}=0&d_{-1,1}=1\\ \hline 
d_{1,1}=0&[0:1],[1:0]&[-1:1],[1:0] \\\hline 
d_{1,1}=1&[0:1],[-1:1]&[\frac{-1-i \sqrt{3}}{2}:1],[\frac{-1+i \sqrt{3}}{2}:1]\\ \hline 
\end{array}  $$
In the four situations, we see that $|\alpha_{j}|<1$ implies $\alpha_{j}=0$, proving our claim.\par 
Recall that we are interested in the poles of $xy$ with $|x|<1$. \ju{Thanks to the claim just proved} we have to determine whether $([0:1],[1:0])$ is a pole of $xy$. We have seen in the proof of the claim that \ju{for the walks in $\walkse$}, $Q_{1}\neq Q_{2}$ since their $x$-coordinates are different. So  $y$ has at most \ju{a} simple pole at $([0:1],[1:0])$. This shows that $([0:1],[1:0])$ is not a pole of $xy$.
  Combined with the claim, this shows that $xy$ has no poles for $|x|<1$, proving the lemma.\par 
 \end{proof}
 
\begin{lem}\label{lem3}
The function $g$ has at least one pole in $\Delta_x \cup \Delta_y$, that is, one pole at a point on the elliptic cuve in  $\{|x|<1\}\cup \{|y|<1\}$.
\end{lem}

\begin{proof}[Proof of Lemma \ref{lem3}]
According to Figures \ref{fig:listwalksdivorb} and \ref{fig2}, there exists a pair $(Q,n)$ with $n>0$ such that $Q$ and $\tau^{n}(Q)$ are poles of $b_{2}$, and for all $\ell\in ]-\infty, -1]\cup [n+1,\infty[$, $\tau^{\ell}(Q)$ is not a pole of $b_{2}$. Since Condition~\ref{condition:linde} is satisfied, one sees that $\tau^{n}(Q)$ and $\tau (Q)$ should be \ju{poles} of $g$. Using Figures \ref{fig:listwalksdivorb} and \ref{fig2}, we see that \ju{for all walks in $\walkse$ except $\walk{IIB.7}$},  this implies that $g$ should have a pole for an $\omega\in \C$ that corresponds to $\{|x|<1\}\cup \{|y|<1\}$ in the elliptic curve, proving the lemma for those {\mfs cases}.\par 
It remains to treat \ju{$\walk{IIB.7}$}. In this case, $b_{2}$ has two double poles $Q,\tau^{4}(Q)$  and two simples poles $\tau(Q),\tau^{3}(Q)$  with $Q=([1:0],[-1:1])$. The operator $L$ in  Condition~\ref{condition:linde} have coefficients in $\C$. Let $n$ be its order. With Lemma \ref{lemma:derivationvalutaion}, we find that $L(b_{2})$ have only two poles of order $n+2$, that are $Q,\tau^{4}(Q)$, and no poles of higher order. With $L(b_{2})=\tau(g)-g$, we find that $g$ should have a pole of order $n+2$ at $\tau^{4}(Q)$. So $\tau(g)$ have a pole of order $n+2$ at $\tau^{3}(Q)$. Since $L(b_{2})=\tau(g)-g$ and $L(b_{2})$ have no poles of order $n+2$ or higher at $\tau^{3}(Q)$, $g$ should have a pole of order $n+2$ at $\tau^{3}(Q)$. But the $y$-coordinates of $\tau^{3}(Q)$ is $[0:1]$, proving the lemma in this case.
\end{proof}

Let us complete the proof of the theorem. From Lemma \ref{lem2} and Lemma \ref{lem3},  we see that there exists $\omega_{0}\in \C$, such that $g$ has a pole at $\omega_{0}$ and such that $r_{y}$ is analytic at $\omega_{0}$. Since $L$ has coefficients in $\C$, $L(r_{y})$ is analytic at $\omega_{0}$.
This contradicts $L(r_y) -g  \in \C$.
\end{proof}

 \section{Some comments concerning singular and weighted walks}\label{sec:weighted}
\subsection{Singular walks} If the walk is singular, then $\Etproj$ is a rational curve, {\it i.e.}, is birational to $\mathbb{P}^{1}(\C)$. 
The difference equations on elliptic curves involved in the nonsingular case should be replaced by finite difference or $q$-difference equations on a rational curve. It seems plausible that our Galoisian methods can be used in order to study the generating series of the singular walks as well.  

\subsection{Weighted walks}
We consider a walk with small steps in the quarter plane $\Z_{\geq 0}^{2}$.
{Following \cite{FIM,KauersYatchak}, the step $(i,j)\in  \{0,\pm 1\}^{2} $ is weighted by some $d_{i,j}\in \Q_{\geq 0}$. Let $\mathcal{D}$ be the set $\{d_{i,j} \ \vert \ (i,j)\in  \{0,\pm 1\}^{2}\} $.}
For $i,j,k\in \Z_{\geq 0}$, we let $q_{\mathcal{D},i,j,k}$ be the number of walks in $\Z_{\geq 0}^{2}$ with {weighted steps in $\mathcal{D}$}  starting at $(0,0)$ and ending at $(i,j)$ in $k$ steps and we consider the corresponding trivariate generating series 
$$
Q_{\mathcal{D}}(x,y,t):=\displaystyle \sum_{i,j,k\geq 0}q_{\mathcal{D},i,j,k}x^{i}y^{j}t^{k}.
$$

We can then ask for these weighted walks the same questions as for the unweighted walks considered in the \ju{present} paper. It turns out that in the weighted context, we have generalizations of the basic tools used in the unweighted case (generalizations of the kernel, of the functional equation \eqref{eq:funcequ}, {\it etc}).  
Moreover, explicit conditions can be deduced from \cite{FIM}, Section 2.3.2, Corollary 4.2.11 and Theorem~3.2.1 in order to have \ju{the properties listed in Remark \ref{rem:remfundpropused} satisfied}. This should be the starting point to apply our technics in this context.

\appendix
\section{Galois Theory of Difference Equations}\label{diffgalois} In this section we describe some basic facts  concerning the Galois theory of linear difference equations and indicate  how these lead to a proof of Proposition~\ref{prop:caract hyperalg}. Many of these facts we state without proof but proofs can be found in \cite{VdPS97}.

The appropriate setting for this  Galois is {\it difference algebra}, that is, the study of  algebraic objects endowed with an automorphism, so we begin with

\begin{defn} A {\em difference ring} is a pair $(R,\tau)$ where $R$ is a ring and $\tau$ is an automorphism of $R$.  A {\em difference ideal} $I\subset R$ is an ideal such that $\tau(I) \subset I$.
\end{defn}

One can define difference subring, difference homomorphism, difference field, etc.~in a similar way. A difference subring of particular importance in  any difference ring is given in the following definition.

\begin{defn} The {\em constants} $R^\tau$ of a difference ring $R$ are
\[R^\tau = \{ c \in R \ | \ \tau(c) = c\}.\]
\end{defn} 
One can show that $R^\tau$ forms a ring and, if $R$ is a field, then $R^\tau$ is also a field.

\begin{ex} 1.~$(\C[x], \tau)$ where $\tau(x) = x+1$. The only difference ideals are $\C[x]$ and $\{0\}$. The constants of this ring are $\C$.\\
2.~$(\C[x], \tau)$ where $\tau(x) = qx$, $q$ not a root of unity.  The difference ideals are $\C[x], \{0\}$ and the $(x^{k})$ for $k\in \Z_{\geq 1}$. The constants of this ring are $\C$.\\
3.~$(\calM(\C), \tau)$ where $\calM(\C)$ is the field of meromorphic functions on $\C$ and $\tau(\omega) = \omega+\omega_3$ for some $\omega_3 \in \C$.  This is and the next example are  difference fields. The constants of this field form the field of $\omega_3$-periodic meromorphic functions. \\
4.~$(\C(\overline{E_t}), \tau)$ where $\C(\overline{E_t})$ is the field of meromorphic functions on $\overline{E_t})$ and for some fixed $P \in \overline{E_t} \ , \tau(f(X)) = f(X\oplus P)$ for all $f \in \C(\overline{E_t})$. If $P$ is of infinite order, then the constants are $\C$ (see the argument following Definition~\ref{def:tauconstselliptic}). \end{ex}

When considering linear difference equations, it is most convenient to consider first order matrix equations, that is, equations of the form $\tau(Y) = AY$ where $A\in \GL_n(K)$ where $K$ is a difference field.  Often one wants to deal with equations of the form  form $L(y) = \tau^n(y)+ a_{n-1}\tau^{n-1}(y) + \ldots + a_0y = 0, \ a_i \in K$.  If $a_0 = \ldots = a_{j-1} = 0, a_j \neq 0$, we can make a change of variables $z=\tau^j(y)$ and assume $a_0 \neq 0$.  One then sees that questions concerning solutions  of $L(y) = 0$ can be reduced to questions concerning the system $\tau(Y) = A_LY$ where 
  \begin{equation*}
 A_L=\footnotesize{\begin{pmatrix}
0&1&0&\cdots&0\\
0&0&1&\ddots&\vdots\\
\vdots&\vdots&\ddots&\ddots&0\\
0&0&\cdots&0&1\\
-a_0& -a_{1}&\cdots & \cdots & -a_{n-1}
\end{pmatrix}} \in \GL_{n}(K).
\end{equation*}
If $z$ is a solution of $L(y) = 0$ in some difference ring containing $K$, then $(z, \tau(z), \ldots, \tau^{n-1}(z))^T$ is a solution of $\tau(Y) = A_LY$. 

In addition to considering individual solutions of $\tau(Y) = AY$, it is useful to consider matrix solutions and, in particular

\begin{defn} Let $R$ be a difference ring  and $A \in \GL_n(R)$.  A {\em fundamental solution matrix} of $\tau(Y) = AY$ is a matrix $U \in \GL_n(R)$ such that $\tau(U) = AU$.
\end{defn}

Note that if $U_1$ and $U_2$ are fundamental solution matrices of $\tau(Y) = AY$, then $\tau(U_1^{-1}U_2) = U_1^{-1}U_2$ so $U_c = U_1 D$ where $D \in \GL_n(R^\tau)$.\\

The usual Galois theory of polynomial equations is cast in terms of a splitting field of the polynomial and a group of automorphisms of this field.  For linear difference equations, the following takes the place of the splitting field.

\begin{defn} Let $K$ be a difference field and $A \in \GL_n(K)$.  We say that a $k$-algebra $R$ is a {\em Picard-Vessiot ring} for $\tau(Y) = AY$ if 
\begin{enumerate}
\item $R$ is a simple difference ring extension of $K$ (i.e., the only difference ideals of $R$ are $R$ and $\{0\}$).
\item $R = K[U, 1/\det(U)]$ for some fundamental solution matrix $U\in \GL_n(R)$ of $\tau(Y) = AY$.
\end{enumerate}
\end{defn} 

It can be shown (cf. \cite[Chapter 1.1]{VdPS97}) that Picard-Vessiot rings always exist and  if $K^\tau$ is algebraically closed they are unique up to $k$-difference isomorphisms.  Furthermore, when $K^\tau$ is algebraically closed , we have $R^\tau = K^\tau$. Although some of the following results hold in more general situations, we will for simplicity assume from now on that 
\begin{quotation}
$K^\tau$ is algebraically closed and of characteristic zero. 
\end{quotation}
Even under this assumption, the Picard-Vessiot ring need not be an integral domain. In  \cite[Corollary 1.16]{VdPS97}  a precise description of its structure is given but we will only use some basic facts listed below and not  delve further.

We can now define the Galois group. 

\begin{defn} Let $R$ be the Picard-Vessiot ring of $\tau(Y) = AY, \ A \in \GL_n(K)$. The {\em Galois group G} of $R$ (or of $\tau(Y) = AY$) is 
\[ G = \{ \sigma:R\rightarrow R \ | \ \sigma \text{ is a $K$-algebra automorphism of $R$ and } \sigma\tau = \tau \sigma\}\]
\end{defn}

Using the notation of the definition, fix a fundamental solution matrix in $\GL_n(R)$ of the equation $\tau(Y) = AY$.  If $\sigma \in G$, then 
\[ \tau(\sigma(U)) = \sigma(\tau(U)) = \sigma(AU) = A\sigma(U). \]
Therefore, $\sigma(U)$ is again a fundamental solution matrix and so $\sigma(U) = U[\sigma]_U$ where $[\sigma]_U \in \GL_n(K^\tau)$.  A key fact (cf. \cite[Chapter 1.2]{VdPS97}) forming the basis of the Galois theory of linear difference equations is
\begin{quotation} \noindent The map $\rho:G \rightarrow \GL_n(K^\tau)$ given by $\rho(\tau) = [\tau]_U$ is a group homomorphism whose image is a {\em linear algebraic group}.
\end{quotation}
A subgroup $G \subset \GL_n(K^\tau)$ is a linear algebraic group if it is a closed in the Zariski topology on $\GL_n(K^\tau)$, the topology whose closed sets are common solutions of systems of polynomial equations in $n^2$ variables. 
\begin{ex}\label{ex:add} 1. \toto{Consider } the equation 
\begin{eqnarray} \label{eq:add}
\tau(y) - y = b, \ b \in K.
\end{eqnarray} This equation is not a homogeneous linear difference equation but it is equivalent to the matrix equation
\begin{eqnarray*} \tau(Y) = \begin{pmatrix} 1&b\\0&1\end{pmatrix}Y.\end{eqnarray*}
If $z$ satisfies \eqref{eq:add}, then 
$U = \begin{pmatrix} 1 &z \\0&1\end{pmatrix}$
is a fundamental solution of the matrix equation.  The Picard-Vessiot extension of $k$ is then given by $R = K[z]$. If $\sigma \in G$, then $y=\sigma(z)$ also satisfies \eqref{eq:add}.  Therefore $\sigma(z) - z = d_\sigma \in \K^\tau$.  This implies that the Galois group of the matrix equation may be identified with a Zariski closed subgroup of 
\[\toto{\left. \left\{\begin{pmatrix} 1& d\\0&1 \end{pmatrix} \ \right| \ d \in K^\tau\right\}.}\]
Note that this latter group is just the additive group $(K^\tau, +)$. The Zariski closed subgroups of this group are identified with $K^\tau$ and $\{0\}$.\\

2. Consider the system of equations 
\begin{eqnarray} \label{eq:addsys}
\tau(y_0) - y_0 = b_0,  \ldots , \tau(y_n) - y_n = b_n\ b_0, \ldots , b_n  \in K.
\end{eqnarray}
As above, this system is equivalent to the matrix equation
\begin{eqnarray*} \tau(Y) = \begin{pmatrix} B_0 & 0 & \ldots & 0\\
0& B_1 &\ldots & 0\\
\vdots &\vdots & \vdots &\vdots \\
0 & 0 & \vdots & B_n  \end{pmatrix}Y, \ \text{ where } B_i = \begin{pmatrix} 1 & b_i \\0 & 1 \end{pmatrix}\end{eqnarray*}

The Picard-Vessiot extension  of this equation is $R=K(z_0, \ldots , z_n)$ where $\tau(z_i) - z_i = b_i$ and the Galois group is a subgroup of 

\toto{\begin{eqnarray*} \left.\left\{\begin{pmatrix} C_0 & 0 & \ldots & 0\\
0& C_1 &\ldots & 0\\
\vdots &\vdots & \vdots &\vdots \\
0 & 0 & \vdots & C_n  \end{pmatrix} \ \right| \ \ C_i = \begin{pmatrix} 1 & d_i \\0 & 1 \end{pmatrix} d_i \in K^\tau, i = 0, \ldots, n\right\}.\end{eqnarray*}}

This latter group is just the direct sum of $n+1$ copies of the  additive group $(K^\tau, +)$, that is $G \subset ((K^\tau)^{n+1}, +)$. The Zariski closed subgroups of $(K^\tau)^{n+1}$ are the vector subspaces and are all connected in the Zariski topology.  From this we can deduce

\begin{lem}\label{lem:subgp} Using the notation of Example~\ref{ex:add}.2, if $G$ is a proper linear algebraic subgroup of   
$(K^\tau)^{n+1}$ then there exist $c_0, \ldots , c_n \in K^\tau$, not all zero, such that 
$$\toto{G \subset \left\{(d_0, \ldots , d_n) \in (K^\tau)^{n+1} \ \left| \sum_{i=0}^n c_id_i = 0\right\}\right. .}$$ \end{lem}
\end{ex}
In \cite[Chapter 1.3]{VdPS97} a Galois correspondence and other basic facts are described. For our purposes, we only need
\begin{enumerate}
\item[(Gal 1)] An element $z\in R$ is in $K$ if and only if $z$ is left fixed by all elements of $G$.
\item[(Gal 2)] The ring $R$ is an integral domain if and only if $G$ is connected in the Zariski topology.
\item[(Gal 3)] When $G$ is connected, the dimension of $G$ as an algebraic variety over $K^\tau$ is equal to the transcendence degree of the quotient field of $R$ over $K$.
\end{enumerate}
Concerning (Gal 3), when $G$ is not connected then the dimension equals the Krull dimension of $R$.

We now present the main tool used in proving Proposition~\ref{prop:caract hyperalg}.

\begin{prop}\label{prop:tool} Let $R$ be the Picard-Vessiot extension  for the system \eqref{eq:addsys} and $z_0, \ldots , z_n \in R$ be solutions of this system.  If $z_0, \ldots , z_n $ are algebraically dependent over $K$, then there exist $c_i \in K^\tau$, not all zero,  and $g \in K$ such that 
\[c_0 b_0 + \ldots + c_n b_n = \tau(g) - g.\]\end{prop}  
\begin{proof} We follow ideas due to M.~van der Put appearing in the appendix of \cite{Hard08}. As in Example~\ref{ex:add}.2, the Galois  group $G$ is a subgroup of $(K^\tau)^{n+1}$ and so is connected.  (Gal 2) implies that $R = K[z_0, \ldots , z_n]$ is a domain. Since $z_0, \ldots , z_n$ are algebraically dependent, the transcendence degree of      the quotient field of $R$ is less than $n+1$. Therefore (Gal 3) implies that  $G$ is a proper subgroup of $(K^\tau)^{n+1}$. Lemma~\ref{lem:subgp} implies that $G \subset \{(d_0, \ldots , d_n) \ | \sum_{i=0}^n c_id_i = 0\}$ for some $c_i \in K^\tau$.  For  any $\sigma \in G$, we have
\[\sigma(\sum_{i=0}^n c_iz_i) = \sum_{i=0}^n c_i(z_i + d_i) = \sum_{i=0}^n c_iz_i+ \sum_{i=0}^n c_id_i = \sum_{i=0}^n c_iz_i.\]
From (Gal 1) we conclude that $\sum_{i=0}^n c_iz_i = g \in K$. Applying $\tau$ to this last equation and subtracting yields
\[\tau(g) - g = \tau(\sum_{i=0}^n c_iz_i)-\sum_{i=0}^n c_iz_i = \sum_{i=0}^n c_i(z_i+b_i) - \sum_{i=0}^n c_iz_i = \sum_{i=0}^n c_ib_i\]
\end{proof}

We now turn to \\

\noindent {\it Proof of Proposition~\ref{prop:caract hyperalg}.} Let $f$ satisfy $P(f, \delta(f),\ldots ,\delta^n(f)) = 0$ for some polynomial $P$ with coefficients in $K = \C(\overline{E_t})$.  Since $\delta$ and $\tau$ commute, we have 
\[\tau(f) - f = b, \tau(\delta(f)) - \delta(f) = \delta(b), \ldots , \tau(\delta^n(f)) - \delta^n(f) = \delta^n(b).\]
Let $I$ be a maximal difference ideal in the difference ring  $K[f, \delta(f), \ldots , \delta^n(f)]$ and let $R =  K[f, \delta(f), \ldots , \delta^n(f)]/I$. The difference ring $R$ is a simple difference ring of the form $K[z_0, \ldots, z_n]$ where  $z_i$ is the image of $\delta^i(f)$ and $\tau(z_i) - z_i = b_i, \ b_i = \delta(b)$.  Therefore, $R$ is a Picard-Vessiot ring for a system of the form \eqref{eq:addsys}. Applying Proposition~\ref{prop:tool}, yields the first conclusion of Proposition~\ref{prop:caract hyperalg}.

Now assume that $K[f, \delta(f), \ldots , \delta^n(f)]^\tau = \C$.  A computation shows that $\tau(L(f) -g) -(L(f) - g) =0$, where $L = a_n\delta^n + a_{n-1}\delta^{n-1} + \ldots + a_0$.  Therefore $L(f) = g+ c, \text{ for some } c \in \C$. If $g+c = 0$, this equation shows that $f$ is holonomic over $K$.  If $g+c \neq 0$, we can derive the same conclusion by considering $\delta(L(f)) - (\delta(g)/(g+c))L(f) = 0$. \hfill $\square$\\

The above proof very much depends on the fact that we are considering systems of first order scalar difference equations of the form \eqref{eq:addsys}. In \cite{DHR}, a proof of Proposition~\ref{prop:caract hyperalg} was given based on the {\it \toto{differential Galois theory of linear difference equations}}.  This is a theory, presented in \cite{HS},  that allows one to describe differential properties of general linear difference equations. A general introduction to this theory as well as an elementary introduction to the Galois theory of linear differential equations and the analytic theory of $q$-difference equations can be found in the articles in \cite{HSS16}.

\section{Telescopers and orbit residues} \label{introsec}

We have seen in Section \ref{sec:applichyptrF1F2} that the study of the hypertranscendance of $F^{1}(x,t)$ and $F^{2}(y,t)$ is intimately related to the study of equations of the form ${L(b)=\tau(g)-g}$ for some nonzero linear differential operator $L$ with coefficients in $\C$ and some $b,g \in \C(\Etproj)$. In other contexts (cf.~\cite{Chen_Singer12}), $L$ is referred to as a {\it telescoper for $b$} and $g$ as a {\it witness}.  The aim of this appendix is to study in more details these equations. \par 

Let $E$ be an elliptic curve defined over an algebraically closed field $k$ of characteristic zero and  $k(E)$ be its function field.  Let $P$ be a non-torsion point on $E$ and let $\tau:k(E)\rightarrow k(E)$ denote map corresponding to $Q\mapsto   Q\oplus P$ on $E$, where $\oplus$ denotes the group law on $E$.  We let $\Omega$ be a non zero regular differential form on $E$. 
 A straightforward generalization of Lemma \ref{lem:derivellcurvecommtau} shows the following result~: 
\ju{ 
\begin{lemma} \label{lem:derivellcurvecommtaugen}
The derivation $\delta$ of $k(E)$ such that $d(f)=\delta(f) \Omega$ commutes with $\tau$. 
\end{lemma}
}
We will prove the following

\begin{prop}\label{Prop0} Let $b \in k(E)$.  The following are equivalent. 
\begin{enumerate}
\item\label{Prop0.1}There exist $g \in k(E)$ and a nonzero operator $L \in k[\delta]$ such that ${L(b) = \tau(g) - g}$.
\item\label{Prop0.2} For all poles $Q_0$ of $b$, we have that 
\[h(X) =\sum_{i=1}^t b(X\oplus n_iP)\]
is regular at $X=Q_0$ where $Q_0\oplus n_1P, \ldots , Q_0\oplus n_tP$ are the poles of $b$ that belong to $Q_0\oplus\ZX P$.
\end{enumerate}
\end{prop}

This proposition allows one to give the following useful criteria guaranteeing when Condition \ref{eq:lindiffc} does or does not hold.

\begin{cor}\label{cor1} Let $b\in k(E)$ and assume that there exists $Q_0 \in E$ such that
\begin{enumerate}
\item $b$ has a pole of order $m > 0$ at $Q_0$, and
\item $b$ has no other pole of order $\geq m$ in $Q_0 \oplus \ZX P$.
\end{enumerate}
Then there is no nonzero $L \in k[\delta]$ and $g\in k(E)$ such that $L(b) = \tau(g) - g$.
\end{cor}

\begin{proof}
This follows easily from Proposition~\ref{Prop0} since the pole of $b(X)$ at $Q_0$ cannot be cancelled by any pole of any $b(X\oplus nP)$ and so $h(X)$ is not regular at $X=Q_0$.
\end{proof}

\begin{cor}\label{cor2} Let $b \in k(E)$ and assume that there exists $Q_0 \in E$ such that
\begin{enumerate}
\item all poles of $b$ occur in $Q_0 \oplus \ZX P$, and
\item all poles of $b$ are simple.
\end{enumerate}
Then there exist $g \in k(E)$ and a nonzero operator $L \in k[\delta]$ such that $L(b) = \tau(g) - g$.
\end{cor}

\begin{proof} {\mfs Basically, this is true because the sum of the residues of a differential form on a compact Riemann surface is zero. More precisely,  
using}  Lemma~\ref{lem5} below, one can show that the hypotheses of Corollary~\ref{cor2} imply condition $(2)$ of Proposition~\ref{Prop1} and therefore that the conclusion holds (see the remark following Lemma~\ref{lem5}).
\end{proof}

To prove Proposition~\ref{Prop0} we shall prove two ancillary results, Propositions~\ref{Prop1} and~\ref{Prop2}.  These results give conditions equivalent to the conditions in Proposition~\ref{Prop0}. 
 
Before proceeding, we recall the following standard notation.
If $D$ is a divisor of $E$, we will denote by $\calL(D)$ the finite dimensional $k$-space $\{ f \in k(E) \ | \ (f) + D \geq 0\}$, where $(f)$ is the divisor of $f$.  In \ju{Subsection}~\ref{prop1sec} we will prove

\begin{prop}\label{Prop1}Let $b \in k(E)$.  The following are equivalent.
\begin{enumerate}
\item There exist $g \in k(E)$ and a nonzero operator $L \in k[\delta]$ such that ${L(b) = \tau(g) - g}$.
\item There exists $Q\in E$, $e\in k(E)$ and $h \in \calL(Q+(Q\ominus P))$\footnote{The symbol ``+'' represents the  formal sum of divisors.  We will use $\oplus$ and $\ominus$ for addition and subtraction of points on the curve.} such that
\[b = \tau(e) - e+ h.\]
\end{enumerate}
\end{prop}

To state the next equivalence, we need two definitions. Corresponding to each point $Q \in E$ there exists a valuation ring $\frakO_Q \subset k(E)$. A generator $u_Q$ of the maximal ideal of $\frakO_Q$ is called a local parameter at $Q$. Local parameters are unique up to multiplication by a unit of $\frakO_Q$.

\begin{define}\label{def:coherentlocparam} Let $\mathcal{S} = \{ u_Q \ | \ Q\in E\}$ be a set of local parameters at the points of $E$.  We say $S$  is a {\em coherent set of local parameters} if for any $Q \in E$, 
\[ u_{Q\ominus P} = \tau(u_Q).\]
\end{define}
We fix, once and for all, a coherent set of local parameters. All local  parameters mentioned henceforth will be from this set.

Let $u_Q$ be a local parameter at a point $Q \in E$ and let $v_Q$ be the valuation corresponding to the valuation ring at $Q$.  If $f \in k(E)$ has a pole at $Q$ or order $n$, we may write

\[f = \frac{c_{Q,n}}{u_Q^n} + \ldots + \frac{c_{Q,2}}{u_Q^2} + \frac{c_{Q,1}}{u_Q} + \tilde{f}\]
where $v_Q(\tilde{f}) \geq 0$. The following definition is similar to Definition 2.3 of \cite{Chen_Singer12}.

\begin{define} Let $f \in k(E)$ and $S =\{ u_Q \ | \ Q\in E\}$ be a coherent set of local parameters and $Q\in E$. For each $j \in \NX_{>0}$ we define the {\rm orbit residue of order $j$ at $Q$} to be
\[ \ores_{Q,j}(f) = \sum_{i \in \ZX} c_{Q\oplus iP, j.}\]
\end{define}

Note that if  $Q' = Q \oplus tP$ for some $t \in \ZX$, then $ \ores_{Q',j}(f) =  \ores_{Q,j}(f)$ for any $j \in \NX_{>0}$. Furthermore $\ores_{Q,j}(f) = \ores_{Q,j}(\tau(f))$. We shall prove the next result in \ju{Subsection}~\ref{prop2sec}.  

\begin{prop}\label{Prop2} Let $b \in k(E)$ and $S =\{ u_Q \ | \ Q\in E\}$ be a coherent set of local parameters.  The following are equivalent.
\begin{enumerate}
\item There exists $Q\in E$, $e\in k(E)$ and $g \in \calL(Q+(Q\ominus P))$ such that
\[b = \tau(e) - e+ g.\]
\item For any $Q \in E$ and $j \in \NX_{>0}$
\[ \ores_{Q,j}(b) = 0.\]
\end{enumerate}
\end{prop}

 Proposition~\ref{Prop1}, Proposition~\ref{Prop2} and the following lemma immediately imply Proposition~\ref{Prop0}.

\begin{lem}\label{lem:lemhregandresj}Let $b \in k(E)$. The following are equivalent
\begin{enumerate}
\item  For all poles $Q_0$ of $f$, we have that 
\[h(X) =\sum_{i=1}^t b(X\oplus n_iP)\]
is regular at $X=Q_0$ where $Q_0\oplus n_1P, \ldots , Q_0\oplus n_tP$ are the poles of $b$ that belong to $Q_0\oplus\ZX P$.
\item For any $Q \in E$ and $j \in \NX_{>0}$
\[ \ores_{Q,j}(b) = 0.\]
\end{enumerate}
\end{lem}
\begin{proof}If $u$ is the local parameter at $Q_0$, we may write 
\[h = \frac{c_n}{u^n} + \ldots + \frac{c_1}{u} + h'\]
where $v_{Q_0}(h') \geq 0$. One easily sees that 
\[c_j = \ores_{Q_0,j}(b).\]
The conclusion now follows.\end{proof}

\begin{rmk}\label{rmk:analylocparam}
Assume that $k=\C$. Then, one can consider the analytification $E^{an}$ of $E$.  
Instead of considering algebraic local parameters on $E$, one can consider analytic local parameters $\{u_{Q} \ \vert \ Q \in E\}$, {\it i.e.}, for any $Q \in E$, $u_{Q}$ is a biholomorphism between a neighborhood of $Q$ in $E^{an}$ and a neighborhood of $0$ in $\C$. There is an obvious notion of coherent analytic local parameters, extending the notion introduced in Definition \ref{def:coherentlocparam}, and a corresponding notion of $\ores_{Q,j}$. Lemma \ref{lem:lemhregandresj} remains true in this context, with the same proof. 
\end{rmk}

\begin{rmk} The proof that (2) implies (1) in Proposition~\ref{Prop2} is constructive.  One only needs a constructive method for finding the bases of certain $\calL$ spaces (e.g. \cite{He02}). The proof that (2) implies (1) in Proposition~\ref{Prop1} is also constructive.  Therefore given $b \in k(E) $ one can decide if there exist $g \in k(E)$ and a nonzero operator $L \in k[\delta]$ such that $L(b) = \tau(g) - g$.
\end{rmk} 

\subsection{Proof of Proposition~\ref{Prop1}} \label{prop1sec}

In the following lemma, we will collect some facts concerning the local behavior of functions under the actions of $\tau$.  Its proof is a straightforward generalization of the proof of Lemma \ref{lemma:derivationvalutaion}.

 \begin{lem} 
 Let $u$ be a local parameter of  $k(E)$ and let $v_u$ be the associated valuation.   Then $v_u(\delta(u)) = 0$ and,  for any $f \in k(E)$ such that $v_u(f)\neq 0$, we have \begin{enumerate}
\item if $v_u(f) \geq 0$ then $v_u( \delta(f)) \geq  0$;
\item if $v_u(f)<0$ then $v_u( \delta(u))=  v_u(f)-1$.
\end{enumerate}
\end{lem}

 We will also need a consequence of the Riemann-Roch Theorem for elliptic curves:  If $D$ is a positive divisor on $E$   and $l(D)$ is the dimension of the space $\calL(D)$ then 
\[l(D) = \mbox{degree of }D.\]
This implies that if $Q$ is a point on $E$,  $u$ is a local parameter at $Q$, $n\geq 2$,  and $c_2, \ldots , c_n \in k$, then  there exists an $f \in \calL(nQ)$  and $c_1 \in k$ such that 
\[f = \frac{c_n}{u^n} + \ldots + \frac{c_2}{u^2} + \frac{c_1}{u} + \tilde{f}\]
where $v_u(\tilde{f}) \geq 0$. {\it A priori}, we have no control of the element $c_1$.\\[0.1in]

Finally we need some definitions:

\begin{define} Let $f \in k(E)$ and $Q \in E$. 
\begin{enumerate}
\item If $Q$ is a pole of $f$, the \emph{polar dispersion of $f$ at $Q$, $\pdisp(f,Q)$} is the largest nonnegative integer $\ell$ such that $Q\oplus\ell P$ is also a pole of $f$.
\item The \emph{polar dispersion of $f$, pdisp(f),} is $\max\{\pdisp($f$,Q) \  | \ Q \mbox{ a pole of } f\}$.
\item The \emph{weak polar dispersion of $f$, wpdisp(f),} is \\$\max\{\ell \ | \ \exists Q \in E  \mbox{ s.t. $f$ has a pole of order at least $2$ at $Q$ and $Q\oplus \ell P$} \}$.
\end{enumerate} \end{define}

 The following is an analogue of \cite[Lemma 6.2]{HS}.

\begin{lem}\label{normform} Let $f \in k(E)$.  There exist $f^*, g \in k(E)$ such that  ${\pdisp(f^*) \leq 1}$, wpdisp($f^*$) = $0$ and 
$f = f^* + \tau(g) - g.$ \end{lem}
\begin{proof}  We  begin by showing that there exist $f^*, g \in k(E)$ such that wpdisp($f^*$) = $0$ and 
$f = f^* + \tau(g) - g.$  We  will then further refine $f^*$ so that $\pdisp(f^*) \leq 1$ as well.

 Let $N =$ wpdisp($f$)$\geq 1$ and $n_f =$ the number of points $Q \in E$  such that $f$ has poles of order at least two at $Q$ and $Q\ominus NP$.  Fix such a point $Q$ and let $u$ be a local parameter at $Q$.  We  may write
 
\[ f = \sum_{i=1}^m \frac{a_i}{u^i} + h_f\]

where $m \geq 2$ and $v_{Q}(h_f) \geq 0$. The Riemann-Roch Theorem implies that there exists a $\tilde{g} \in \calL(mQ)$ such that%
\[ \tilde{g} = \sum_{i=1}^m \frac{b_i}{u^i} + h_{\tilde{g}}\]

where $b_i = -a_i $ for $i = 2, \ldots , m$ and $v_{Q}(h_{\tilde{g}}) \geq 0$. Note that $\tau(\tilde g)$ has a pole of order $m$ at $Q\ominus P$.
 Letting $\tilde{f} = f - (\tau(\tilde{g}) - \tilde{g})$, one sees that $\tilde{f}$  has a pole of order at most $1$ at  $Q$.  Therefore either the wpdisp($f$) = wpdisp($\tilde{f}$)  and $n_{\tilde{f}} < n_f$ or wpdisp($f$) $>$ wpdisp($\tilde{f}$). An induction allows us to conclude that there exist $f^*, g \in k(E)$ such that wpdisp($f^*$) = $0$ and 
$f = f^* + \tau(g) - g.$

We  may now assume that wpdisp($f$) $= 0$ and let $\pdisp(f) = N \geq 2$.  Let $f$ have poles at both $Q$ and $Q\oplus NP$.  Since wpdisp($f$) $= 0$, $f$ has a pole of order greater than one at no more than  one of these two points. We  deal with the two cases separately.

\underline{$f$ has a pole of order $1$ at $Q\oplus NP$.} The Riemann-Roch Theorem implies that there exists a nonconstant $\tilde{g} \in \calL((Q\oplus (N-1)P) + (Q \oplus NP))$. Note that $\tau(\tilde{g}) \in \calL((Q\oplus (N-2)P) + (Q \oplus (N-1)P))$.  For some $a \in k$, $\tilde{f} - (\tau(a\tilde{g}) - a\tilde{g})$ has no pole at $Q\oplus NP$ and so $\pdisp(f,Q) < N$.  An induction finishes the proof.

\underline{$f$ has a pole of order $1$ at $Q$.} The  Riemann-Roch Theorem implies that there exists a nonconstant $\tilde{g} \in \calL((Q\oplus P) + (Q \oplus 2P))$. Note that $\tau(\tilde{g}) \in \calL((Q) + (Q \oplus P))$. For some $a \in k$, $\tilde{f} - (\tau(a\tilde{g}) - a\tilde{g})$ has no pole at $Q$ and so $\pdisp(f,Q) < N$.  An induction again finishes the proof.\end{proof}

We now turn to the  
\begin{prop2proof} Applying Lemma~\ref{normform}, we may assume that $\pdisp(b) \leq 1$ and wpdisp($b$) = $0$ {\mfs (here one uses the fact that $L\circ \tau = \tau \circ L$, since $L \in \C[\delta]$ and $\tau\circ\delta = \delta\circ\tau$).} We will first show that for any $Q \in E$, if $b$ has a pole at $Q$ then this pole must be simple and it has another pole of the same order at $Q\oplus P$ or at $Q\ominus P$. To see this note that if  $b$ has a pole at $Q$, then either $g$ or $\tau(g)$ has a pole at $Q$. Assume that $g$ has a pole at $Q$ (the argument assuming $\tau(g)$ has a pole at $Q$ is similar).  Let $r$ be the largest integer such that $Q\oplus rP$ is a pole of $g$ and $s$ be the largest integer such that $Q\ominus sP$ is a pole of $g$.  We  then have that $Q\oplus rP$ and $Q\ominus (s +1)P$ are both poles of $\tau(g)-g$ and therefore of $L(b)$. Using Lemma~\ref{lemma:derivationvalutaion} above, one sees that they  must also be poles of $b$. Since $\pdisp(b) \leq 1$, we have $r = s = 0$. In particular, the only pole of $g$ in $Q\oplus \ZX P$ is at $Q$, the only pole of $\tau(g)$  in $Q\oplus \ZX P$ is at $Q\ominus P$ and they must have the same orders. Once again, Lemma~\ref{lemma:derivationvalutaion}  above implies that $b$ has poles at these points of equal orders.  Since  wpdisp($b$) = $0$, the orders of these poles must be $1$. 

 We  can therefore conclude that $b$ has only poles of order $1$ and the poles of $b$ occur in pairs $\{Q_1, Q_1\ominus P\}, \ldots, \{Q_r, Q_r\ominus P\}$   where \ju{$(Q_i \oplus \ZX P) \cap (Q_j \oplus \ZX P) = \emptyset$} for $i \neq j$.  
 
We  will now show how one can construct an element $e$ such that $b -(\tau(e) - e)$ has at most one pair of poles $\{Q, Q\ominus P\}$. This will yield \ju{(2) and our contention}.  Assume $r > 1$ and that $b$ has simple poles at the pairs $\{Q_1, Q_1\ominus P\}$ and $\{Q_2, Q_2\ominus P\}$.  Let $h\in k(E)$ be a nonconstant element of $\calL(Q_1+Q_2)$.  There exists an $a \in k$ such that { $\tilde{b} = b - (\tau(ah)-ah)$} has no pole at $Q_1$. The element $\tilde{b}$ has only simple poles and $\mbox{\pdisp}(b) \leq 1$.  Therefore its poles occur at possibly $Q_1\ominus P, \{Q_2, Q_2\ominus P\}, \ldots , \{Q_r, Q_r\ominus P\}$.  Since $\tilde{b}$ satisfies an equation of the form {$L(\tilde{b}) = \tau(\tilde{g}) - \tilde{g}$ for some $\tilde{g} \in k(E)$} the poles of such an $\tilde{b}$ must occur in pairs. Therefore we have that $\tilde{b}$ has no pole at $Q_1\ominus P$.  Continuing in this way we find an $e \in k(E)$ such that $b -(\tau(e) - e)$ has at most one pair of poles $\{Q, Q\ominus P\}$. \end{prop2proof}

In the proof that (2)~implies (1)~in \ju{Proposition~\ref{Prop1}} we will need the following technical lemma. Let $u$ be a local parameter at $Q$. Note that  $\tau(u)$ is a local parameter at $Q\ominus P$.

\begin{lem}\label{lem5} If $g \in \calL(Q + (Q\ominus m_1 P) + \ldots + (Q\ominus m_t P) )$ where $m_1, \ldots, m_t \in \ZX\backslash \{0\}$ then 
\[\ores_{Q,1}(g) = 0.\]
\end{lem}
\begin{proof} This result will follow from the fact that the sum of the residues of a differential form \ju{on a compact Riemann surface} must be zero.    We  start by noting that Lemma~\ref{lemma:derivationvalutaion} states that $v_{u_Q}(\delta(u_Q)) = 0$ so we may write $\delta(u_Q)^{-1} = \alpha + \bar u$ where $0 \neq \alpha \in k$ and $\bar{u}$ is regular and zero at $Q$.  For each $i \in \ZX$ we write
\[g = \frac{c_{Q\ominus iP, -1}}{u_{Q\ominus iP}} + g_{Q\ominus iP}\]
where $g_{Q\ominus iP}$ is regular and zero at $Q\ominus iP$. Now consider the differential $g\Omega$.  Since for any $i \in \ZX$,  $\Omega =\delta(u_{Q\ominus iP})^{-1}du_{Q\ominus iP}$, we have
\begin{eqnarray*}
g\Omega & = & (\frac{c_{Q\ominus iP, -1}}{u_{Q\ominus iP}} + g_{Q\ominus iP})(\delta(u_{Q\ominus iP})^{-1}du_{Q\ominus iP})\\
 & = &  (\frac{c_{Q\ominus iP, -1}}{u_{Q\ominus iP}} + g_{Q\ominus iP})(\tau^i(\delta(u_Q)^{-1})du_{Q\ominus iP})\\
 & = &  (\frac{c_{Q\ominus iP, -1}}{u_{Q\ominus iP}} + g_{Q\ominus iP})((\alpha + \tau^i(\bar{u}))du_{Q\ominus iP})
\end{eqnarray*}
where the second equality follows from the fact that $u_{Q\ominus iP} = \tau^i(u_Q)$ and $\tau\delta =\delta \tau$. Therefore the residue of $g\Omega$ at $Q\ominus iP$ is $\alpha c_{Q\ominus iP, -1}$. Since $\alpha \neq 0$ and the sum of the residues of a differential form is $0$ we have $\ores_{Q,1}(g) = 0$.\end{proof}

\begin{rmk}\label{rmk:ores1vanishiffResvanish}
The proof of Lemma \ref{lem5} shows that if the poles of $g \in k(E)$ are simple and belong to $Q \oplus \Z P$, then there exists $0 \neq \alpha \in \C$ such that $\ores_{Q,1} (g)=\alpha \sum_{i \in \Z}\Res_{Q \oplus iP} (g \Omega)$. Therefore, $\ores_{Q,1} (g)=0$ if and only if $\sum_{i \in \Z}\Res_{Q \oplus iP} (g \Omega)=0$. 
\end{rmk}

\begin{rmk} 
Lemma~\ref{lem5} and Proposition~\ref{Prop0} imply Corollary~\ref{cor2}.  To see this note that for $f$ as in  Corollary~\ref{cor2} we have that $f \in \calL(Q + (Q\ominus m_1 P) + \ldots + (Q\ominus m_t P) )$  where $m_i = -n_i$.  The residue of $h(X) = \sum_{i=1}^tf(X\oplus n_iP)$ at $X=P$ is $\ores_{Q,1}(f)$, so $h(X)$ is regular at $Q_0$. Applying Proposition~\ref{Prop0} yields the conclusion of Corollary~\ref{cor2}.
\end{rmk}

\begin{prop2aproof} Let us assume that condition (2)~holds.  We  claim that it is enough to find an element $\tilde{g}$ and a nonzero operator $L$ such that $L(h) = \tau(\tilde{g}) -\tilde{g}$. Assume that we have done this. Then
\[L(b) = L(\tau(e) - e +h)= \tau(L(e)) - L(e) +\tau(\tilde{g}) - \tilde{g} = \tau(g) - g\]
where $g = L(e) + \tilde{g}$. 

If $h$ is constant, then the result is obvious (take $L=\delta$ and $\tilde{g} =0$). We shall now assume that $h$ is not constant.

To simplify notation, we write $u$ for $u_Q$ and let  $\delta(u) = u_0 + \bar u$, where $0 \neq u_{0} \in k$ and $\bar{u}$ is regular and zero at $Q$, and so $\delta(\tau(u)) = u_0 + \tau(\bar u)$. Using Lemma~\ref{lem5}, one sees that   
\begin{eqnarray*}
\delta(h) &=& \frac{-u_0a}{u^2}(1+ h_u)\\
&=& \frac{u_0a}{\tau(u)^2}(1+ h_{\tau(u)})
\end{eqnarray*}
where $v_u(h_u) > 0$ and $v_{\tau(u)}(h_{\tau(u)}) > 0$.  {\mfs Selecting an element $f \in \calL(2Q)$ such that $f = u_0a/u^2 + \ldots,$ we have}

\[\delta(h) - (\tau(f) - f) \in \calL(Q+(Q\ominus P)).\]
Since $\{1,h\}$ forms a basis of  $\calL(Q+(Q\ominus P))$ (recall that $h$ is not constant), there exist  elements $c, d \in k$ such that
\[\delta(h) - (\tau(f) - f) -ch -d= 0. \]
Therefore
\[\delta^2(h) - c\delta(h) = \tau(\delta(f)) - \delta(f)\]
and  conclusion  (2)~holds for $L = \delta^2 - c\delta$ and $\tilde{g} = \delta(f)$.\end{prop2aproof}


{\mfs \begin{rem} One cannot weaken condition (2) in Proposition~\ref{Prop1}, that is, for a general $b \in k(E)$, condition (1) of Proposition~\ref{Prop1} does {\bf not} imply  the following condition :
\begin{itemize}
\item[(3)] There exist $Q\in E$, $e\in k(E)$ and a constant $c \in k$ such that
\[b = \tau(e) - e+ c.\]
\end{itemize}
To see this, let   $b$ be a nonconstant element of  $\calL(Q+(Q\ominus P))$.  Note that $\pdisp(b) =1$. We  have just shown that $b$ satisfies (1)~of Proposition~\ref{Prop2}.  Now assume $b = \tau(e)-e +c$ for some $e\in k(E), c\in k$. Since $\pdisp(\tau(e) - e) = \pdisp(e) + 1$ if $e \notin k$, we have   $\pdisp(e) =0$. Since $b$ has no poles outside of $\{Q, Q\ominus P\}$, we would have that $e$ has at most one pole and this pole would be simple.  Therefore $e $  must be constant. A contradiction with the fact that $b \notin k$.
\end{rem}}

\subsection{Proof of Proposition~\ref{Prop2}} \label{prop2sec}

\begin{prop3proof} For any $Q\in E$ and $j \in \NX_{>0}$, we have $\ores_{Q,j}(e) = \ores_{Q,j}(\tau(e))$. Furthermore Lemma~\ref{lem5} implies that $\ores_{Q,j}(g) = 0$.  Therefore $\ores_{Q,j}(f) =\ores_{Q,j}(\tau(e)-e+g) =\ores_{Q,j}(\tau(e))-\ores_{Q,j}(e) +\ores_{Q,j}(g) = 0$.\end{prop3proof}

\begin{prop3aproof} The proof of this implication is similar to the proof that (1)~implies (2)~in Proposition~\ref{Prop1}. Lemma~\ref{normform} implies that we may assume that ${\rm pdisp}(f)\leq 1$ and $\rm{wdisp}(f) = 0$.  Therefore if $f$ has a pole of order $j \geq 2$ at some $Q \in E$, then $Q$ is the only point in $Q + \ZX P$ at which $f$ has a pole.  Since $\ores_{Q,j}(f) = 0$, we have that $f$ has no poles of order greater than $1$.  Since we also have  ${\rm pdisp}(f)\leq 1$, we can conclude that that $f$ has only poles of order $1$ and the poles of $f$ occur in pairs $\{Q_1, Q_1\ominus P\}, \ldots, \{Q_r, Q_r\ominus P\}$   where \ju{$(Q_i \oplus \ZX P) \cap (Q_j \oplus \ZX P) = \emptyset$} for $i \neq j$.

We  will now show how one can construct an element $e$ such that $f -(\tau(e) - e)$ has at most  one pair of poles $\{Q, Q\ominus P\}$. This will yield  condition 2.~of the Proposition. We  can assume that $r >1$.  Let $h\in k(E)$ be a nonconstant element of $\calL(Q_1+Q_2)$.  There exists an $a \in k$ such that $\tilde{f} = f - (\tau(ag)-ag)$ has no pole at $Q_1$. The element $\tilde{f}$ has only simple poles and $\mbox{\pdisp}(f) \leq 1$.  Therefore its poles occur at possibly $Q_1\ominus P, \{Q_2, Q_2\ominus P\}, \ldots , \{Q_r, Q_r\ominus P\}$. Since  $\ores_{Q_1,1}(f) = 0$, $f$ cannot have a singe pole in $Q_1+\ZX P$.   Therefore we have that $\tilde{f}$ has no pole at $Q_1\ominus P$.  Continuing in this way we find an $e \in k(E)$ such that $f -(\tau(e) - e)$ has at most one pair of poles $\{Q, Q\ominus P\}$.
\end{prop3aproof}
\section{Some computation of orbit residues}\label{appendixB}
Let $E$ be an elliptic curve defined over an algebraically closed field $k$ of characteristic zero and  $k(E)$ be its function field.  Let $P$ be a non-torsion point on $E$  and let $\tau:k(E)\rightarrow k(E)$ denote map corresponding to $Q\mapsto   Q\oplus P$ on $E$, where $\oplus$ denotes the group law on $E$. Let $\iota_1$ and $\iota_2$ two involutions of $E$ such that $\tau =\iota_2 \circ \iota_1$. We let $\Omega$ be a non zero regular differential form on $E$ and we keep notation as in \S \ref{introsec}.

 \begin{lemma}\label{lem1}
Let $b \in k(E)$ such that $\iota_1(b)=-b$.  Let $Q \in E$ be a simple pole of $b$ {such that $Q \neq \iota_1(Q)$}. Then, $\iota_{1}(Q)$ is a simple pole of $b$ and the residue of $b \Omega$ at $Q$ coincides with its residue at $\iota_1(Q)$.
\end{lemma}

\begin{proof}
The assertion follows from the fact that, since $\iota_1(b)=-b$ and ${\iota_1^*(\Omega)=-\Omega}$ (see \cite[Lemma 2.5.1 and Proposition 2.5.2]{DuistQRT}), the form $\eta=b\Omega$ satisfies ${\eta = \iota_1^*(\eta)}$. Indeed, if $u_Q$ is a local parameter at $Q$, then we have $\eta = v du_Q $ with $v=\frac{c_Q}{u_Q}+\overline{v}$ where $c_Q \in \C$ is the residue of $\eta$ at $Q$ and $\overline{v}$ is regular at $Q$. Hence, $\iota_1(u_Q)$ is a local parameter at $\iota_1(Q)$, and we have $\iota_1^*(\eta) = \iota_1(v) d\iota_1(u_Q)$ with $\iota_1(v)=\frac{c_Q}{\iota_1(u_Q)}+\iota_1(\overline{v})$ where $\iota_1(\overline{v})$ is regular at $\iota_1(Q)$. So, $\Res_{\iota_1(Q)} (\eta) = \Res_{\iota_1(Q)} (\iota_1^*(\eta)) = c_Q = \Res_{Q} (\eta)$.
\end{proof}

For the notion of coherent analytic parameters used below, we refer to \ju{Appendix~\ref{introsec}}, especially to Remark \ref{rmk:analylocparam}. 

\begin{lemma}\label{lem:horesidues}
There exists a coherent set of analytic local parameters ${\{u_{Q} \ \vert \ Q \in E \}}$ on $E^{\mathrm{an}}$ such that $\iota_{1}(u_{Q})=-u_{\iota_{1}(Q)}$. Let $b \in k(E)$ such that $\iota_1(b)=-b$. For such a set of local parameters, if  
\[b  = \frac{c_{Q,n}}{u_Q^n} + \ldots + \frac{c_{Q,2}}{u_Q^2} + \frac{c_{Q,1}}{u_Q} + \tilde{f}\]
where $v_Q(\tilde{f}) \geq 0$, then 
\[
b  =\frac{c_{\iota_{1}(Q),n}}{u_{\iota_{1}(Q)}^n} + \ldots + \frac{c_{\iota_{1}(Q),2}}{u_{\iota_{1}(Q)}^2} + \frac{c_{\iota_{1}(Q),1}}{u_{\iota_{1}(Q)}} + \tilde{g}
\]
where $v_{\iota_{1}(Q)}(\tilde{g}) \geq 0$ and $c_{\iota_{1}(Q),j}=(-1)^{j+1}c_{Q,j}$. 
If follows that, if all the poles of $b$ belong to the same $\tau$-orbit, then, for any even number $j$, we have $\ores_{Q,j}(b)=0$. 
\end{lemma}

\begin{proof}
We first prove the existence of analytic local parameters with the desired properties.  According to \cite[p.35 and Remark 2.3.8]{DuistQRT}, $\iota_{1}(P)=[-1]P\oplus P_{0}$ for some $P_{0} \in E$. By uniformisation, it is equivalent to prove the following result :  {\it Consider a lattice $\Lambda \subset \C$ and two endomorphisms of the complex torus $\C/\Lambda$ given by $\iota_{1} : \overline{z} \mapsto \overline{-z+p_{0}}$ and $\tau : \overline{z} \mapsto \overline{z+q_{0}}$ for some $p_{0},q_{0} \in \C$. Then, there exists a set of analytic local parameters $\{u_{\overline{\omega}} \ \vert \ \overline{\omega} \in \C/\Lambda\}$ on the complex torus $\C/\Lambda$ such that $\iota_{1}(u_{\overline{\omega}})=-u_{\iota_{1}(\overline{\omega})}$ and $\tau(u_{\overline{\omega}})=u_{\tau(\overline{\omega})}$.} Such local parameters are given by $u_{\overline{\omega}} : \overline{z} \mapsto z-\omega$ for $z$ close to $\omega$.  
The rest of the Lemma is a direct consequence of the following easy computation. 
Indeed, applying $\iota_{1}$ to  
\[ b= \frac{c_{Q,n}}{u_Q^n} + \ldots + \frac{c_{Q,2}}{u_Q^2} + \frac{c_{Q,1}}{u_Q} + \tilde{f},\]
we get  
\begin{multline*}
-b=\iota_{1}(b) = \frac{c_{Q,n}}{\iota_{1}(u_Q)^n} + \ldots + \frac{c_{Q,2}}{\iota_{1}(u_Q)^2} + \frac{c_{Q,1}}{\iota_{1}(u_Q)} + \iota_{1}(\tilde{f})\\
= \frac{(-1)^{n}c_{Q,n}}{u_{\iota_{1}(Q)}^n} + \ldots + \frac{(-1)^{2}c_{Q,2}}{u_{\iota_{1}(Q)}^2} + \frac{(-1)^{1} c_{Q,1}}{u_{\iota_{1}(Q)}} + \iota_{1}(\tilde{f})
\end{multline*}
where $v_{\iota_{1}(Q)}(\iota_{1}(\tilde{f})) \geq 0$, as expected. 
\end{proof}

\begin{lem}\label{lem:res0order2}
If $g \in \calL(2Q + 2(Q\ominus m_1 P) + \ldots + 2(Q\ominus m_s P) )$ where $m_1, \ldots, m_s \in \ZX\backslash \{0\}$ is such that $\ores_{Q,2}(g)=0$ then 
\[\ores_{Q,1}(g) = 0.\]
\end{lem}
\begin{proof}
We may write $\delta(u_Q)^{-1} = \alpha + \beta u_{Q} + \bar u$ where $0 \neq \alpha \in k$, $\beta \in k$ and $\bar{u}$ is regular and has a zero of order $2$ at $Q$.  For each $i \in \{0, m_1, \dots,m_s \}$ we write
\[g = \frac{c_{Q\ominus iP, 2}}{u_{Q\ominus iP}^{2}} + \frac{c_{Q\ominus iP, 1}}{u_{Q\ominus iP}} + g_{Q\ominus iP}\]
where $g_{Q\ominus iP}$ is regular and zero at $Q\ominus iP$. Now consider the differential $g\Omega$.  Since for any $i \in \ZX$,  $\omega =\delta(u_{Q\ominus iP})^{-1}du_{Q\ominus iP}$, we have
\begin{eqnarray*}
g\Omega & = & (\frac{c_{Q\ominus iP, 2}}{u_{Q\ominus iP}^{2}} +\frac{c_{Q\ominus iP, 1}}{u_{Q\ominus iP}} + g_{Q\ominus iP})(\delta(u_{Q\ominus iP})^{-1}du_{Q\ominus iP})\\
 & = &  (\frac{c_{Q\ominus iP, 2}}{u_{Q\ominus iP}^{2}} +\frac{c_{Q\ominus iP, 1}}{u_{Q\ominus iP}} + g_{Q\ominus iP})(\tau^i(\delta(u_Q)^{-1})du_{Q\ominus iP})\\
 & = &  (\frac{c_{Q\ominus iP, 2}}{u_{Q\ominus iP}^{2}} +\frac{c_{Q\ominus iP, 1}}{u_{Q\ominus iP}} + g_{Q\ominus iP})((\alpha +  \beta u_{Q \ominus iP} + \tau^i(\bar{u}))du_{Q\ominus iP})
\end{eqnarray*}
where the second equality follows from the fact that $u_{Q\ominus iP} = \tau^i(u_Q)$ and $\tau\delta =\delta \tau$. Note that $\tau^{i}(\bar{u})$ is regular and has a zero of order $2$ at $Q\ominus iP$. Therefore the residue of $g\omega$ at $Q\ominus iP$ is $\alpha c_{Q\ominus iP, 1} + \beta c_{Q\ominus iP, 2}$. Since the sum of the residues of a differential form is $0$ we get $\alpha \ores_{Q,1}(g) + \beta \ores_{Q,2}(g)= 0$. Since $\alpha \neq 0$ and $\ores_{Q,2}(g)=0$, we get $\ores_{Q,1}(g)=0$. 
\end{proof}

\bibliography{walkbib}

\end{document}